\documentclass[leqno,12pt]{amsart} 
\title[{Graded  contractions of $\mathfrak g_2$}]{Graded  contractions of \\
the   $\bb Z_2^3$-grading on $\mathfrak g_2$}
\author{Cristina Draper, Thomas Leenen Meyer, Juana S\'anchez-Ortega}
\usepackage{amssymb,color, bm}


\thanks{${}^*$ The first author is supported by Junta de Andaluc\'{\i}a  through projects  FQM-336 and UMA18-FEDERJA-119,   and  by the Spanish Ministerio de Ciencia e Innovaci\'on   through projects  PID2019-104236GB-I00 and PID2020-118452GB-I00, all of them with FEDER funds.
The  second author   has received a University of Cape Town Science Faculty PhD Fellowship and the Harry Crossley Research Fellowship.
The third author is supported by URC fund 459269.} 

\usepackage{amssymb, mathrsfs, amsmath, tabu, amsthm}
\usepackage{hyperref}
\usepackage{tikz-cd, graphicx} 
\usepackage{pdfpages}
\usepackage{enumitem}
\usepackage{ragged2e}
\usetikzlibrary{decorations.markings}

\hypersetup{
	colorlinks = true,
	urlcolor = blue,
	linkcolor = blue,
	citecolor = red
}
\newcommand{\gla}{\mathfrak{gl}} 
\newcommand{\la}{\mathfrak{L}} 
\newcommand{\setbar}[2]{\{#1 \, \lvert \, #2\}} 
\newcommand{\bfemph}[1]{\emph{\textbf{#1}}}

\newcommand{\f}[1]{\mathfrak{#1}}

\newcommand{\bb}[1]{\mathbb{#1}}

\newcommand{\ep}{\varepsilon}

\newcommand{\lae}{\la^{\ep}}
\newcommand{\weyl}{\mathcal{W}}

\newcommand{\comment}[1]{}

\DeclareMathOperator{\Supp}{Supp}
\DeclareMathOperator{\Der}{Der}
\DeclareMathOperator{\Aut}{Aut}
\DeclareMathOperator{\Stab}{Stab}
\DeclareMathOperator{\Diag}{Diag}
\DeclareMathOperator{\ad}{ad}
\DeclareMathOperator{\id}{id}
\DeclareMathOperator{\Spec}{Spec}

\theoremstyle{plain} 
\newtheorem{theorem}{Theorem}[section]
\newtheorem{lemma}[theorem]{Lemma}
\newtheorem{cor}[theorem]{Corollary}
\newtheorem{prop}[theorem]{Proposition}

\theoremstyle{definition} 
\newtheorem{define}[theorem]{Definition}
\newtheorem{defis}[theorem]{Definitions}
\newtheorem{remark}[theorem]{Remark}
\newtheorem{example}[theorem]{Example}
\newtheorem{notation}[theorem]{Notation}

\newtheorem{conclusion}[theorem]{Conclusion}

\begin{document}


\maketitle

\vspace{-20pt}\begin{abstract}
Graded contractions of the   $\bb{Z}_2^3$-grading on the complex exceptional Lie algebra $\f{g}_2$ are   
classified up to equivalence and up to strong equivalence. 
 The non-toral fine $\bb{Z}_2^3$-grading is highly symmetric, with all the homogeneous components Cartan subalgebras.
 This  makes possible a combinatorial treatment based on certain \emph{nice} subsets of the set of 21 edges of the Fano plane. 
 There are 24 such nice sets up to collineation. Each of these is the support of an \emph{admissible} graded contraction, one of which is present in every equivalence class of graded contractions.
Each nice set gives rise to a single Lie algebra, except for three of the cases in which families depending on one or two parameters are found. 
In particular, a large family of 14-dimensional Lie algebras arise, most of which are solvable. The properties of each of these Lie algebras are studied.
\end{abstract}
\vspace{-10pt}\tableofcontents
	
 \section{Introduction}  
 
 The notion of a Lie group   contraction (i.e. Lie algebra contraction) comes from Physics, being   introduced by Segal in 1951 \cite{Segal} and by  In\"on\"u and Wigner in 1953 \cite{IW53}.
 Segal considers a sequence of groups whose structure constants   converge to  the structure constants of a non-isomorphic group. 
 As explained in  \cite{IW53}, the fact that classical mechanics is a limiting case of relativistic mechanics means that the Galilei group must   in some sense be a limiting case of the relativistic mechanics group. Similarly the Lorentz group must be in some sense a limiting case of the de Sitter group. 
 So the general purpose is to try to formalize what can be the \emph{limit} or contraction of groups. 
 In Fialowski  and de Montigny's words \cite{Fialo} it was the need to relate the symmetries underlying Einstein's mechanics and
Newtonian mechanics which motivated   the concept of a contraction. This consists of multiplying the generators of the symmetry by certain  ``contraction parameters'', such that when
these parameters reach some singularity point, one obtains a non-isomorphic Lie algebra of the same dimension. More precisely, take $L$ either a real or complex Lie algebra and $U \colon(0,1] \to \mathrm{GL}(L)$ a continuous map. For $\varepsilon \in (0, 1]$, define $[x, y]_\varepsilon := {U_\varepsilon}^{-1}([U_\varepsilon(x), U_\varepsilon(y)])$, for all $x, y \in L$. Then $L_\varepsilon = \big(L, [\cdot, \cdot]_\varepsilon \big)$ is a Lie algebra isomorphic to $L$. Write $[x, y]_0$ to denote the limit of $[x, y]_\varepsilon$ when $\varepsilon \to 0^+$; if $[x, y]_0$ exists for all $x, y\in L$, then $L_0 = \big(L,[\cdot, \cdot]_0\big)$ becomes a Lie algebra, which is called a one-parametric continuous contraction of $L$, or simply, a \emph{contraction} of $L$.

Degenerations, contractions and deformations of many algebraic structures have turned out to be very useful in both Mathematics and Physics. Looking at the literature, it seems that physicists have been more interested in degenerations and contractions, while most of the results about deformations can be found in Mathematics journals.
These three notions have been introduced in many different ways, depending on the approach taken;
here we present a very brief introduction within the framework of Lie algebras. In \cite{Fialo} deformations and contractions procedures are used to construct and classify new Lie algebras; see also \cite{Burde07} for a nice exposition for Lie algebras and algebraic groups, and \cite[Chapter~7]{enci} for a more complete review for Lie algebras and Lie groups.
 
A contraction can be viewed as a special case of \emph{degeneration}. Let $V$ be an $n$-dimensional vector space over a field $F$ and $\mathcal{L}_n(F)$ the variety of Lie algebra laws, that is, the alternating bilinear maps $\mu\colon V\times V\to V$ such that the pair $(V, \mu)$ is a Lie algebra. The general linear group $\mathrm{Gl}(V)$ acts on $\mathcal{L}_n(F)$ by $(g\cdot \mu)(x, y) = g\big(\mu(g^{-1}x, g^{-1}y)\big)$, for $g \in \mathrm{Gl}(V)$ and $x, y \in V$. If $O(\mu)$ denotes the orbit of $\mu\in \mathcal{L}_n(F)$ under the previous action and $\overline{O(\mu)}$ the closure of the orbit with respect to the Zariski topology, then $\lambda\in \mathcal{L}_n(F)$  degenerates to $\mu \in \mathcal{L}_n(F)$ if $\mu \in \overline{O(\lambda)}$.

Contractions and \emph{deformations} are opposite procedures. 
Roughly speaking, both contractions and deformations of Lie algebras are modifications of their structure constants; but contractions transform a Lie
algebra into a  ``more Abelian'' Lie algebra, while deformations produce a Lie algebra with a more complicated Lie bracket. 
A one-parameter deformation of a
Lie algebra $L=(V,\mu)$ as before is a continuous curve over $\mathcal{L}_n(F)$.
 A formal one-parameter deformation is defined by the Lie brackets
$[a, b]_t = F_0(a, b) + tF_1(a, b) + \ldots + t^m F_m (a, b) + \ldots$
where $F_0$ denotes the original Lie bracket. 
 
The type of contractions we will be working with, known as \emph{graded contractions}, appeared first in the early 90s in some Physics journals, as a generalization of the Wigner-In\"on\"u contractions. Since their introduction, graded contractions have been investigated in other algebraic structures; see, for instance, \cite{ref19} for affine algebras, \cite{ref15} for Jordan algebras and \cite{ref17} for Virasoro algebras. The idea behind graded contractions consists essentially in preserving Lie gradings through the contraction process. Their name might confuse the reader at first since a graded contraction is not a contraction that is graded. As we will see, graded contractions are defined algebraically and not by a limiting process.  

The first work  on the subject of graded contractions of Lie algebras is \cite{otro91}. In it, de Montigny and Patera  studied   the grading-preserving contractions of complex Lie algebras and superalgebras of any type; appearing a new type of discrete contractions 
which are not Wigner-In\"on\"u-like  continuous contractions.
The first examples including other gradings other than $\bb Z_2$ appear in \cite{CPSW}, 
which applies
the new defined formalism to the toral $\bb Z^2$-grading  of the classical simple Lie algebra $\mathfrak{sl}(3,\bb C)$ of dimension 8, that is,   the root decomposition.
 This Lie algebra of $A_2$-type admits exactly 3 fine non-toral gradings, with universal groups $\bb Z_3^2$, $\bb Z_2^3$ and $\bb Z\times\bb Z_2$.
In 2004, a group of Physicists studied in \cite{checos} the Pauli grading of $\mathfrak{sl}(3,\bb C)$ over $\bb Z_3^2$. 
The Pauli grading has the advantage that all the non-zero homogeneous components are 1-dimensional, so 
the Jacobi identity associated to a graded contraction produces a system of quadratic equations for
the contraction parameters, which can be reduced using the group of symmetries of the grading to find a non-trivial solution. Two years later, in 2006, the Pauli grading of $\mathfrak{sl}(3,\bb C)$ over $\bb Z_3^2$ is investigated further in \cite{checos06}, and this time the system of 48 contraction equations involving 24 contraction parameters is completely solved.  A list of all the equivalence classes of solutions is provided (there are 188 inequivalent solutions, 13 of them depending or one or two continuous parameters); the algorithms developed in \cite{Rand} were used to identify a Lie algebra starting from its structure constants.
Graded contractions of the Gell-Mann grading of $\mathfrak{sl}(3,\bb C)$ over the group $\bb Z_2^3$ were studied by Hrivn\'ak and Novotn\'y \cite{gr-cont} in 2013. The difficulty now is that this grading has one homogeneous component of dimension 2; the authors restricted their study to the graded contractions preserving the 2-dimensional homogeneous component. Using the group of symmetries of the Gell-Mann grading, the system of contraction equations is reduced and solved, and 53 types of Lie algebras were found.   
Despite the fact that the results in \cite{checos06, gr-cont} are a contribution to the classification problem of solvable Lie algebras, 
the mathematical community seems not to have continued this line of work.
This might be due to the   long  calculations that require the use of computer systems, which suggests, perhaps,  that a different approach to investigate graded contractions  may revive interest in the topic.
The works   containing the most general results on the subject of graded contractions are \cite{ref33, ref34}; in \cite{ref33} the contraction matrix is studied provided that the abelian grading group has finite order, introducing invariants as
pseudobasis, higher-order identities, and sign invariants;
while \cite{ref34} focusses on how the distribution of zero entries  in the contraction matrix affects the structure of the associated contracted algebra.

Our manuscript is inspired in all the above works on concrete simple Lie algebras, but 
tries to go a few steps further
in some aspects. We 
focus our attention on the    beautiful exceptional Lie algebra of type  $\f{g}_2$,  a step further not only for the dimension, 14, but also for the least size of a matrix algebra containing $\f{g}_2$, 7.
 It is worth mentioning here that this is the first time graded contractions on an exceptional Lie algebra are investigated. 
Gradings on the octonion algebra $\mathcal{O}$ were studied by Elduque in \cite{gradsO}; his paper inspired other authors \cite{GradsG2, gradsg22} to investigate gradings on the Lie algebra of derivations of $\mathcal{O}$, that is,  the Lie algebra $\f{g}_2$.
It turns out that any grading on $\f{g}_2$ is induced by a grading on $\mathcal{O}$, although there is not a one-to-one correspondence between non-equivalent gradings on $\mathcal{O}$ and on $\f{g}_2$. Amongst the 25 non-equivalent gradings on $\f{g}_2$, there is only one non-toral, which is precisely the 
$\bb Z_2^3$-grading we are interested in (defined in Equation~\eqref{graduaciong2}). 
The challenge of trying to classify its graded contractions 
 is that its nonzero homogeneous components have all of them dimension 2,   
and that the grading is non-toral  (i.e., the homogeneous spaces are not sums of root spaces).   
An additional motivation for choosing this grading is its high symmetry, since the greater the symmetry, the easier it is to classify the graded contractions; and some particularities such as, for example, the fact that all the pieces are Cartan subalgebras, which makes the choice of a basis of $\f{g}_2$ formed by semisimple elements trivial.  

Our calculations and techniques are innovative,   in the sense that we do not require the use of a computer system.
Our main goal is   to introduce new tools and procedures to classify all graded-contractions of the $\bb Z_2^3$-grading on $\f{g}_2$ up to isomorphism of the related contracted graded algebras. These tools permit, in particular,   to investigate Weimar-Woods's Conjecture \cite[Conjecture~2.15]{ref33}, which turns out to be true for our grading, but false in general (see Example~\ref{ex_encontreunoenC}).  

Our second goal consists of  exploring  a new family of 14-dimensional solvable and nilpotent Lie algebras. Although complex semisimple Lie algebras were classified by Killing at the end of the XIX century, the situation with solvable algebras is completely different and we are still far from obtaining a full classification. This classification problem has attracted the attention of several researchers, and some solvable Lie algebras of small dimensions have been classified; for instance, see \cite{Graaf} for the classification of solvable Lie algebras up to dimension 4 over an arbitrary field, and \cite{book} for a classification up to dimension 6. Some other classifications have been obtained by adding extra hypotheses onto the solvable Lie algebras; see, for example, \cite{citadoreciente} for  a classification of solvable Lie algebras with an $n$-dimensional (nilpotent of nilindex $n-3$) nilradical, and \cite{Lu} for a classification of finite-dimensional complex solvable Lie
algebras with nondegenerate symmetric invariant bilinear form. See \cite{review} and the references therein for a nice review of the classifications of Lie algebras. \smallskip

The paper is structured as follows: in \S\ref{sec2} we gather together the required background on graded contractions of Lie algebras and their
 equivalence relations, we introduce the $\bb Z_2^3$-grading $\Gamma_{\f{g}_2}$ on $\f{g}_2$ that we will be working with, and provide some results on this grading (Lemmas~\ref{lem_base} and \ref{lem_jacobi_li}), which will allow in \S\ref{sec3} to restrict our attention to graded contractions adapted in some way to $\Gamma_{\f{g}_2}$, the \emph{admissible} graded contractions. Thus, 
 \S\ref{sec3} deals with graded contractions of $\Gamma_{\f{g}_2}$; 
considering  admissible graded contraction  allows us to view our classification problem from a combinatorial point of 
view in the projective plane $P^2(\bb F_2)=PG(2,2) $.
Besides, collineations of this plane read the action of the Weyl group of the grading
 (Proposition~\ref{colli}).
The main tool is to consider certain sets, called {\it nice sets}, that are defined via a certain absorbing property (see Definition~\ref{def_nice}),
which 
are the supports of the admissible graded contractions. Important  examples of nice sets are presented in Definition~\ref{niceSubsets} and Proposition~\ref{candidatos}, and a full classification (up to collineations) is achieved in Theorem~\ref{classiN}.

\S\ref{sec4} is devoted to the classification of admissible graded contractions up to normalization. Using our   combinatorial approach, we are able to avoid the use of the computer,  as mentioned earlier. We found (see Theorem~\ref{prop_equiv})
21 non-isomorphic (up to normalization) graded algebras obtained by graded contraction of $\f{g}_2$,
jointly with three families parametrized by $\bb{C}^\times$, $\bb{C}^\times/\bb{Z}_2$ and $(\bb{C}^\times)^2/\bb{Z}^2_2$. 
As we will see in \S\ref{sec5}, 
strongly equivalent graded contractions and equivalent graded contractions up to normalization turn out to be the same thing for our grading (Theorem~\ref{teo_normalizacionbastaba}). 
Then \S\ref{sec_revision} addresses the relationship between equivalence and strong equivalence, studying in Proposition~\ref{nueva}
when two admissible equivalent graded contractions $\eta$ and $\eta'$ admit a collineation $\sigma$ such that $\eta^\sigma\approx\eta'$. 
This happens frequently, but not always.
The result allows us to obtain the equivalence classes in Theorem~\ref{th_lasclasesdeverdad}, namely: 20 classes jointly with 3 parametrized families.
Thus, our main classification results are  Theorem~\ref{prop_equiv}, Theorem~\ref{teo_normalizacionbastaba} and Theorem~\ref{th_lasclasesdeverdad}.

We finish by investigating in \S\ref{sec6} the properties that these new Lie algebras satisfy. Besides the trivial cases (a simple and an abelian Lie algebra), 
an algebra which is the sum of a semisimple Lie algebra and its center, 
and another one not reductive,  we obtain: 12 nilpotent Lie algebras (11 of them with nilindex 2, and the other one with nilindex 3), 4 solvable (not nilpotent) Lie algebras (3 of them with solvability index 2, and the other one with solvability index 3) and 3 infinite parametrized families of 2-step solvable but not nilpotent Lie algebras (see Theorem~\ref{teo_lasalgebras}  for more details).


\section{Preliminaries}  \label{sec2}

In order to make this paper self-contained and suitable for a broader audience, we recall here the necessary background and introduce some notation.

\subsection{The Lie algebra $\mathfrak{g}_2$ and its fine $\mathbb Z_2^3$-grading}   \label{sec21}


The \bfemph{complex octonion algebra} $\mathcal{O}$, also known as the bioctonions, is the only eight-dimensional unital composition algebra over $\bb{C}$. The set $\big \{1, \, {\bf i}, \, {\bf j}, \, {\bf k}, \, {\bf l}, \, {\bf il}, \, {\bf jl}, \, {\bf kl} \big \}$ constitutes a basis for $\mathcal{O}$ with product given by the figure below, where $1$ denotes the identity element,  and all the squares of basic elements equal $-1$.

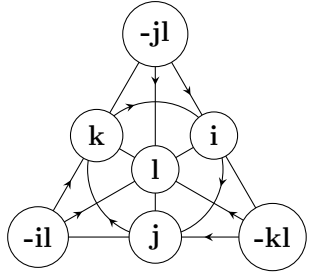
\begin{center}
\begin{tikzpicture}[scale=0.9, every node/.style={transform shape}]  \label{fano}
\draw [postaction={decoration={markings, mark= at position 0.75 with {\arrowreversed{stealth}}}, decorate}] 
(30:1) -- (210:2);
\draw [postaction={decoration={markings, mark= at position 0.75 with {\arrowreversed{stealth}}}, decorate}]
(150:1) -- (330:2);
\draw [postaction={decoration={markings, mark= at position 0.75 with {\arrowreversed{stealth}}}, decorate}]
(270:1) -- (90:2);
\draw [postaction={decoration={markings, mark= at position .24 with {\arrowreversed{stealth}}}, decoration={mark= at position .57 with {\arrowreversed{stealth}}},decoration={mark= at position .9 with {\arrowreversed{stealth}}}, decorate}]
(90:2)  -- (210:2) -- (330:2) -- cycle;
\draw [postaction={decoration={markings, mark= at position .3 with {\arrowreversed{stealth}}}, decoration={mark= at position .63 with {\arrowreversed{stealth}}},decoration={mark= at position .96 with {\arrowreversed{stealth}}}, decorate}] (0:0)  circle (1);
\draw 
(30:1) node[circle, draw, fill=white]{{\bf i}}
(210:2) node[circle, draw, fill=white]{{\bf -il}}
(150:1) node[circle, draw, fill=white]{{\bf k}}
(330:2) node[circle, draw, fill=white]{{\bf -kl}}
(270:1) node[circle, draw, fill=white]{{\bf j}}
(90:2) node[circle, draw, fill=white]{{\bf -jl}}
(0:0) node[circle, draw, fill=white]{{\bf l}};
\end{tikzpicture}
\end{center}
 The treatment on the octonions based on the Fano plane can be consulted, for instance, in the book \cite{Okubo}, which deals with both real and complex octonion algebra and its use in the field of mathematical physics.
The map $\bar x := -x$, for any $x \ne 1$ in the basis and $\bar 1:= 1$, is an involution on $\mathcal{O}$ satisfying that $n(x) := x \bar x \in \bb C$ and $t(x) := x + \bar x \in \bb C$. (We identify here $\bb C$ and $\bb C1$.) The map $t: \mathcal{O} \to \bb C$ is linear, while the map  $n\colon \mathcal{O}\to \bb C$ is a norm {admitting composition}, that is, $n(xy) = n(x)n(y)$, for all $x, y \in \mathcal{O}$. Moreover, $x^2 - t(x)x + n(x)1 = 0$, for all $x \in \mathcal{O}$, which makes $\mathcal{O}$ a quadratic algebra.
		
The Lie algebra of derivations of $\mathcal{O}$ is the complex exceptional simple Lie algebra of dimension 14 of type $G_2$; 
we denote as 
\[
\f{g}_2 = \Der(\mathcal{O}) = \left \{d\in\mathfrak{gl}(\mathcal{O})\mid  d(xy)=d(x)y+xd(y), \ \forall \, x,y\in \mathcal{O}\right \}.
\] 
The derivations of the alternative algebra $\mathcal{O}$ are well known (see, for instance, \cite[Chapter~III, \S8]{Schafer}).
Namely, the {left} $L_x$ and {right} $R_x$  {multiplication operators} on $\mathcal{O}$, $L_x(y)=xy$ and $R_xy=yx$, give rise to a concrete derivation $D_{x, y}$: 
\[
D_{x, y} := [L_x, \, L_y] + [L_x, \, R_y] + [R_x,R_y]\in \Der(\mathcal{O}) .
\]
Notice that $D_{x, y}(z) = [[x, y], z] - 3(x, y, z)$  for all $z\in \mathcal{O}$, where $[\cdot, \cdot]$ and $(\cdot, \cdot, \cdot)$ refer to the commutator and associator of $\mathcal{O}$, respectively. The skew-symmetric bilinear map $D\colon   \mathcal{O} \times \mathcal{O} \to \Der(\mathcal{O})$ given by $(x, y) \mapsto D_{x, y}$ is   $\Der( \mathcal{O})$-invariant: $[d, D_{x, y}] = D_{d(x),y} + D_{x, d(y)}$, for all $d \in \Der(\mathcal{O})$, and, although it is not surjective, the derivation algebra is spanned by its image,
\[
\f{g}_2 = \Der(\mathcal{O}) = \left \{\sum_i D_{x_i, y_i}\mid  \,  x_i, \, y_i \in \mathcal{O}, \, i \in \bb{N} \right \}.
\] 
Moreover,  for all $x, y, z \in \mathcal{O}$,
\begin{equation}\label{identidadciclica}
D_{x, yz} + D_{y, zx} + D_{z, xy} = 0.
\end{equation}

In this paper we are interested in studying the  fine $\bb{Z}_2^3$-grading of the Lie algebra $\f{g}_2$.
Many details on this grading can be found in the AMS-monograph \cite{EK13}, devoted to gradings on simple Lie algebras over algebraically closed  fields  without restrictions on the characteristic; which also encloses the description of gradings on simple associative algebras and on the octonion algebra $\mathcal{O}$. 

First, we review some general facts about gradings. Let $G$ be   an abelian group. A \bfemph{$G$-grading} $\Gamma$ on an arbitrary algebra $A$ is a vector space decomposition $\Gamma:A = \oplus_{g\in G} A_g$ satisfying that $A_gA_h \subseteq A_{g + h}$, for all $g, \, h$ in $G$. The \bfemph{support} of $\Gamma$ is the set given by $\Supp(\Gamma) := \{g \in G: \, A_g \ne 0 \}$; it is usually assumed that $\Supp (\Gamma)$ generates the whole grading group $G$.

Gradings on $A$ and on its algebra of derivations $\Der(A)$ (which is always a Lie subalgebra of the general linear algebra $\mathfrak{gl}(A)$) are closely related. In fact, any $G$-grading $\Gamma\colon A = \bigoplus_{g\in G} A_g$ on $A$ induces a $G$-grading on $\gla(A)$ and on $\Der(A)$, respectively, as: 
\begin{align}
\gla(A) &= \bigoplus\limits_{g\in G}\gla(A)_g, \mbox{ where } \gla(A)_g: = \{f \in \gla(A): f(A_h)\subseteq A_{g+h},\ \forall \, h\in G\}, \notag
\\
\Der(A) &= \bigoplus\limits_{g\in G} \Der(A)_g, \mbox{ where} \Der(A)_g:=\Der(A)\cap \gla(A)_g. \label{ex_grading_o}
\end{align} 
Let $\Gamma$ be a $G$-grading on $A$; 
the \bfemph{automorphism group} $\Aut(\Gamma)$ of $\Gamma$ consists of all self-equivalences: $\, \, \Aut(\Gamma) = \{f \in \Aut(A) \colon \mbox{for any } g \in G \mbox{ there exists } g' \mbox{ with  } f(A_g)\subseteq A_{g'}\}$; 
the \bfemph{stabilizer} $\Stab(\Gamma)$ of $\Gamma$ consists of all graded automorphisms of $A$, that is,  $\Stab(\Gamma)=\{f\in \Aut(A)\colon f(A_g)\subseteq A_g,\ \forall \, g\in A\}$; 
the \bfemph{diagonal group} $\Diag(\Gamma)$ of $\Gamma$ consists of all automorphisms of $A$ such that each $A_g$ is contained in some eigenspace, that is,
$\Diag(\Gamma) = \{f\in \Aut(A): \mbox{for all } g \in G \textrm{ there exists } \alpha_g \in\bb C \mbox{ such that }  f\vert_{A_g} = \alpha_g\mathrm{id}_{A_g}\}$. 
The \bfemph{Weyl group} of $\Gamma$ is the factor group $\weyl(\Gamma)=\Aut(\Gamma)/\Stab(\Gamma)$, 
which is a subgroup of $\mathrm{Sym}\big(\mathrm{Supp}(\Gamma)\big)$. \smallskip

Second, we focus on describing the mentioned grading $\Gamma_{\f{g}_2}$, which will be the main character of this paper.
 Let $G = \bb{Z}_2^3$. The decomposition 
$\Gamma_{\mathcal{O}}:\mathcal{O}=\oplus_{g\in \bb{Z}_2^3}\mathcal{O}_g$ given by 
\begin{align*}
\mathcal{O}_{(1,0,0)} &= \bb{C}{\bf i}, && \mathcal{O}_{(0,1,0)} = \bb{C}{\bf j}, \notag 
\\
\mathcal{O}_{(0,0,1)} &= \bb{C}{\bf l}, && \mathcal{O}_{(1,1,0)} = \bb{C}{\bf k},  
\\
\mathcal{O}_{(1,0,1)} &= \bb{C}{\bf il}, && \mathcal{O}_{(0,1,1)} = \bb{C}{\bf jl}, \notag
\\
\mathcal{O}_{(1,1,1)} &= \bb{C}{\bf kl}, && \mathcal{O}_{(0,0,0)} = \bb{C}1, \notag
\end{align*}
is a $G$-grading on $\mathcal{O}$. It is said to be a division grading, since every homogeneous element is invertible.
We write $\Gamma_{\f{g}_2}$ to denote the $G$-grading on $\f{g}_2$ obtained from $\Gamma_{\mathcal{O}}$ via  \eqref{ex_grading_o}. 
For any $x \in \mathcal{O}_k$ and $y \in \mathcal{O}_h$, the multiplication operators $L_x$ and  $R_y$   are homogeneous maps, which imply that $D_{x, y}\in (\f{g}_2)_{k + h}$, and so the homogeneous components of the grading become
\begin{equation}\label{graduaciong2}
\begin{array}{ll}
\Gamma_{\f{g}_2}: &\f{g}_2= \oplus_{g\in \bb{Z}_2^3} (\f{g}_2)_g,\vspace{2pt}\\
&(\f{g}_2)_{g} = \sum_{k+h=g} D_{\mathcal{O}_{k }, \mathcal{O}_{h }}.
\end{array}
\end{equation}
In particular, ${(\f{g}_2)}_e = \sum_{k\in G}D_{\mathcal{O}_{k}, \mathcal{O}_{k}} = 0$, since $\mathcal{O}_{k}$ is 1-dimensional and 
$D_{x, x} = 0$ for all $x \in \mathcal{O}$.

The Weyl groups of these gradings $\Gamma_{\f{g}_2}$ and $\Gamma_{\mathcal{O}}$ are known to be isomorphic to the whole set of automorphisms of the group $G=\bb Z_2^3$,  which is relatively big, with $7\cdot6\cdot 4=168$ elements.
Indeed,  note that any $f\in\Aut(\mathcal{O})$ induces   $\tilde f\in \Aut(\f{g}_2)$ defined by $\tilde f(d)=  fdf^{-1}$, for all $d\in \f{g}_2$. This automorphism satisfies $\tilde f(D_{x, y}) = D_{f(x), f(y)}$, for all $x, y \in \mathcal{O}$. In particular, if $f\in\Aut(\Gamma_{\mathcal{O}})$,  there is a group homomorphism $\alpha\colon G\to G$ such that $  f(\mathcal{O}_g)=\mathcal{O}_{\alpha(g)}$ and the related map $\tilde f\in \Aut(\Gamma_{\f{g}_2})$ also satisfies $\tilde f((\f{g}_2)_g)=(\f{g}_2)_{\alpha(g)}$. Now, given 
 any triplet of elements $g, h, k$ generating $G$, there exists an algebra automorphism $f\colon  \mathcal{O} \to \mathcal{O}$ such that
\begin{equation} \label{eq_Cayleytriplets}
f(\mathcal{O}_g) = \bb{C}{\bf i}, \qquad
f(\mathcal{O}_h) = \bb{C}{\bf j}, \qquad
f(\mathcal{O}_k) = \bb{C}{\bf l};
\end{equation}
since any pair of Cayley triples are connected by an automorphism of $\mathcal{O}$ (see, for instance,  \cite[Remark~5.13]{NotesG2}).
The related map $\tilde f$ satisfies $\tilde f((\f{g}_2)_{g}) = (\f{g}_2)_{(1, 0, 0)}$, $\tilde f((\f{g}_2)_{h}) = (\f{g}_2)_{(0, 1, 0)}$ and $\tilde f((\f{g}_2)_{k})= (\f{g}_2)_{(0, 0, 1)}$, which implies
\begin{equation} \label{eq_weyl}
\weyl(\Gamma_{\f{g}_2})\cong \Aut(\bb Z_2^3).  
\end{equation}
\smallskip

Many facts of the grading $\Gamma_{\f{g}_2}$ were investigated in \cite{GradsG2, EK13}. In the next two results, we collect   the main properties needed for our purposes. 

\begin{lemma} \label{lem_base}
The following assertions hold for any distinct $g_1$ and $g_2$ in $\bb Z_2^3 \setminus \{e\}$.
\begin{enumerate}
\item[\rm (i)] The subalgebra generated by $\mathcal{O}_{g_1}$ and $ \mathcal{O}_{g_2}$ is a quaternion subalgebra.
\item[\rm (ii)] The homogeneous component ${(\f{g}_2)}_{g_1}$ is a 2-dimensional Cartan subalgebra. In particular any homogeneous basis is formed entirely by semisimple elements.
\item[\rm (iii)] The subalgebra $(\f{g}_2)_{g_1} \oplus (\f{g}_2)_{g_2} \oplus (\f{g}_2)_{g_3}$, for $g_3 := g_1 + g_2$, is a semisimple Lie algebra isomorphic to $\f{sl}(2, \mathbb C) \oplus \f{sl}(2, \mathbb C)$.
Moreover, for $k \in \{1, 2, 3\}$ there exists a basis $B_{g_k} = \{x_k, y_k\}$ of $ (\f{g}_2)_{g_k}$ such that 
\begin{equation}\label{eq_para5}
[x_k, x_{k+1}] = x_{k +2}, \quad
[y_k, y_{k+1}] = y_{k + 2}, \quad
[x_k, y_{k'}] = 0,
\end{equation}
for any $k' \in \{1, 2, 3\}$  and the sum of subscripts  taken modulo 3. In particular:
\[
[(\f{g}_2)_{g_1}, (\f{g}_2)_{g_2}] = (\f{g}_2)_{g_1+g_2}.
\]
\end{enumerate}
\end{lemma}

\begin{proof}
It is enough to prove the result for $g_1 := (1, 0, 0)$, $g_2 := (0, 1, 0)$ by Equation~\eqref{eq_weyl}.

\noindent (i) is clear, since the subalgebra generated by $\mathcal{O}_{g_1} = \bb{C}{\bf i}$ and $ \mathcal{O}_{g_2} = \bb{C}{\bf j}$ is just 
 the algebra of the complex quaternions $\mathcal{H}=\langle 1, {\bf i},{\bf j},{\bf k}\rangle $.
\smallskip 

\noindent (ii) follows from (iii) since $(\f{g}_2)_{g_1} =\bb{C} x_1\oplus \bb{C} y_1  $ is abelian. 

\smallskip

\noindent (iii) For $k \in \{1, 2, 3\}$, consider the derivation $x_k \in \f{g}_2$ of $\mathcal{O}$ mapping $\mathcal{H}$ onto $0$ and 
\[
x_1(q {\bf l}) = \frac12({\bf i}q){\bf l}, \qquad
x_2(q {\bf l}) = \frac12({\bf j}q){\bf l}, \qquad
x_3(q {\bf l}) = \frac12({\bf k}q){\bf l},
\]
for any $q \in \mathcal{H}$; and the following elements of $\f{g}_2$:  
\[
y_1 = \frac14 D_{\bf j,\bf k}, \qquad
y_2 = \frac14 D_{\bf k,\bf i}, \qquad
y_3 = \frac14 D_{\bf i,\bf j}.
\]
Checking that Equation~\eqref{eq_para5} holds is a straightforward calculation and we leave it to the reader. On the other hand, it is not difficult
to see  that $x_k, \, y_k \in(\f{g}_2)_{g_k}$ are linearly independent derivations, so $\dim(\f{g}_2)_{g_k}\geq 2$ and   (iii) follows, since $\dim \f{g}_2 = 14$. 
\end{proof}

We provide a short proof of the next result due to the lack of a suitable reference.    

\begin{lemma} \label{lem_jacobi_li} 
If the subgroup generated by $g, h, k\in \bb Z_2^3$ is the entire group,  
then there exist $x \in (\f{g}_2)_g, \, y\in (\f{g}_2)_h, \, z\in (\f{g}_2)_k$ such that $[x,[y,\, z]], \, [y,[z,\, x]]$ are linearly independent.
\end{lemma}

\begin{proof}
We can assume without loss of generality that $g = (1, 1, 0)$, $k = (0, 1, 0)$ and $h = (0, 1, 1)$, since by Equation~\eqref{eq_weyl} there is an automorphism of the   Lie algebra $\f{g}_2$ sending $(\f{g}_2)_g$, $(\f{g}_2)_h$ and $(\f{g}_2)_k$ onto $(\f{g}_2)_{(1, 1, 0)}$, $(\f{g}_2)_{(0, 1, 0)}$ and $(\f{g}_2)_{(0, 1, 1)}$, respectively.
Recall that  ${\bf i} \in \mathcal{O}_{g+k},$ ${\bf j} \in \mathcal{O}_k,$ ${\bf k} \in \mathcal{O}_g,$ and  ${\bf l} \in \mathcal{O}_{k+h},$ and let $x:= D_{{\bf i},{\bf j} }\in (\f{g}_2)_g$, $y := D_{{\bf j}, {\bf l}}\in (\f{g}_2)_h$, and $z := D_{{\bf i},{\bf k}}\in (\f{g}_2)_k$. 				
Taking into account that any three elements in a quaternion subalgebra associate with each other, 
and that $\mathcal{O}$ is an alternative algebra so that three elements associate if two of them are repeated, we get:   
\begin{align*}
D_{{\bf i}, {\bf k}}({\bf i}) &= [[{\bf i}, {\bf k}], {\bf i}] - 3({\bf i}, {\bf k}, {\bf i})=
[-2{\bf j}, {\bf i}] - 0 = 4{\bf k},
\\
D_{{\bf i}, {\bf k}}({\bf j}) &= [[{\bf i}, {\bf k}], {\bf j}] - 3({\bf i}, {\bf k}, {\bf j})=
-2[{\bf j}, {\bf j}] = 0, 
\\
D_{{\bf j}, {\bf l}}({\bf j}) &= [[{\bf j}, {\bf l}], {\bf j}] - 3({\bf j},{\bf l},{\bf j}) =
[2{\bf jl}, {\bf j}] - 0 = 4{\bf l},
\\
D_{{\bf j},{\bf l}}({\bf k})&= 4({\bf jl}){\bf k}-3(({\bf jl}){\bf k}-{\bf j}({\bf lk})) =-4{\bf il}-3(-{\bf il}-{\bf il})=2{\bf il}.
\end{align*}
From the invariance of $D$, we derive that 
\begin{align}
[z, x] & = [D_{{\bf i}, {\bf k}}, D_{{\bf i}, {\bf j}}] = D_{D_{{\bf i}, {\bf k}}({\bf i}), {\bf j}} + 
D_{{\bf i}, D_{{\bf i}, {\bf k}}({\bf j})} = 4D_{{\bf k}, {\bf j}} = -4D_{{\bf j}, {\bf k}}, \notag
\\
[y,[z, x]] & = -4[D_{{\bf j},{\bf l}},D_{{\bf j},{\bf k}}]=-4(D_{D_{{\bf j},{\bf l}}({\bf j}),{\bf k}}+D_{{\bf j},D_{{\bf j},{\bf l}}({\bf k})})
=-16D_{{\bf l},{\bf k}}-8D_{{\bf j},{\bf il}}. \label{[y,[z,x]]}
\end{align}
Similarly, \eqref{identidadciclica} implies $D_{{\bf i},{\bf jl}}=D_{{\bf j}, {\bf il}}-D_{{\bf l}, {\bf k}} $ and we get 		
\begin{equation} \label{[x,[y,z]]1}
[x, [y, z]] = -4D_{{\bf l},{\bf k}} + 8D_{{\bf j}, {\bf il}} + 4D_{{\bf i},{\bf jl}} = -8D_{{\bf l}, {\bf k}} + 
12D_{{\bf j}, {\bf il}}.
\end{equation} 
The required independence follows from Equations~\eqref{[y,[z,x]]} and \eqref{[x,[y,z]]1}, since   it is easy to check that
$D_{{\bf l},{\bf k}}$ and $D_{{\bf j}, {\bf il}}$ are   linearly independent.
\end{proof}  


\subsection{Graded contractions of Lie algebras} \label{chap_graded_contractions}

Throughout this section all vector spaces are assumed to be finite-dimensional over the field $\mathbb{C}$ of complex numbers, $G$ denotes an abelian group, $\la$ a Lie algebra and $\mathbb{C}^\times$ the nonzero complex numbers. The results collected here are an adaptation from \cite{checos06, otro91} to our needs.
	
\begin{define} \label{defepbracket}
Let $\Gamma: \la = \bigoplus_{g\in G} \la_g$ be a $G$-grading on $\la$. 
A \bfemph{graded contraction} of $\Gamma$ is a map $\varepsilon\colon G\times G\to \bb{C}$ such that the vector space $\la$ endowed with product $[x, y]^\varepsilon := \varepsilon(g, h)[x, y]$, for $x\in \la_g,  y\in \la_h$, is a Lie algebra. We write $\la^\ep$ to refer to $(\la, [\cdot, \cdot]^\epsilon)$.
 Note that $\la^\ep$ is indeed a $G$-graded Lie algebra with  $(\la^\ep)_g:=\la_g$.
\end{define}	
Trivial examples of graded contractions are the constant maps $\varepsilon = 1$ and $\varepsilon = 0$, which produce the original algebra $\la^\ep = \la$ and an abelian Lie algebra, respectively.

Here, we are interested in how many graded Lie algebras arise
from a fixed graded algebra by considering its graded contractions; we begin by introducing some notation.

\begin{define} \label{defequivnorm1}
Two graded contractions $\ep$ and $\ep'$ of a $G$-grading $\Gamma$ on $\la$ are called
 \bfemph{equivalent}, written $\ep \sim \ep'$, if the graded Lie algebras $\la^\ep$ and $\la^{\ep'}$ are isomorphic; that is, there exists an algebra isomorphism   $\varphi\colon\la^\ep \to\la^{\ep'}$ satisfying that for any $g\in G$ there is $h \in G$ such that $\varphi((\la^{\ep})_g) = (\la^{\ep'})_h$. 
\end{define}

Notice that, although $\la^\ep$ and $\la^{\ep'}$ are   isomorphic as (ungraded) Lie algebras, it could happen that $\ep \not\sim \ep'$.
An example of this situation follows.

\begin{example} \label{ex_noeslomismo}
Consider the Lie algebra $\la =  \mathfrak{so}(3,\bb C)\oplus\langle z\rangle$, where $z$ denotes a central element and $\{x_1, x_2, x_3\}$ a basis of the skew-symmetric matrices $ \mathfrak{so}(3,\bb C)$ with bracket given by $[x_i, x_{i+1}] = x_{i+2}$, where the subscripts are taken modulo 3. (The field $\bb C$ is not very relevant here.)
Let $H = \bb Z_2^2$ and consider the $H$-grading on $\la$ given by
$$
\la_{h_0}=\{0\},\quad
\la_{h_1}=\langle x_1,z \rangle,\quad
\la_{h_2}=\langle x_2 \rangle,\quad
\la_{h_3}=\langle x_3 \rangle,
$$
where
$h_0 = (0, 0)$, $h_1 = (1, 0)$, $h_2 = (1, 1)$ and $h_3 = (0, 1)$  are all the elements of $H$.
Given a map $\ep\colon H \times H \to \bb C$, we 
write $\ep_{ij} = \ep(h_i, h_j)$. It is straightforward to check that the following maps 
$\ep$ and $\ep'$ are both graded contractions: 
$$
 \ep_{12}=\ep_{21}=1,\quad
 \ep'_{23}=\ep'_{32}=1,\quad
$$
where $\ep_{ij}=0=\ep'_{ij}$ for the pairs of indices which are yet to be assigned a value.
The Lie algebra $\la^\ep = \langle x_1, x_2, x_3, z\rangle$
is  2-step nilpotent: it satisfies  $[x_1, x_2]^\ep = -[x_2,x_1]^\ep = x_3$ and all the remaining brackets among basic elements
are  equal to 0. 
Similarly, in $\la^{\ep'} $ we have 
$[x_2, x_3]^{\ep'} = -[x_3, x_2]^{\ep'} = x_1$, and $x_1$ and $z$ are central.
Then, $\la^\ep$ is isomorphic to $\la^{\ep'}$ via
\begin{equation}\label{eq_ej1}
x_1\mapsto x_2,\quad
x_2\mapsto x_3,\quad
x_3\mapsto x_1,\quad
z\mapsto z.
\end{equation} 
We claim that $\la^{\ep}$ and $\la^{\ep'}$ are not isomorphic as graded algebras, that is, $\ep \not \sim\ep'$; in fact, suppose that $\varphi\colon \la^\ep \to \la^{\ep'}$ is an isomorphism of Lie algebras satisfying that for any $g\in H$ there exists $h \in H$ with $\varphi((\la^{\ep})_g) = (\la^{\ep'})_h$. But then $\varphi((\la^{\ep})_{h_1}) = (\la^{\ep'})_{h_1}$, since these are the only 2-dimensional homogeneous components; which is impossible since $(\la^{\ep'})_{h_1} \subseteq \f{z}(\la^{\ep'})$ and $x_1\in (\la^{\ep})_{h_1}$ is not a central element. 
\end{example}

  Definition \ref{defequivnorm1} can be confusing at first, since in the literature the relation $\sim$ sometimes refers to isomorphic Lie algebras. 
However  it makes sense to impose that the pieces of the grading are preserved without breaking.
 In fact, it    is very much in line with the philosophy of the definition of graded contraction.
 
Sometimes is important that the isomorphism is a graded algebra isomorphism too. 
The next equivalence relation will be very valuable  in our study.

\begin{define} \label{defequivnorm2}
Two graded contractions $\ep$ and $\ep'$ of a $G$-grading $\Gamma$ on $\la$ are called
  \bfemph{strongly equivalent}, written $\ep \approx \ep'$, if the graded Lie algebras $\la^\ep$ and $\la^{\ep'}$ are isomorphic as graded algebras; that is, there exists an isomorphism of Lie algebras $\varphi\colon \la^\ep \to \la^{\ep'}$  such that $\varphi((\la^{\ep})_g) = (\la^{\ep'})_g$ for any $g \in G$, 
  also called a \emph{graded isomorphism}.
\end{define}

Clearly, $\ep \approx \ep'$ implies $\ep \sim \ep'$. Notice that the converse does not need to be true as shown in the next example.

 \begin{example} \label{ex_unamasfuerte}
Consider the Lie algebra $\la = \mathfrak{so}(3,\bb C)$ of skew-symmetric matrices with basis $\{x_1, x_2, x_3\}$ as in Example \ref{ex_noeslomismo}. Let $H = \bb Z_2^2$ and write its elements as in Example \ref{ex_noeslomismo}; consider the $H$-grading given by
\[
\la_{h_0}=\{0\},\quad
\la_{h_1}=\langle x_1 \rangle,\quad
\la_{h_2}=\langle x_2 \rangle,\quad
\la_{h_3}=\langle x_3 \rangle.
\]
Let $\ep, \ep': H \times H \to \bb C$ be the graded contractions defined  as in Example \ref{ex_noeslomismo}.

One can easily see that the Lie algebra $\la^\ep = \langle x_1,x_2,x_3\rangle$ is again 2-step nilpotent with $[x_1,x_2]=-[x_2, x_1] = x_3\in\f{z}(\la^{\ep})$.
Now $\ep\sim\ep'$ since 
\begin{equation}\label{eq_ej2}
x_1 \stackrel{\varphi}{\mapsto} x_2,\quad
x_2 \stackrel{\varphi}{\mapsto} x_3,\quad
x_3 \stackrel{\varphi}{\mapsto} x_1,\quad
\end{equation}
is a Lie algebra isomorphism satisfying that $\varphi((\la^{\ep})_{h_1}) = (\la^{\ep'})_{h_2}$, $\varphi((\la^{\ep})_{h_2}) = 
(\la^{\ep'})_{h_3}$, 
and $\varphi((\la^{\ep})_{h_3}) = (\la^{\ep'})_{h_1}$. But $\ep \not\approx \ep'$, since $(\la^{\ep})_{h_3} = \f{z}(\la^{\ep})$, while 
$(\la^{\ep'})_{h_3} \ne\f{z}(\la^{\ep'})$.
\end{example}

 \begin{remark}\label{re_nodeltotal}
 At this point it is worth noticing that the maps given in \eqref{eq_ej1} and \eqref{eq_ej2} are automorphisms of the original $\la$, although this is just a coincidence.
For instance, consider the $H$-graded Lie algebra $\la = \mathfrak{so}(3,\bb C)$ as in the previous example. Then 
the graded contractions defined by $\ep_{12} = \ep_{21} = 1$,
$\ep'_{12} = \ep'_{21} = 4$, and zero elsewhere, are strongly equivalent, 
but the graded isomorphism $\varphi\colon \la^\ep \to \la^{\ep'}$ given by $\varphi(x_1) = \frac12x_1$, $\varphi(x_2) = \frac12x_2$ and $\varphi(x_3) = x_3$ is not, of course, an automorphism of $\mathfrak{so}(3,\bb C) $.
 \end{remark}

Given an arbitrary map $\ep\colon G\times G\to \bb{C},$ it is natural asking under what conditions $\ep$ is a graded contraction of $\Gamma$. Such a question is not as easy as it seems since it relies on the properties of $\Gamma$.	

\begin{remark} \label{rem_gc}
Let $\ep\colon G\times G\to \bb{C}$ be an arbitrary map. 

\noindent (i) We define a ternary map $\ep\colon G \times G \times G \to \bb{C}$ by 
$$\ep(g, h, k) := \ep(g, h + k)\ep(h, k).$$
It is straightforward to check that  $\ep$ is a graded contraction of $\Gamma$ if and only if (a1) and (a2) below 
hold for all $g, h, k \in G$ and $x\in \la_g, \, y\in \la_h, \, z\in\la_k$:
\begin{enumerate}
\item[(a1)] $\big(\ep(g, h) - \ep(h, g)\big)[x, y] = 0,$
\item[(a2)] $\big(\ep(g, h, k) - \ep(k, g, h)\big)[x, [y, z]]
+ \big(\ep(h, k, g) - \ep(k, g, h)\big)[y, [z, x]] = 0$.
\end{enumerate}

\smallskip

\noindent (ii) Condition (a1)  is equivalent to $\ep$ being {\it nearly symmetric}: $\ep(g, h) = \ep(h, g)$, for all $g, h \in G$ with $[\la_g,\la_h] \neq 0$. A sufficient condition (not necessary, in general) for $\ep$ to satisfy (a2) is $\ep(g, h, k) = \ep(h, k, g)$, for all $g, h, k \in G$.
\end{remark}

The following example will be very important for our purposes. 

\begin{example} \label{newgc}
Let $\ep$ be a graded contraction of $\Gamma$  and $\alpha\colon  G \to \mathbb{C}^\times$ an arbitrary map. 
The map $\ep^\alpha\colon G\times G \to \mathbb{C}^\times$ defined as
\[
\ep^\alpha(g, h) := \ep(g, h) \dfrac{\alpha(g)\alpha(h)}{\alpha(g + h)}, \quad \forall \, g,h\in G,
\] 
is always a graded contraction (since it satisfies (a1) and (a2)). If $\alpha$ is a group homomorphism, then $\ep^\alpha = \ep$, but in general $\ep^\alpha$ is different from $\ep$. 
\end{example}

This inspires another equivalence relation.

\begin{define} \label{defequivnorm3}
Two graded contractions $\ep$ and $\ep'$ of a $G$-grading $\Gamma$ on $\la$ are called
 \bfemph{equivalent via normalization}, written $\ep \sim_n \ep'$, if there exists a map $\alpha\colon  G \to \mathbb{C}^\times$ such that $\ep' = \ep^\alpha$ (as per in Example \ref{newgc}). 
\end{define}

Notice that $\ep \sim_n \ep'$ implies that $\ep \approx \ep'$, since  
the map $\varphi\colon \la^{\ep^\alpha} \to \la^{\ep}$ given by 
\begin{equation*} \label{eq_isomorfas}
\varphi(x) = \alpha(g)x, \ \, \forall \, x \in \la_g,
\end{equation*} 
is a graded algebra isomorphism in the sense of Definition~\ref{defequivnorm2}. 
Some authors have conjectured (see \cite[Conjecture 2.15]{ref33}) that the converse is also true, that is, $\ep \approx \ep'\Rightarrow\ep \sim_n \ep'$. 
Although, this is in general false as shown in the next example,  
we will prove in \S\ref{sec5} that this result is true for $\Gamma_{\f{g}_2}$.

 \begin{example} \label{ex_encontreunoenC}
 Let $\la = \f{sl}(2,\bb C)\oplus \f{sl}(2,\bb C)$ be the Lie algebra consisting on the direct sum of two simple ideals, both copies of traceless matrices. For $i = 1, 2$, consider    
\[
\left \{x_i=\tiny\begin{pmatrix}1&0\\0&-1\end{pmatrix},e_i=\tiny\begin{pmatrix}0&1\\0&0\end{pmatrix},f_i=\tiny\begin{pmatrix}0&0\\1&0\end{pmatrix}\right\}
\]
the standard basis of $\f{sl}(2,\bb C)$. Let $H = \bb Z_2^2$ and use the same notation as in Example~\ref{ex_noeslomismo}.
 Then $\la$ becomes a $H$-graded algebra with 
 $$
 \la_{h_0}=\langle x_1,x_2\rangle,\quad
 \la_{h_1}=\langle e_1,f_1 \rangle,\quad
  \la_{h_2}=\langle e_2,f_2 \rangle,\quad
   \la_{h_3}=\{0\}.
 $$
Let $\ep,\ep'\colon H \times H \to \bb C$ be the graded contractions given by $\ep\equiv 1$ (constant map) and
$$
 \ep'_{01}= \ep'_{10}=  -1,\quad
 \ep'_{ij}=1, \, \, \textrm{ elsewhere.}
$$
Then $\la^\ep=\la$, while the products in $\la^{\ep'}$ are given by $[a_1,a_2]^{\ep'}=0$ if $a_i\in\{x_i,e_i,f_i\}$,
$$
\begin{array}{lll}
[x_1,e_1]^{\ep'}=-2e_1,\ & [x_1,f_1]^{\ep'}=2f_1, \ &[e_1,f_1]^{\ep'}=x_1,\\
{[x_2,e_2]^{\ep'}}=2e_2,\  &[x_2,f_2]^{\ep'}=-2f_2, \ &[e_2,f_2]^{\ep'}=x_2.
\end{array}
$$
It is straightforward to check that 
the map $\varphi\colon \la^\ep \to \la^{\ep'}$ given by 
\[
\varphi(x_1)=-x_1,\ 
\varphi(x_2)=x_2,\quad 
\varphi(e_1)= -e_1,\ 
\varphi(f_1)= f_1,\quad 
\varphi(e_2)=e_2,\ 
\varphi(f_2)=f_2,\ 
\]  
 is a graded isomorphism, which implies that $\la^{\ep'}$ is a Lie algebra, $\ep'$ is a graded contraction, and $\ep \approx \ep'$.
 On the other hand, we claim $\ep \not \sim_n \ep'$. 
 In fact, assume there exist $\alpha,\beta,\gamma\in\bb C^\times$ such that the map $\phi\colon \la^\ep \to \la^{\ep'}$ given below is a Lie algebra isomorphism,
\[
\phi(x_1)=\alpha x_1,\ 
\phi(x_2)=\alpha x_2,\quad 
\phi(e_1)=\beta e_1,\ 
\phi(f_1)=\beta f_1,\quad 
\phi(e_2)=\gamma e_2,\ 
\phi(f_2)=\gamma f_2.\ 
\]
In such a case, we obtain that  
$$
\begin{array}{lcl}
2\beta e_1=\phi([x_1,e_1]^\ep)=\alpha\beta[x_1,e_1]^{\ep'}= -2\alpha\beta e_1&\Rightarrow&\alpha=-1, \\
2\gamma e_2=\phi([x_2,e_2]^\ep)=\alpha\gamma[x_2,e_2]^{\ep'}= 2\alpha\gamma e_2&\Rightarrow&\alpha=1,
\end{array}
$$
a contradiction.
 \end{example}

\begin{remark} 
At this point the reader might have noticed the similarity between the  equivalence relations $\sim$, $ \approx  $ and $\sim_n$ and the automorphism group $\Aut(\Gamma)$, the stabilizer $\Stab(\Gamma)$ and  the diagonal group $\Diag(\Gamma)$ of $\Gamma$. To be more precise, one can easily show that $\ep \sim_n \ep'$ if and only if there exist a graded isomorphism $\varphi\colon  \la^{\ep}\to\la^{\ep'}$ and a collection of nonzero scalars $\{\alpha_g: g\in G\} \subseteq \bb C^\times$ with $\varphi\vert_{\la_g}=\alpha_g\mathrm{id}_{\la_g}$.
\end{remark}

The classification problem consists in determining the equivalence classes of graded contractions of $\Gamma$ via $\sim$. This problem 
turns out to be   closely related to determining the equivalence classes via $\approx$, provided that the Weyl group of $\Gamma$ is known 
(see Proposition~\ref{colli}).
Now, finding the equivalence classes via $\sim_n$ is much easier than finding the equivalence classes via $\approx$, since the first ones can be computed in an algorithmic way; while one may describe the task of finding graded isomorphisms (i.e. equivalence class via $\approx$) as a wild jungle, hard to be approached algorithmically, at least if the homogeneous components are not   1-dimensional.



\section{Graded contractions of $\Gamma_{\f{g}_2}$} \label{sec3}


This section is devoted to the study of the graded contractions of $\Gamma_{\f{g}_2}$, the fine $\bb{Z}_2^3$-grading on $\f{g}_2$
introduced in  \eqref{graduaciong2}. 
We develop here the machinery needed to classify such graded contractions up to equivalence via normalization  in \S\ref{sec4}, 
  up to  strong equivalence in \S\ref{sec5}, and up to equivalence in \S\ref{sec_revision}.

\subsection{Admissible maps}\label{sec31}
 	
Throughout this section, to ease the notation, we will write $G$ to refer to  
$\bb{Z}_2^3$,   $e$ to denote its identity element and $\la$ to denote the Lie algebra $\f{g}_2$.

\begin{define}
We call $\ep\colon G\times G \to \bb{C}$ an \bfemph{admissible} map if $\ep(g, g) = \ep(e, g) = \ep(g, e) = 0$, for all $g \in G$.
\end{define}

Given a graded contraction of $\Gamma_{\f{g}_2}$, we first prove that there is always an equivalent admissible graded contraction.
	
\begin{lemma}\label{ref_existeadmisible}
Let $\ep$ be a graded contraction of $\Gamma_{\f{g}_2}$. Then there exists an admissible graded contraction $\ep'$ of $\Gamma_{\f{g}_2}$ equivalent to $\ep$.
\end{lemma}

\begin{proof}
For $g, h \in G $ we let $\ep'(g, h) := \ep(g, h)$ if $g \neq h$, $g,h\ne e$ and 0 otherwise. 
Notice that $\ep'$ is clearly admissible and satisfies  
$[\cdot, \cdot]^{\ep} = [\cdot, \cdot]^{\ep'}$. In fact, for $x \in \la_g$ and $y \in \la_h$ we distinguish two cases:
\begin{itemize}
\item[-] If $g\ne h$ and $g, h\ne e$, then $[x, y]^{\ep'} = \ep'(g, h)[x, y] = \ep(g, h)[x, y] = [x, y]^{\ep}$. 
\item[-] If either $g = h$ (so $g + h = e$) or $g = e$ or $h = e$, then $[\la_g, \la_h] = 0$, since $\la_e=0$. In particular, we  have $[x,y] = 0$, which implies $[x, y]^{\ep'} =   [x, y]^{\ep}=0$.
\end{itemize}
Thus $[\cdot, \cdot]^{\ep} = [\cdot, \cdot]^{\ep'}$, and so the identity is a Lie algebra isomorphism between  
$\la^{\ep'}$ and $\la^{\ep}$. 
\end{proof}

Consider $\mathcal{G}$ the set of all admissible graded contractions of $\Gamma_{\f{g}_2}$, and $\tilde{\mathcal{G}}$ the set of all graded contractions of $\Gamma_{\f{g}_2}$; clearly, $\mathcal{G} \subseteq \tilde{\mathcal{G}}$. 
Lemma~\ref{ref_existeadmisible} allows us to  shift   from classifying graded contractions of $\Gamma_{\f{g}_2}$ up to equivalence (i.e., Lie algebras obtained by graded contractions of ${\f{g}_2}$ up to   isomorphism) to classifying the orbits in the quotient set $\mathcal{G}/\sim$. In other words, there is a bijection from $\mathcal{G}/\sim$ onto $\tilde{\mathcal{G}}/\sim$.

Admissible graded contractions of $\Gamma_{\f{g}_2}$ have many symmetry-type properties: 
	
\begin{lemma} \label{rem_admiss}
Let $\ep$ be an admissible graded contraction of $\,\Gamma_{\f{g}_2}$ and $g, h, k\in G$. Then:
\begin{enumerate}
\item[\rm (i)] $\ep(g, h) = \ep(h, g)$;
\item[\rm (ii)] $\ep(e, \cdot, \cdot) = \ep(\cdot, e, \cdot) = \ep(\cdot, \cdot, e) = 0$;
\item[\rm (iii)] $\ep(g, k, k) = 0$;
\item[\rm (iv)] $\ep(g, h, k) = \ep(g, k, h)$;
\item[\rm (v)]  $\ep(h + k, h, k) = 0$.
\end{enumerate}
\end{lemma}

\begin{proof}
(i) If $e \in \{g, h, g + h \}$, then $\ep(g, h) = \ep(h, g) = 0$, since we are assuming that $\ep$ is admissible. 
Suppose now that $e \notin \{g, h, g + h\}$. Then $[\la_g, \la_h] = \la_{g + h} \neq 0$ by Lemma~\ref{lem_base} (iii). The result now follows from Remark~\ref{rem_gc} (ii). 

The proofs of (ii)--(v) are straightforward, and therefore omitted.
\end{proof}

Not every admissible map is an admissible graded contraction of $\Gamma_{\f{g}_2}$, but
a necessary and sufficient condition is given in the next result.  
The advantage of admissible graded contractions is just that this condition  does not explicitly refer    to the elements in $\f{g}_2 $.
	 
\begin{prop}\label{pr_condicionesgrcont}
An admissible map $\ep\colon  G \times G \to \bb{C}$ is a graded contraction of $\Gamma_{\f{g}_2}$ if and only if the following conditions hold for all $g, h, k \in G$:
\begin{enumerate}
\item[\rm (b1)] $\ep(g, h) = \ep(h, g)$,  
\item[\rm (b2)] $\ep(g, h, k) = \ep(k, g, h)$, provided that $G = \langle g, h, k \rangle$. 
\end{enumerate}
\end{prop}

\begin{proof}
Let $\ep\colon  G \times G \to \bb{C}$ be admissible. Suppose first that $\ep$ satisfies (b1) and (b2). We prove that (a1) and (a2) from Remark~\ref{rem_gc} are satisfied. Notice that (a1) clearly follows from (b1). Take $g, h, k \in G$ and $x \in \la_g$, $y \in \la_h$ and $z \in \la_k$ homogeneous elements. If $G = \langle g, h, k \rangle$, then (a2) follows from (b2). Otherwise, we can assume without loss of generality that $k \in \langle g, h \rangle = \{e, g, h, g + h\}$. We distinguish a few cases: 
\begin{itemize}
\item[-] $k = e$: then $\la_k = 0$ and (a2) follows. 
\item[-]  $k = g$: this implies $k + g = e$ and so $\la_{k + g} = 0$. From here, we obtain that $[z, x] = 0$, and (a2) now follows from
$\ep(g,h,g)=\ep(g,g,h)$ by (b1).
\item[-]  $k = h$: proceed like in the previous case. 
\item[-]  $k = g + h$: then $[x, [y, z]] \in \la_{g + (h + (g + h))} = \la_e = 0$ and similarly $[y, [z, x]] \in \la_e = 0$. Thus, (a2) holds.  
\end{itemize}

Conversely, suppose that $\ep$ is an admissible graded contraction. Then (b1) holds by Lemma~\ref{rem_admiss} (i). 
It remains to show the validity of (b2). 
We know that $\ep$ satisfies (a1) and (a2) from Remark~\ref{rem_gc}.
Suppose that $G = \langle g, h, k \rangle$, then we can find $x \in \la_g$, $y \in \la_h$ and $z \in \la_k$ such that $[x, [y, z]]$ and $[y, [z, x]]$ are linearly independent by Lemma~\ref{lem_jacobi_li}. From here, an application of (a2) yields (b2).
\end{proof}


\subsection{A combinatorial approach:  supports and nice sets} \label{sec32}


The results in \S\ref{sec31} allow us to focus our attention on admissible maps satisfying conditions (b1) and (b2) from Proposition~\ref{pr_condicionesgrcont}; this simplifies
 our classification problem considerably since we can forget about the grading itself and deal with the grading group.

In this section, we reformulate the problem of classifying (up to equivalence) admissible maps satisfying (b1) and (b2) onto a combinatorial problem. We begin by introducing some notation. As before, we write $G$ to denote $\bb{Z}_2^3$, and we write its elements as follows:  
\[
\begin{array}{cccc}
g_0 := (0, 0, 0), & \quad g_1 :=(1, 0, 0), & \quad g_2 := (0, 1, 0), & \quad g_3 :=(0, 0, 1),
\\
g_4 := (1, 1, 1), & \quad g_5:= (1, 1, 0), & \quad g_6:= (1, 0, 1), & \quad g_7 := (0, 1, 1).
\end{array}
\]
We let $I: = \{1, 2, \ldots, 7\}$, $I_0 := I \cup \{0\}$, 
and introduce a commutative  binary operation $\ast$ on $I_0$ as follows: for $i, j \in I_0$ we let $i \ast j$ to be the element in $I_0$ such that $g_i + g_j = g_{i \ast j}$. Restricted to $I$, this  
operation can be pictured as: 
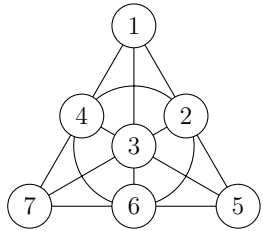
\begin{center}
\begin{tikzpicture}[scale=0.8, every node/.style={transform shape}] 
\draw 
(30:1) node[circle, draw, fill=white]{2}  -- (210:2) node[circle, draw, fill=white]{7}
(150:1) node[circle, draw, fill=white]{4} -- (330:2) node[circle, draw, fill=white]{5}
(270:1) node[circle, draw, fill=white]{6} -- (90:2) node[circle, draw, fill=white]{1}
(90:2)  -- (210:2) -- (330:2) -- cycle
(0:0) node[circle, draw, fill=white]{3}   circle (1);
\end{tikzpicture}
\end{center}
More precisely, for $i \neq j$ in $I$, the element $i \ast j \in I$ is the third element in the unique line containing $i$ and $j$. This means that $\{i, j, i \ast j\}$ is one of the lines of the so called Fano plane $P^2(\bb F_2)$ (pictured above).  
Recall that the Fano plane is a finite projective plane with the smallest possible number of points and lines: 7 points and 7 lines, with 3 points on every line and 3 lines through every point. 

\begin{defis}
Pairwise distinct elements $i, j, k \in I$ are called \bfemph{generative}
if $k\neq i\ast j$; this is equivalent to saying that $G = \langle g_i, g_j, g_k \rangle$. 
We also say that $\{i, j, k\}$ is a \bfemph{generating triplet}, since the definition does not depend on the order of the elements $i, j$ and $k$.  
\end{defis}

Consider the set of 21 edges of the Fano plane:
$$X := \{ \{i, j\} \mid i, j \in I, \, i \neq j \}.$$

\begin{notation}
\noindent For a map $\eta\colon  X \to \bb{C}$, we write $\eta_{ij} := \eta(\{i, j\})$ and $\eta_{ijk} := \eta_{i \, j*k}\, \eta_{j k}$, provided $i, j, k$ are generative. (It is well defined, since $\{i, j*k\}\in X $.) We consider the set 
\begin{equation}\label{def_A}
\mathcal{A}:= \{\eta\colon X\to \bb{C} \mbox{ such that } \eta_{ijk} = \eta_{jki}, \text{ for all } i, j, k \in I \text{ generative}\}.
\end{equation} 
\end{notation}

Working with $\mathcal{A}$ is more convenient than dealing with $\mathcal{G}$; likely, we will be able to do so, since there is a natural bijective correspondence between these two sets.

\begin{prop} \label{bijection}
The map $\Phi\colon \mathcal G \to \mathcal A$ given by  
$\Phi(\ep)(\{i, j\}) := \ep(g_i, g_j)$ is bijective, with inverse defined as 
$\Phi^{-1}(\eta)(g_i, g_j) = \eta(\{i, j\})$ if $\{i, j\} \in X$, and 0 otherwise. 
\end{prop}	

\begin{proof}
To ease the notation, for $\ep \in \mathcal G$ and $\eta \in \mathcal A$, we write $\eta^\ep \equiv \Phi(\ep)$ and $\ep_\eta \equiv \Phi^{-1}(\eta)$. 
We first check that $\eta^{\ep} \in \mathcal{A}$ provided $\ep \in \mathcal{G}$.
Notice that $\eta^\ep$ is well defined since  $\ep(g_i, g_j) = \ep(g_j, g_i)$. For $i, j, k \in I$ generative, we have that
\begin{align*}
\eta^{\ep}_{ijk} & =
\eta^{\ep}_{i \, j*k}\, \eta^{\ep}_{jk} =
\ep(g_i, g_{j*k})\ep(g_j,g_k)=
\ep(g_i, g_j + g_k)\ep(g_j, g_k) = \ep(g_i, g_j, g_k),  
\\
\eta^{\ep}_{jki} & =
\eta^{\ep}_{j \, k*i}\, \eta^{\ep}_{ki} =
\ep(g_j, g_{k*i})\ep(g_k, g_i)=
\ep(g_j, g_k + g_i)\ep(g_k, g_i) = \ep(g_j, g_k, g_i), 
\end{align*}
and $\eta^{\ep}_{ijk} = \eta^{\ep}_{jki}$ by Proposition~\ref{pr_condicionesgrcont} (b2). 
Second, we check that $\ep_\eta \in \mathcal{G}$ provided $\eta \in \mathcal{A}$. Notice that $\ep_\eta$ is admissible by definition, and satisfies Proposition~\ref{pr_condicionesgrcont} (b1). If $i, j, k$ are generative, then $\ep_\eta(g_i, g_j, g_k) = \eta_{ijk} = \eta_{jki} = \eta_{kij} = \ep_\eta(g_k, g_i, g_j)$, and (b2) follows.  

\noindent Lastly, notice that the maps
$\mathcal{G} \to \mathcal{A}$, $\ep \mapsto \eta^{\ep}$ and $\mathcal A \to \mathcal{G}$, $\eta \mapsto \ep_\eta$, are inverses.
\end{proof}

In what follows, we will refer to admissible graded contractions and maps in $\mathcal{A}$ interchangeably. In particular, two maps $\eta, \eta'\in \mathcal{A}$ are equivalent (respectively, strongly equivalent) if $\ep_\eta$ and $\ep_{\eta'}$ are so; in such a case, we will use the same notation $\eta \sim \eta'$ (respectively, $\eta \approx \eta'$).

An invariant in the classification of strong equivalence classes in $\mathcal A$ is their support, defined in the usual way: the \bfemph{support} $S^\eta$ of $\eta \in \mathcal A$ is given by 
$$S^\eta:=\{t\in X\mid  \eta(t) \ne 0 \}.$$ 

\begin{lemma} \label{unaporsoporte}
Let  $\eta, \eta' \in \mathcal{A}$. 
 If $\eta \approx \eta'$, then $S^\eta = S^{\eta'}$.
\end{lemma}	

\begin{proof} 
We denote   $\la=\f{g}_2$ and $\la_i=(\f{g}_2)_{g_i}$.
Let $\varphi\colon \la^{\ep_{\eta}}\to \la^{\ep_{\eta'}}$ be an isomorphism such that $\varphi(\la_i)=\la_i$ for all $i\in I$.
Note, for $\{i,j\}\in X$, any $x\in\la_i$  and $y\in\la_j$, that  
$$
 \eta_{ij}\varphi([x,y])=\varphi\big([x,y]^{\ep_\eta}\big) = [ \varphi(x), \varphi(y)]^{\ep_{\eta'}}=\eta'_{ij}[ \varphi(x), \varphi(y)].
$$ 
This gives $ \eta_{ij}\ne0$ if and only if  $ \eta'_{ij}\ne0$, taking Lemma~\ref{lem_base} (iii) into consideration. In other words, $S^\eta = S^{\eta'}$.
\end{proof}

In order to determine the possible supports of the admissible graded contractions, we
 will prove that these satisfy the absorbing-type property defined below.

\begin{defis} \label{def_nice}
 For $i, j, k \in I$ generative, we let $P_{\{i,j,k\}}$ to be the subset of $X$: 
\[
P_{\{i,j,k\}} := \{\{i, j\}, \{j,k\},  \{k,i\}, \{i, j\ast k\}, \{j, k\ast i\}, \{k, i\ast j\}\}.
\] 
A subset $T\subseteq X$ is called \bfemph{nice} if $\{i, j\},   \{i\ast j, k\} \in T$ implies $P_{\{i,j,k\}} \subseteq T$, for all $i, j, k \in I$ generative.

It is clear that the trivial subsets $X$ and $\emptyset$ are both nice sets. Non-trivial nice sets and some pictures are shown in Definition~\ref{niceSubsets}.
\end{defis}

 As mentioned,    nice sets and supports of admissible graded contractions are closely related.

\begin{prop} \label{suppisnice}
If $\eta\in\mathcal A$, then the support $S^\eta$ is a nice set.
\end{prop}  

\begin{proof}
  Let $\eta \in \mathcal{A}$ and $i, j, k \in I$ generative such that $\{i, j\}$ and $\{k, i \ast j\}$ are in $S^\eta$. We need to prove that 
$P_{\{i, j, k\}} \subseteq S^\eta$. From $\{i, j\}, \{k, i \ast j\} \in S^\eta$ we have that $\eta_{ij} \neq 0$ and $\eta_{k \, i \ast j} \neq 0$. Then $\eta_{kij} \neq 0$, which implies that $\eta_{ijk} = \eta_{jki} \neq 0$, since $\eta \in \mathcal{A}$. These yield that $\eta_{i \, j \ast k}$, $\eta_{jk}$, $\eta_{j \, k \ast i}$ and $\eta_{ki}$ are all nonzero. That is to say that $\{i, j \ast k\}, \{j, k\}, \{j, k \ast i\}$ and $\{k, i\}$ are all in $S^\eta$, proving that $S^\eta$ is nice.
\end{proof}

\begin{prop} \label{defetaT}
Every non-trivial nice subset $T$ of $X$ gives rise to an element  of $\mathcal{A}$ with support $T$  and image $\{0, 1\}$.
\end{prop}   

\begin{proof}
Let $T$ be a nice subset of $X$ and $\eta^T\colon  X \to \bb{C}$ given by 
\begin{equation}\label{defietaT}
\eta^T(t) :=
\begin{cases}
1, &$if $ t\in T, \\
0, &$if $ t\notin T.
\end{cases}
\end{equation}
Clearly, $S^{\eta^T} = T$. Suppose that $i, j, k \in I$ are generative, and let us show that $\eta^T_{ijk} = \eta^T_{jki}$. We distinguish two cases:
\begin{itemize}
\item[-] $P_{\{i, j, k\}} \subseteq T$. In this case, one can easily check that $\eta^T_{ijk} = 1 = \eta^T_{jki}$.
\item[-] $P_{\{i, j, k\}} \nsubseteq T$. We claim that $\eta^T_{ijk} = 0 = \eta^T_{jki}$. Suppose on the contrary that either $\eta^T_{ijk} \neq 0$ or $\eta^T_{jki} \neq 0$. If $0 \neq \eta^T_{ijk} = \eta^T_{i \, j \ast k} \eta^T_{jk}$, then $\{i,  j \ast k\}, \{j, k\} \in T$, which implies $P_{\{  j, k, i\}} \subseteq T$ since $T$ is nice. If $\eta^T_{jki} \neq 0$, one can reach a contradiction in a similar way. Thus, necessarily $\eta^T_{ijk} = \eta^T_{jki} = 0$.
\end{itemize}
In any case, we have proved that $\eta^T_{ijk} = \eta^T_{jki}$, which tells us that $\eta^T \in \mathcal{A}$.
\end{proof}

In what follows, we make  the group $\Aut(\bb Z_2^3)  $  act on  the set of admissible maps $\mathcal{A}$ in \eqref{def_A}.

\begin{define}
A bijective map $\sigma\colon I \to I$ is called a \bfemph{collineation}   of $I$ if $\sigma(i*j) = \sigma(i)*\sigma(j)$, for all $i, j \in I$, $i\ne j$. 
Any collineation sends lines to lines, where, for $i, j \in I$ distinct, we denote the set $L_{ij} := \{i, j, i \ast j\}$ as the \bfemph{line} in $I$ containing $i$ and $j$. Notice that $L_{ij} = L_{i \, i \ast j} = L_{j \, i \ast j}$ and that three pairwise distinct elements $i, j, k \in I$ form a line if and only if they are not generative. \

We write $S_\ast (I)$ for the set consisting on all the collineations of $I$, which is a subgroup of the symmetric group on $I$. 
By the Fundamental theorem of projective geometry, the full collineation group of the Fano plane (also called automorphism group, or symmetry group) is the projective linear group $S_\ast (I)\cong \Aut(\bb Z_2^3) \cong \mathrm{PGL}(3, 2)$, which is a simple group of order 168 \cite{Hirschfeld79}.  
\end{define}

\begin{remark}\label{permutaciones}
Collineations preserve generating triplets: if $i, j, k \in I$ are generative, then $\sigma(i), \sigma(j), \sigma(k)$ are generative for all $\sigma \in S_\ast(I)$. Moreover, and this is important, for any two generating triplets $\{i, j, k\}$ and $\{i', j', k'\}$, there is a unique $\sigma\in S_\ast(I)$ such that $\sigma(i) = i'$, $\sigma(j) = j'$, $\sigma(k) = k'$. 
\end{remark}

\begin{define}\label{de_etasigma}
For $\sigma \in S_\ast(I)$ and $\eta \in \mathcal{A}$, we define 
$$
\eta^\sigma\colon  X \to \bb{C}\quad\textrm{by}\quad\eta^\sigma_{ij}: = \eta_{\sigma(i)\, \sigma(j)},
$$
 for all 
$\{i, j\} \in X$. Note that 
 $\eta^\sigma\in\mathcal{A}$, consequence of  $\eta^\sigma_{ijk} =\eta_{\sigma(i) \sigma(j) \sigma(k)} $ and Remark~\ref{permutaciones}.
 This gives an action $S_\ast(I)\times \mathcal{A}\to  \mathcal{A}$, $(\sigma,\eta)\mapsto \sigma\cdot\eta:=\eta^\sigma$.
\end{define}

This action preserves equivalence classes, since it translates the action of the Weyl group 
 of $ \Gamma_{\f{g}_2}$.
%

\begin{prop} \label{colli}    
For $\sigma \in S_\ast(I)$ and $\eta \in \mathcal{A}$,   ${\eta^\sigma}\sim \eta$.
\end{prop}

\begin{proof} Write $\la$ for $\f{g}_2$ and $\la_i$ for the homogeneous component  $(\f{g}_2)_{g_i}$.
The map $\hat\sigma\colon G\to G$ defined by $\hat\sigma(g_i) = g_{\sigma(i)}$ for $i \in I$ and 
$\hat\sigma(g_0) = g_0$ is a group automorphism since $\sigma \in S_\ast(I)$. 
Proceeding like in \eqref{eq_Cayleytriplets}, we can find an automorphism $f_{\sigma}\colon\mathcal{O}\to\mathcal{O}$ such that $f_\sigma(\mathcal{O}_g) = \mathcal{O}_{\hat\sigma(g)}$ for all $g \in G$. Consider the induced map $\tilde f_{\sigma}\colon\f{g}_2\to \f{g}_2$ given by 
$\tilde f_{\sigma}(d) = f_{\sigma}df_{\sigma}^{-1}$, which is an automorphism of Lie algebras satisfying that $\tilde f_{\sigma}(\la_i)= \la_{\sigma(i)}$ for all $i\in I$.
We claim that $\tilde f_\sigma\colon \la^{\ep_{\eta^\sigma}} \to \la^{\ep_{\eta }}$ is an isomorphism. 
In fact, for $i,j\in I$, $i\ne j$, $x \in \la_{ i}$ and $y \in \la_{ j}$, we have that 
\begin{equation*}\label{eqs_deultimahora}
\tilde f_{\sigma}([x,y]^{\ep_{\eta^\sigma}})=\tilde f_{\sigma}( \eta^\sigma_{ij}[x,y])=
\eta^\sigma_{ij}\tilde f_{\sigma}( [x,y])=\eta_{\sigma(i)\sigma(j)}[\tilde f_{\sigma}(x),\tilde f_{\sigma}(y)]=[\tilde f_{\sigma}(x),\tilde f_{\sigma}(y)]^{\ep_\eta},
\end{equation*}
since $\tilde f_{\sigma}(x)\in\la_{\sigma(i)}$ and $\tilde f_{\sigma}(y)\in\la_{\sigma(j)}$. 
This shows that $\ep_{\eta^\sigma} \sim \ep_\eta$, and so $\eta^\sigma \sim \eta$.
%
 \end{proof}
 
 When a collineation $\sigma$ \emph{moves} an admissible map, it moves its support accordingly. 
 More precisely:


\begin{define}
Extend the action of the   group of collineations $S_\ast(I)$ to the set $\mathcal P(X)$ by defining $\tilde \sigma(\{i, j\}): = \{\sigma(i), \sigma(j)\}$, for all $\sigma \in S_\ast(I)$ and $\{i, j\} \in X$.
We say that two subsets $T$ and $T'$ of $X$ are \bfemph{collinear}, and we write $T \sim_c T'$,  
if there exists $\sigma \in S_\ast(I)$ such that $\tilde \sigma(T) = T'$. 
\end{define}

Notice that $\sim_c$ is also an equivalence relation on the family of nice sets, since $\tilde{\sigma}(P_{\{i, j, k\}}) = P_{\{\sigma(i),\sigma(j),\sigma(k)\}}$, which implies that $\tilde{\sigma}(T)$ is a nice subset of $X$ provided $T$ is so.
As mentioned, the natural action of $S_\ast(I)$ on $\mathcal A$ is compatible with the action on the supports.  
\begin{lemma}\label{le_relacionsoportes}
For $\sigma \in S_\ast(I)$ and $\eta \in \mathcal{A}$,  
$S^{\eta^\sigma}\sim_c S^{\eta}$.
\end{lemma}

\begin{proof} 
This is quite clear. From the own Definition~\ref{de_etasigma},
$\{i, j\}\in S^{\eta^\sigma}$ if and only if 
$\tilde \sigma(\{i, j\}): = \{\sigma(i), \sigma(j)\}\in S^{\eta}$. That is, $\tilde \sigma(S^{\eta^\sigma})=S^\eta$.
\end{proof}
 
From all the above, given $\eta \in \mathcal{A}$ with support $T$, for any $T'$ collinear to $T$, there is $\eta' \in \mathcal{A}$   with support $T'$
 in the equivalence class of $\eta$. This makes the problem of classifying all the nice subsets of $X$ up to collineations crucial for 
 our goal of classifying the  graded contractions of $\Gamma_{\f{g}_2}$.


\subsection{Classification of nice sets up to collineations} \label{sec33}

The results from \S\ref{sec32}   indicate that our classification problem involves determining all the nice subsets of $X$ up to collineations. We begin by introducing certain ``special'' types of subsets of $X$, which, as we will see, happen to be nice.

\begin{define} \label{niceSubsets}
Let $i, j, k \in I$ be generative. 
\begin{enumerate}
\item $X_{L_{ij}}:= \big\{\{i, j\}, \{i, i\ast j\}, \{j, i\ast j\}\big\}$
\item $X_{{L^C_{ij}}}:= \big \{\{i', j'\} \mid i', j' \in I, \, i' \neq j', \, i', j' \notin L_{ij} \big \}$
\item $X_{(i)}:=\setbar{\{i, \ell\}}{\ell \in I,\, \ell \neq i}$
\item $X^{(i)}:=\setbar{\{i', j'\}\in X}{i'\ast j' = i}$ 
\item $T_{\{i, j, k\}}:= P_{\{i, j, k\}} \cup P_{\{i, j, \,i \ast k\}} \cup P_{\{i, k,\,i \ast j\}} \cup P_{\{i,\, i \ast j, \,i \ast k\}}$
\end{enumerate}		
Notice that $T_{\{i,j,k\}}$ is well defined, since all the triplets involved are generating triplets. 
\end{define}
The reader might find the pictures below very helpful to retain these definitions.\smallskip

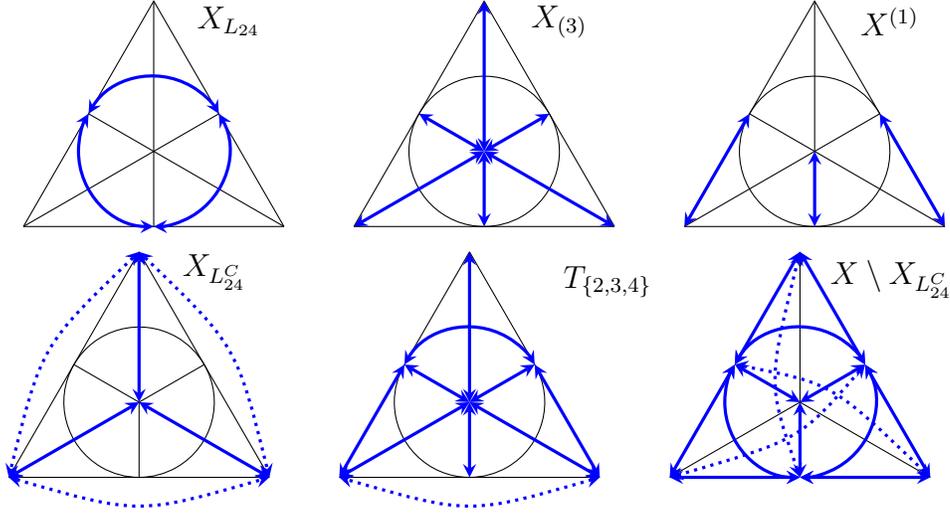
\begin{center}
\begin{tikzpicture}[>=stealth]  
	\draw (60: 2) node {$X_{L_{24}}$};
	\draw (210:2) -- (150:1);
	\draw (150:1) -- (90:2);
	\draw (90:2) -- (30:1);
	\draw (30:1) -- (330:2);
	\draw (330:2) -- (270:1);
	\draw (270:1) -- (210:2);
	
	\draw (210:2) -- (30:1);
	\draw (90:2) -- (270:1);
	\draw (330:2) -- (150:1);
	
	\draw [<->] [blue, very thick] (270:1) arc (270 : 150 :1);
	\draw [<->] [blue, very thick] (150:1) arc (150: 30 : 1);
	\draw [<->] [blue, very thick] (30:1) arc (30 : -90 : 1); 
	
\begin{scope}[xshift=+125pt]  
	\draw (60: 2) node {$X_{(3)}$};
	\draw (210:2) -- (150:1);
	\draw (150:1) -- (90:2);
	\draw (90:2) -- (30:1);
	\draw (30:1) -- (330:2);
	\draw (330:2) -- (270:1);
	\draw (270:1) -- (210:2);
	
	\draw (270:1) arc (270 : 150 :1);
	\draw (150:1) arc (150: 30 : 1);
	\draw (30:1) arc (30 : -90 : 1);
	
	\draw [<->] [blue, very thick] (210:2) -- (0:0);
	\draw [<->] [blue, very thick] (90:2) -- (0:0);
	\draw [<->] [blue, very thick] (330:2) -- (0:0);
	\draw [<->] [blue, very thick] (0:0) -- (30:1);
	\draw [<->] [blue, very thick] (0:0) -- (270:1);
	\draw [<->] [blue, very thick] (0:0) -- (150:1);
\end{scope}
	
\begin{scope}[xshift=+250pt]  
	\draw (60: 2) node {$X^{(1)}$};
	\draw (150:1) -- (90:2);
	\draw (90:2) -- (30:1);
	\draw (330:2) -- (270:1);
	\draw (270:1) -- (210:2);
	
	\draw (210:2) -- (30:1);
	\draw (90:2) -- (0:0);
	\draw  (330:2) -- (150:1);

	\draw (270:1) arc (270 : 150 :1);
	\draw (150:1) arc (150: 30 : 1);
	\draw (30:1) arc (30 : -90 : 1);

	\draw [<->] [blue, very thick] (210:2) -- (150:1);
	\draw [<->] [blue, very thick] (30:1) -- (330:2);
	\draw [<->] [blue, very thick] (0:0) -- (270:1);	
\end{scope}\end{tikzpicture}
\end{center} \ 

\begin{center}\begin{tikzpicture}[>= stealth]
\begin{scope}[  yshift=-1pt]
	\draw (60: 2) node {$X_{L_{24}^C}$};	
	\draw (150:1) -- (90:2);
	\draw (90:2) -- (30:1);
	\draw (30:1) -- (330:2);
	\draw (330:2) -- (270:1);
	\draw (0:0) -- (30:1);
	\draw (150:1) -- (210:2);
	\draw (210:2) -- (270:1);
	\draw (150:1) -- (0:0);
	\draw (270:1) -- (0:0);
	
	\draw (150:1) arc (150: 30 : 1);
	\draw (30:1) arc (30 : -90 : 1);
	\draw(150:1) arc (150: 270: 1);
	
	\draw [<->] [blue, very thick] (90:2) -- (0:0);
	\draw [<->] [blue, very thick] (0:0) -- (330:2);
	\draw [<->] [blue, very thick] (0:0) -- (210:2);
	\draw [<->] [blue, very thick, dotted] (90:2) .. controls (30:1.5) .. (330:2);
	\draw [<->] [blue, very thick, dotted] (90:2) .. controls (150:1.5) .. (210:2) ;
	\draw [<->] [blue, very thick, dotted] (210:2) .. controls (270:1.5) .. (330:2);
\end{scope}

\begin{scope}[xshift=+125pt, yshift=-1pt]  
	\draw (38: 2.6) node  {\hspace{-12pt}$T_{\{2,3,4\}}$};
	\draw (330:2) -- (210:2);

	\draw (90:2) -- (30:1);
	\draw (90:2) -- (150:1);

	\draw (270:1) arc (270 : 150 :1);
	\draw (30:1) arc (30 : -90 : 1);
	\draw [<->] [blue, very thick] (210:2) -- (150:1);
	\draw [<->] [blue, very thick] (30:1) -- (330:2);
	\draw [<->] [blue, very thick, dotted] (330:2) .. controls (270: 1.5) .. (210:2);
	
	\draw [<->] [blue, very thick] (210:2) -- (0:0);
	\draw [<->] [blue, very thick] (0:0) -- (30:1);
	\draw [<->] [blue, very thick] (0:0) -- (270:1);
	\draw [<->] [blue, very thick] (90:2) -- (0:0);
	\draw [<->] [blue, very thick] (0:0) -- (150:1);
	\draw [<->] [blue, very thick] (330:2) -- (0:0);
	\draw [<->] [blue, very thick] (150:1) arc (150: 30 : 1);
\end{scope}

\begin{scope}[xshift=+250pt, yshift=-1pt]  
	\draw (55: 2) node {\ $X\setminus X_{L_{24}^C}$};
	\draw (90:2) -- (0:0);
	\draw (210:2) -- (0:0);
	\draw (330:2) -- (0:0);

	\draw [<->] [blue, very thick] (210:2) -- (150:1);
	\draw [<->] [blue, very thick] (150:1) -- (90:2);
	\draw [<->] [blue, very thick] (90:2) -- (30:1);
	\draw [<->] [blue, very thick] (30:1) -- (330:2);
	\draw [<->] [blue, very thick] (330:2) -- (270:1);
	\draw [<->] [blue, very thick] (270:1) -- (210:2);
	
	\draw [<->] [blue, very thick, dotted] (210:2) .. controls (270:0.5) .. (30:1);
	\draw [<->] [blue, very thick] (0:0) -- (30:1);
	\draw [<->] [blue, very thick] (0:0) -- (270:1);
	\draw [<->] [blue, very thick, dotted] (90:2) .. controls (150:0.5) .. (270:1);
	\draw [<->] [blue, very thick] (0:0) -- (150:1);
	\draw [<->] [blue, very thick, dotted] (330:2) .. controls (30: 0.5) .. (150:1);

	\draw [<->] [blue, very thick] (270:1) arc (270 : 150 :1);
	\draw [<->] [blue, very thick] (150:1) arc (150: 30 : 1);
	\draw [<->] [blue, very thick] (30:1) arc (30 : -90 : 1);
	\end{scope}
\end{tikzpicture}
\end{center}

The following result will be very useful when proving the classification theorem (see Theorem~\ref{classiN}).

\begin{lemma} \label{123Sets}
The following assertions hold for $i, j, k \in I$ generative:
\begin{enumerate}
\item[\rm (i)] $X_{(i)} \sim_c X_{(1)}$;
\item[\rm (ii)] $X^{(i)} \sim_c X^{(1)}$;
\item[\rm (iii)] $X_{L_{ij}} \sim_c X_{L_{12}}$;
\item[\rm (iv)] $X_{L^C_{ij}} \sim_c X_{L^C_{12}}$;
\item[\rm (v)] $X \setminus X_{L^C_{ij}} \sim_c X \setminus X_{L^C_{12}}$;
\item[\rm (vi)] $P_{\{i, j, k\}} \sim_c P_{\{1, 2, 3\}}$;
\item[\rm (vii)] $T_{\{i, j, k\}} \sim_c T_{\{1, 2, 3\}}$.
\end{enumerate}
\end{lemma}

\begin{proof}
Define a collineation $\sigma \in S_\ast(I)$ as 
\[
\sigma(1) = i,\ \sigma(2) = j,\ \sigma(3) = k,\ \sigma(4) = i \ast j \ast k,\ 
\sigma(5) = i \ast j,\ \sigma(6) = i \ast k,\ \sigma(7) = j \ast k.
\]
In each case, $\tilde\sigma$ maps the   set associated to $1, 2, 3$ to the corresponding   set associated to $i, j, k$, which finishes the proof. 
\end{proof}

We next prove that the subsets from Definition~\ref{niceSubsets} are all nice. 

\begin{prop} \label{candidatos}
Let $i, j, k \in I$ be generative. The following subsets of $X$ are nice: 
\begin{enumerate}
\item[\rm (i)] $X$;
\item[\rm (ii)] $X_{L_{ij}}$,
  $X_{L^C_{ij}}$,
  $X^{(i)}$,
  $X_{(i)}$, and any of their subsets;
\item[\rm (iii)] $X \setminus X_{L^C_{ij}}$;
\item[\rm (iv)] $P_{\{i, j, k\}}$;
\item[\rm (v)] $T_{\{i, j, k\}}$.
\end{enumerate}
\end{prop}

\begin{proof}
(i) trivially holds. 

\smallskip

\noindent (ii) For $Y \in\{ X_{L_{ij}},X_{L^C_{ij}},X^{(i)} \}$, there is no generating triplet $\{i', j', k' \}$ satisfying that $\{i', j'\}, \{k', i' \ast j'\} \in Y$, and so $Y$ is nice. The same happens for $X_{(i)}$; in fact, assume that $i', j', k' \in I$ is a generating triplet such that $\{i', j'\}, \{k', i' \ast j'\} \in X_{(i)}$. From $\{i', j'\} \in X_{(i)}$ we get that either $i' = i$ or $j' = i$; suppose, for example, that $i' = i$ (a similar argument works for $j' = i$). Now, from $\{k', i' \ast j'\} \in X_{(i)}$ we get that either $k' = i$ or $i'\ast j' = i$. But both of these cases contradict the fact that $i', j', k'$ are generative, so $X_{(i)}$ is nice. The last assertion in (ii) trivially holds. 

\smallskip

\noindent (iii) Notice that $X \setminus X_{L^C_{ij}} = \big\{\{m, n\}\in X \mid   \mbox{either } m \in L_{ij} \mbox{ or } n \in L_{ij}\big \}$. 
Suppose that $\{i', j', k' \}$ is a generating triplet such that $\{i', j'\}, \{k', i' \ast j'\} \in X \setminus X_{L^C_{ij}}$.
From $\{i', j'\} \in X \setminus X_{L^C_{ij}}$ we obtain that either $i' \in L_{ij}$ or $j'\in L_{ij}$. Assume, for instance, that $i' \in L_{ij}$. This implies that $\{i', k'\}, \{i', j' \ast k'\} \in X \setminus X_{L^C_{ij}}$. It remains to show that $\{j', k'\}$ and $\{j', i'\ast k'\}$ belong to $X \setminus X_{L^C_{ij}}$.
From $\{k', i' \ast j'\} \in X \setminus X_{L^C_{ij}}$, we have that either $k' \in L_{ij}$ or $i' \ast j' \in L_{ij}$. If $k' \in L_{ij}$, then $L_{ij} = L_{i'k'}$, since $i',k'\in L_{ij}$, and so $\{j',i'\ast k'\}\in X \setminus X_{L^C_{ij}}$. Lastly, if $i' \ast j' \in L_{ij}$, since we are assuming that $i' \in L_{ij}$, we obtain that $L_{ij} = L_{i'j'}$. Thus $j' \in L_{ij}$, and so $\{j', k'\}, \{j', i'\ast k'\} \in X \setminus X_{L^C_{ij}}$. This shows that $P_{\{i', j', k'\}} \subseteq X \setminus X_{L^C_{ij}}$, and so $X \setminus X_{L^C_{ij}}$ is nice.

\smallskip

\noindent (iv) We claim that the only generating triplet $\{i', j', k' \}$ with $\{i', j'\}, \{k', i' \ast j'\} \in P_{\{i, j, k\}}$ is precisely $\{i, j, k\}$. In fact, suppose that $\{i', j', k' \}$ is one of such generating triplets; then one of the three cases below occur.
\begin{itemize}
\item[-] $\{i', j'\} = \{i, j\}$, so $i'\ast j' = i\ast j$ and $k'\in\{k,k\ast i, k\ast j,k\ast i\ast j\}$, since $i', j', k' $ are generative. But $\{k', i \ast j\} \in P_{\{i, j, k\}}$ implies $k'= k$.

\item[-] $\{i', j'\} = \{i, k\}$, so $i' \ast j' = i\ast k$ and $k'\in\{j,j\ast i, j\ast k,j\ast i\ast k\}$. Again $\{k', i \ast k\} \in P_{\{i, j, k\}}$ forces $k' = j$.
\item[-] $\{i', j'\} = \{j, k\}$ similarly leads to $\{k', i' \ast j'\} = \{i, j \ast k\}$.
\end{itemize}
To finish, notice that $\{i', j'\} \notin\{ \{k, i \ast j\} , \{j, i \ast k\}, \{i,j \ast k\} \}$, since $i' \ast j'=i\ast j \ast k$ is not one of the components in a pair in $P_{\{i, j, k\}}$.

\smallskip 

\noindent (v) Suppose that $i', j', k'$ are generative  such that $\{i', j'\}, \{k', i' \ast j'\} \in T_{\{i, j, k\}}$. Then 
 $\{i', j', k'\} \in \big \{\{i,j,k\}, \, \{i, j, \,i \ast k\}, \, \{i, k,\,i \ast j\}, \, \{i,\, i \ast j, \,i \ast k\} \big\}$ and (v) follows.
\end{proof}

The above collection of examples contains many of the nice sets, but not all of them. 
Next we study the ``big'' nice sets; big in the sense that they contain $P_{\{1, 2, 3\}}$.

\begin{lemma} \label{claim6}
Let $T$ be a nice subset of $X$ such that $P_{\{1, 2, 3\}} \subsetneq T$. Then:
\begin{enumerate}
\item[\rm (i)] $T$ contains some $P_{\{i, j, k\}}$, for $\{i, j, k\}$ a generating triplet different from  $\{1,2,3\}$. 
\item[\rm (ii)] There exists $\sigma \in S_*(I)$ such that $P_{\{1, 2, 3 \}}\cup P_{\{1, 2, i\}}\subseteq \tilde \sigma(T)$ for either $i = 4$ or $i = 6$.
\end{enumerate}
\end{lemma}

\begin{proof}
(i)  
The proof relies on examining all the possible $\{i_1, i_2\}$ that are in $T$ but not in $P_{\{1, 2, 3\}}$
and proving that in every case, there exists a generating triplet satisfying the required condition. For instance, if $\{i, 5\}\in T$ for $i \in \{4, 6, 7\}$, 
since $\{1, 2\}$ and $\{i, 1 \ast 2\} = \{i, 5\}$ are both in $T$, which is nice, we get that $P_{\{1, 2, i\}} \subseteq T$. 
Similarly, if $\{1, 5\} \in T$, as $\{1, 5\}, \{6, 1 \ast 5\} = \{6, 2\} \in T$, we get that $P_{\{1, 5, 6\}} \subseteq T$; 
and if $\{2, 5\} \in T$, we get that $P_{\{2, 5, 7\}} \subseteq T$. 
The case    $\{i, 6\} $ (respectively, $\{i, 7\} $) belonging to $T\setminus \{1, 2, 3\}$ can be reduced 
to the above considered $\{i, 5\}\in T$, since there exists a collineation $\sigma$ such that 
$\tilde\sigma(\{1, 2, 3\})=\{1, 2, 3\}$ with $\sigma(5)=6$ (respectively, $\sigma(5)=7)$. 
Finally, if $\{i, 4\} \in T\setminus \{1, 2, 3\}$ with $i\ne 5,6,7$, then, if $i=1$, we get $P_{\{1, 2, 6\}} \subseteq T$
and if $i\in\{2, 3\}$, then $P_{\{i, 1,7\}} \subseteq T$.

\smallskip

\noindent (ii) Let $\{i, j, k\}$ be the generating triplet that exists by (i), and $U = \{1, 2, 3\}\cap\{i, j, k\}$. If $U = \emptyset$, then either $P_{\{4, 5, 6\}} \subseteq T$ or $P_{\{4, 6, 7\}} \subseteq T$. In any case, we have that $\{1, 2\},  \{4, 1 \ast 2\} = \{4, 5\} \in T$, and so 
$P_{\{1, 2, 4\}} \subseteq T$. If $|U| = 1$, we can assume without loss of generality that $U = \{1\}$, since there is $\sigma \in S_*(I)$ such that 
$\tilde\sigma(\{1, 2, 3\})=\{1, 2, 3\}$ and $\tilde\sigma(U) = \{1\}$. 
From here we obtain that $\{i, j, k \} \in \big\{\{1, 4, 5\}, \{1, 4, 6\}, \{1, 5, 6\}, \{1, 5, 7\},  
\{1, 6, 7\}\big\}$. If $\{i, j, k\} = \{1, 4, 5\}$, then $\{1, 2\}$ and $\{4, 1 \ast 2\} = \{4, 5\}$ are both in $T$, and so $P_{\{1, 2, 4\}} \subseteq T$.
The other 4 possibilities for $\{i, j, k\} $ similarly lead to either  $P_{\{1, 2, 4\}} \subseteq T$ or $P_{\{1, 2, 6\}} \subseteq T$. Lastly, if $|U| = 2$, then $U \in \big \{ \{1, 2\}, \{1, 3\}, \{2, 3\}\big \}$. We can  assume that $U = \{1, 2\}$ by using a convenient $\sigma \in S_*(I)$ as above. 
Then since $1 \ast 2 = 5$ and $\{i, j, k\}$ is a generating triplet, we have that $k \in \{4, 6, 7\}$. For $k \in \{4, 6\}$ we are done, so let $k = 7$; then $P_{\{1, 2, 3\}} \cup P_{\{1, 2, 6\}} \subseteq \tilde \sigma(T)$ for $\sigma = (1 \quad 2)(6 \quad 7)  \in S_\ast(I)$, since $\sigma$ fixes $P_{\{1, 2, 3\}}$ and  
sends $P_{\{1, 2, 7\}} $ to $ P_{\{1, 2, 6\}}$. 
\end{proof}

\begin{lemma} \label{claim1}
If $T$ is a nice subset of $X$ such that $X \setminus X_{L^C_{12}}\subsetneq T$, then $T = X$.
\end{lemma}
 
\begin{proof}
Let $T$ be nice such that $X \setminus X_{L^C_{12}}\subsetneq T$. We will prove that $X_{L^C_{12}} \subseteq T$. Recall that $X_{L^C_{12}} = \{\{3, 4\}, \{3, 6\}, \{3, 7\}, \{4, 6\}, \{4, 7\}, \{6, 7\}\}$ and 
suppose, for example, that $\{3, 4\} \in T$. Then from $\{1, 5\} \in X \setminus X_{L^C_{12}}$, $3 \ast 4 = 5$, we have that $P_{\{1, 3, 4\}} \subseteq T$. In particular, $\{3, 1 \ast 4\} = \{3, 7\}$ and $\{4, 1 \ast 3\} = \{4, 6\}$ are in $T$. Now, using that $\{4, 6\}$ and $\{7, 4 \ast 6\} = \{7, 2\}$ are both in $T$, we get that $P_{\{4, 6, 7\}} \subseteq T$. Thus, $\{4, 7\}, \{6, 7\} \in T$. It remains to show that $\{3, 6\} \in T$, which follows from the fact that $P_{\{3, 6, 7\}} \subseteq T$ since $\{6, 7\}, \{3, 5\} \in T$. This shows that $T = X$, as desired.
\end{proof} 

\begin{prop} \label{prop_p2}
Let $T$ be a nice subset of $X$ containing $P_{\{1, 2, 3\}}$. Then there exists a collineation $\sigma \in S_\ast(I)$ such that $\, \tilde \sigma(T) \in \big\{X, \, X \setminus X_{L_{12}^C}, \, T_{\{1, 2, 3\}}, \, P_{\{1, 2,3\}}\big\}$.
\end{prop}

\begin{proof}
If either $T = X$ or $T = P_{\{1, 2, 3\}}$, then we are done. Otherwise let $P_{\{1, 2, 3\}}  \subsetneq  T \subsetneq X$. Replace $T$ with  $\tilde \sigma(T)$ for $\sigma \in S_\ast(I)$ chosen as in Lemma~\ref{claim6} (ii). We distinguish two cases:

\smallskip

\noindent Case 1. $P_{\{1, 2, 3\}}\cup  P_{\{1, 2, 4\}} \subseteq  T$. 

\smallskip

We begin by proving that $X \setminus X_{L^C_{12}} =P_{\{1, 2, 3\}}\cup  P_{\{1, 2, 4\}} \cup X_{(5)}\subseteq  T$. To do so, we need to check that $\{i, 5\} \in  T$ for all $i = 1, 2, 6, 7$. Using that $\{1, 4\} \in P_{\{1, 2, 4\}}$ and $ \{3, 5\} \in P_{\{1, 2, 3\}}$ are both in $ T$, which is nice, we get that $P_{\{3, 5, 1\}} \subseteq  T$. From here, we obtain that $\{1, 5\}$ and $\{5, 6\}$ are in $ T$. Using now that $\{1, 5\}$ and $\{7, 2\}\in P_{\{1, 2, 4\}}$ are in $ T$, we get that $P_{\{1, 5, 7\}} \subseteq   T$, and so $\{5, 7\} \in  T$. Lastly, $\{2, 6\} \in P_{\{1, 2, 3\}}\subseteq  T$ and so $P_{\{5, 7, 2\}}\subseteq  T$, which yields $\{2, 5\} \in  T$. This shows that $X \setminus X_{L^C_{12}} \subseteq  T$. If $ T = X \setminus X_{L^C_{12}}$, then we are done; otherwise, $X \setminus X_{L^C_{12}} \subsetneq  T$ and Lemma~\ref{claim1} yields $ T = X$.

\smallskip

\noindent Case 2. $P_{\{1, 2, 3\}}\cup P_{\{1, 2, 6\}} \subseteq  T$. 

\smallskip

We first prove that $T_{\{1, 2, 3\}} \subseteq  T$.  
Using that $\{3, 1 \ast 2\} = \{3, 5\} \in P_{\{1, 2, 3\}}$ and $\{1, 2 \ast 6 \} = \{1, 4\} \in P_{\{1, 2, 6\}}$ are in $ T$, which is nice, the fact that $\{3, 5\}, \{1, 3\ast 5\} \in  T$ gives $P_{\{1, 3, 5\}} \subseteq   T$. In particular, we get that $\{1, 5\} \in   T$, proving that 
$T_{\{1, 2, 3\}}=P_{\{1, 2, 3\}}\cup P_{\{1, 2, 6\}}\cup \{\{1, 5\} \} \subseteq  T$.
 If $ T = T_{\{1, 2, 3\}}$, then there is nothing to prove; otherwise, $T_{\{1, 2, 3\}} \subsetneq  T$, which yields that 
$ T$ contains an element from the set $X \setminus T_{\{1, 2, 3\}} = \Big\{\{2,4 \},\{2,5 \},\{2,7 \},\{3,4 \},\{3,6 \},\{3,7 \},\{4,5 \},\{ 4,6\},\{4,7 \},\{ 5,7\},\{6,7 \}\Big \}$.

\noindent If $\{2, 4\} \in   T$, then from $\{2, 4\}, \{1, 2 \ast 4\} = \{1, 6\}  \in  T$ we get that $P_{\{1, 2, 4\}} \subseteq  T$, and $X \setminus X_{L^C_{12}} \subseteq  T$ by Case 1. Similarly, if $\{3, 4\} \in  T$, then using that $T$ is nice, 
\begin{align*}
P_{\{1, 3, 4\}}\subseteq  & T\Rightarrow \{3, 7\} \in  T\stackrel{\{1, 6\} \in   T }{\Longrightarrow} P_{\{1, 6,7\}} \subseteq  T\Rightarrow \{4,6\} \in   T  
\stackrel{\{3, 5\} \in   T}{\Longrightarrow} P_{\{3, 5, 6\}} \subseteq   T.
\end{align*}
From here, we obtain that $X \setminus X_{L^C_{13}} = T_{\{1, 2, 3\}}\cup P_{\{1, 3, 4\}} \cup P_{\{1, 6,7\}} \cup P_{\{3,5,6\}} \subseteq   T$. Taking now $\alpha \in S_\ast(I)$ such that $\tilde \alpha(X \setminus X_{L^C_{13}}) = X \setminus X_{L_{12}^C}$ as in Lemma~\ref{123Sets},
we get that $X \setminus X_{L_{12}^C} \subseteq \tilde \alpha (T)$; and we can proceed like in the proof of Case 1. Lastly, if any other element of $X \setminus T_{\{1, 2, 3\}}$ 
is in $ T$, using that $ T$ is nice, one can show that either $\{2, 4\} \in   T$ or $\{3, 4\} \in   T$.
\end{proof} 

Now, we investigate the ``smaller'' nice sets. We begin with a trivial but useful result. 

\begin{lemma} \label{rem_vacuous_nice}
Let $T$ be a nice subset of $X$ not containing any $P_{\{i', j', k'\}}$, for $i', j', k' \in I$ generative. If $i, j, k \in I$ satisfy  $\{i, j\}, \{i \ast j, k\}$ are both in $T$, then either $k = i$ or $k = j$.     
\end{lemma} 

\begin{proof}
Let $i, j, k \in I$ such that  $\{i, j\}, \{i \ast j, k\} \in T$. In particular, $i \neq j$, $k \neq i \ast j$. If $i, j, k$ are generative, 
then $P_{\{i, j, k\}} \subseteq T$ since $T$ is nice. But this contradicts our hypothesis, so $i, j, k$ are not generative and $k \in \{i, j, i \ast j\}$.  We finish since $k \neq i \ast j$.
\end{proof}

In the following proof, we use this observation: 
 if $i,j\in I$, $i\ne j$, and $L$ is a line, if either $\{i,j\}\subseteq L$ or $\{i,j\}\subseteq I\setminus L$, then $i\ast j\in L$ (since any two lines intersect); otherwise, $i\ast j\notin L$.

\begin{prop} \label{prop_no_p2}
Let $T$ be a non-empty nice  set not containing $P_{\{i, j, k\}}$ for $i, j, k \in I$ generative.
\begin{enumerate}
\item[\rm (i)] If $X_L \subseteq T$ for some line $L$, then $T = X_L$.
\item[\rm (ii)] If $T \not\subseteq X_L$ for any line $L$ and $T\not\subseteq X_{( i)}$ for any $i  \in I$, then 
\begin{enumerate}
\item[\rm (1)] $\{i, j\} \in T$ implies $\{i\ast j, k\} \notin T$ for all  $k \in I$;
\item[\rm (2)] Either $T\subseteq X_{L^C}$ for some line $L$ or $T= X^{(i)}$ for some $i \in I$. 
\end{enumerate}
\end{enumerate}
In particular, $|T|$ is at most $6$.
\end{prop}

\begin{proof}
(i) Let $i, j \in I$ be such that $i \neq j$ and $X_{L_{ij}} \subseteq T$. Suppose on the contrary that $X_{L_{ij}} \subsetneq T$. Then there exists $\{k, \ell\} \in T$ such that $k \notin L_{ij}=\{i, j, i \ast j\}$. In particular $\{i, j, k\}$ is a generating triplet. We distinguish a few cases:
\begin{itemize}
\item[-] $\ell \in \{i, j\}$. If $\ell = i$, then $\{j, i \ast j \}, \{  k, i\} \in T$ and so $P_{\{k, j, i \ast j\}} \subseteq T$, which is impossible. A similar argument works for $\ell = j$. 
    
\item[-] $\ell = i \ast j$. In this case, from $\{k,  i \ast j\}$ and $\{i, j\} \in T$ we get $P_{\{k, i, j\}} \subseteq T$, a contradiction. 

\item[-] $\ell \notin L_{ij}$. Then $k \ast \ell \in L_{ij}$.  
If $k \ast \ell = i$, then $\{k, \ell\},   \{ j, k \ast \ell\} \in T$ implies $P_{\{j, k, \ell\}} \subseteq T$, a contradiction; $k \ast \ell = j$ leads us to a contradiction, as well. Lastly, if $k \ast \ell = i \ast j$, then $\{k, \ell\}, \{i, k\ast \ell  \} = \{i, i\ast j\} \in T$ implies that $P_{\{i, k, \ell\}} \subseteq T$, a contradiction.
\end{itemize}
In any case, we have reached a contradiction, which implies that $T = X_{L_{ij}}$.

\smallskip

\noindent (ii) Suppose that $T \not\subseteq X_L$ for any line $L$ and $T\not\subseteq X_{(i)}$ for any $i \in I$.

\noindent (1) Let $i, j \in I$ be such that  $\{i, j\} \in T$. Suppose on the contrary that $\{i \ast j, k \} \in T$ for some $k \in I$. Then $k \in \{i, j\}$ by Lemma~\ref{rem_vacuous_nice}. If $\{i \ast j, i\}, \{i \ast j, j\} \in T$, then $X_{L_{ij}} \subseteq T$ and $T = X_{L_{ij}}$ by (1), which contradicts our hypothesis on $T$. Thus, we can assume that $\{i\ast j, i\} \in T$ and $\{i \ast j, j\} \notin T$. From $T \not\subseteq X_{L_{ij}}$, we can find $\{k, \ell\} \in T$ such that  $\ell \notin L_{ij}$. We consider two cases:
\begin{itemize}
\item[-] $k \in L_{ij}$. If $k = i$, then $\{i, \ell\}, \{i, j\}, \{i, i\ast j\} \in T$ and since $T\not\subseteq X_{(i)}$ there are $a, b $ in $I \setminus \{i\}$ with $\{a, b\} \in T$. If $a = j$, then $\{  b, j\} , \{i, i \ast j\}\in T$ implies $P_{\{ b, i, i\ast j\}} \subseteq T$, a contradiction; 
if $a = i \ast j$, then $ \{i, j\}, \{  i \ast j,b\} \in T$ yields $P_{\{  i, j,b\}} \subseteq T$, a contradiction. Hence, $a, b \notin L_{ij}$, which implies $a \ast b \in L_{ij}$. From here we obtain that $P_{\{a, b, m\}} \subseteq T$, for some $m$ in $L_{ij}$ (that depends on $a \ast b$), a contradiction. Lastly, if $k =j$
(respectively, if $k=i \ast j$), then we can easily derive that $P_{\{i,i\ast  j, \ell\}} \subseteq T$ (respectively, $P_{\{i, j, \ell\}} \subseteq T$), a contradiction.  
\item[-] $k \notin L_{ij}$. Then $k \ast \ell \in L_{ij}$, as observed before the proposition. If $k \ast \ell \in \{j,i\ast j\}$, then $P_{\{i, k, \ell\}} \subseteq T$;
and if $k \ast \ell=i$, then $P_{\{j, k, \ell\}} \subseteq T$.
\end{itemize}
In any case, we have reached a contradiction, and so (1) follows. 

\smallskip

\noindent (2) We can always assume that $\{1, 2\} \in T$. From (1) we have that $\{5, k\} \notin T$ for all $k$. 
First, assume we can find  $\ell_1, \ell_2 \ne1,2$ such that   $\{\ell_1, \ell_2\} \in T$. Keeping in mind that $\{1,2,\ell_1\}$ is a generating triplet,  
we can find $\sigma \in S_\ast(I)$ fixing $1$ and $2$ and sending $\ell_1$ onto 3; this allows us to take $\ell_1 = 3$, and so $\ell_2 \in \{4, 6, 7\}$. From (1) we obtain that $3 \ast \ell_2 \neq 1, 2$, which yields $\ell_2 \notin \{6, 7\}$ and so $\ell_2 = 4$. In addition, using that $T$ does not contain any subset of the form $P_{\{i, j, k\}}$, from $\{1, 2\} \in T$ we get $\{4, 7\}, \{3, 6\}, \{4, 6\}, \{3, 7\} \not \in T$; similarly, from $\{3, 4\} \in T $ we obtain $\{1, 6\}, \{2, 7\}, \{2, 6\}, \{1, 7\} \not\in T $. Altogether we are left with:
\[	
\{1, 2\}, \{3, 4\}  \in T\subseteq\{ \{1, 2\}, \{3, 4\}, 
\{1, 3\}, \{1, 4\}, \{2, 3\}, \{2, 4\}, \{6, 7\}\}.
\]
If $\{6, 7\} \notin   T $, then  
 $T \subseteq X_{L^C_{67}}$;
otherwise    $X^{(5)}=\{\{1, 2\}, \{3, 4\}, \{6, 7\}\}\subseteq T$.  From (1) there is no $\{i, j\}\in T$ with  
$i\ast j\in \{1,2,3,4,6,7\}$, and then we get $X^{(5)}=T$.
Second, consider the possibility $T \subseteq X_{(1)}\cup X_{(2)}$. From  $T\not\subseteq X_{(1)}$ and $T\not\subseteq X_{(2)}$, we can find
$\ell_1, \ell_2 \ne1, 2$ with $\{1, 2\}, \{1, \ell_1\}, \{  2,\ell_2\}\in T$. As above, a convenient collineation allows us to take   $\ell_1=3$. 
If $T \not\subseteq X_{L^C_{67}}$, this means that there is $\ell_3\in\{1, 2\}$ such that $\{  7,\ell_3\}\in T$. The only possibility is $\{  1,7\}\in T$, since $2\ast 7=3$. Now $1\ast 7 = 4$ forces $\ell_2$ to be $3$. But $2\ast 3 = 7$, which  contradicts (1). Hence, the possibility $T \subseteq X_{(1)}\cup X_{(2)}$ does not lead to any solution.
\end{proof}

\begin{cor}\label{todosnice}
Any subset $T$ of $X$ different from $P_{\{i, j, k\}}$ is nice with cardinal 
$\le6$ if and only if either $T\subseteq X_{L}$ or $T\subseteq X_{L^c}$ for some line $L$, or  $T\subseteq X_{(\ell)}$ or $T\subseteq X^{(\ell)}$ for some $\ell \in I$.
\end{cor}

\begin{proof}
Apply Propositions~\ref{candidatos} and \ref{prop_no_p2}.
\end{proof}

\noindent We are now in a position to classify the nice subsets of $X$.

\begin{theorem} \label{classiN}
Every  nontrivial nice set $T$ is collinear to one and only one of the following subsets:
\begin{enumerate}
\item[-] if $|T| = 0$, then $T=T_1:=\emptyset$;
\item[-] if $|T| = 1$, then $T \sim_c T_{2}:=\{\{1, 2\}\}$;

\item[-] for $|T| = 2$, there are three possibilities:
\[
 T_{3}:=\{\{1, 2\}, \{1, 3\}\},\quad T_{4}:=\{\{1, 2\}, \{1, 5\}\}, 
\quad T_{5}:=\{\{1, 2\}, \{6, 7\}\};
\]
\item[-] if $|T| = 3$, then $T$ is collinear to one of the following sets: 
\begin{align*}
& T_{6}:=X_{L_{12}}, \quad T_{7}:=X^{(1)}, \quad T_{8}:=\{\{1, 2\}, \{1, 3\}, \{1, 4\}\}, \quad T_{9}:=\{\{1, 2\}, \{1, 3\}, \{1, 5\}\}, 
\\
& 
T_{10}:=\{\{1, 2\}, \{1, 3\}, \{1, 7\}\},
\quad 
T_{11}:=\{\{1, 2\}, \{1, 6\}, \{2, 6\}\}, 
\quad 
T_{12}:=\{\{1, 2\}, \{1, 6\}, \{6, 7\}\};
\end{align*}
\item[-] for $|T| = 4$, we have four possibilities:
\begin{align*}
& T_{13}:=\{\{1, 2\}, \{1, 3\}, \{1, 4\}, \{1, 5\}\}, \  T_{14}:=\{\{1, 2\}, \{1, 3\}, \{1, 5\}, \{1, 6\}\},\   
\\
&T_{15}:=\{\{1, 2\}, \{1, 6\}, \{1, 7\}, \{2, 6\} \},\ 
 T_{16}:=\{\{1, 2\}, \{1, 6\}, \{2, 7\}, \{6, 7\}\};
\end{align*}
\item[-] for $|T|= 5$, we have two options:
\[
T_{17}:=\{\{1, 2\}, \{1, 3\}, \{1, 4\}, \{1, 5\}, \{1, 6\}\}, 
\quad 
T_{18}:=\{\{1, 2\}, \{1, 6\}, \{1, 7\}, \{2, 6\}, \{2, 7\}\};
\]
\item[-] if $|T| = 6$, then $  T $ is collinear to one of the following:
\[
 T_{19}:=X_{L^C_{12}}, \quad T_{20}:=X_{(1)}, \quad T_{21}:=P_{\{1, 2, 3\}};
\]
\item[-] if $|T| = 10$, then $T \sim_c T_{22}:=T_{\{1,2,3\}}$;
\item[-] if $|T| = 15$, then $T \sim_c T_{23}:=X\setminus X_{L^C_{12}}$;
\item[-] if $|T| = 21$, then $T \sim_c T_{24}:=X$.
\end{enumerate}
\end{theorem}

\begin{proof} 
The result trivially holds for $|T| = 1$. 

Suppose that $|T| = 2$ and notice that $T$ does not contain any $P_{\{i, j, k\}}$, for $i, j, k \in I$ generative. If $T \subseteq X_L$ for some line $L$,
 then Lemma~\ref{123Sets} applies to get that $\tilde \sigma(X_L) = X_{L_{12}}$ for some $\sigma \in S_\ast(I)$. Thus $ \tilde \sigma( T)  \subseteq X_{L_{12}}$, and so $T \sim_c \{\{1, 2\}, \{1, 5\}\}$, since $\{\{1, 2\}, \{2, 5\}\}$ becomes $\{\{1, 2\}, \{1, 5\}\}$ via   $(1 \quad 2)(6 \quad 7)$ and $\{\{1, 2\}, \{2, 5\}\}$ becomes $\{\{1, 5\}, \{2, 5\}\}$ via any collineation sending $\{1,2,3\}$ into $\{1,5,i\}$, for any $i\notin{L_{12}}$  (see Remark~\ref{permutaciones}). If $T \subseteq X_{(i)}$ for some $i$, then Lemma~\ref{123Sets} allows us to take $i = 1$. From here we obtain one new possibility for $T$ (up to collineations), namely, $\{\{1, 2\}, \{1, 3\}\}$; since $\{\{1, 2\}, \{1, 3\}\} \sim_c \{\{1, 2\}, \{1, i\}\}$ for $i = 3, 4, 6, 7$, by Remark~\ref{permutaciones} ($\{1,2,i\}$ is a generating triplet). 
Suppose now that $T \not\subseteq X_L$ and $T \not\subseteq X_{(\ell)}$ for any line $L$ and any $\ell \in I$. Then 
$T \subseteq X_{L_{12}^C}$ (up to collineations) by Proposition~\ref{prop_no_p2} and Lemma~\ref{123Sets}. There are two nice sets (up to collineations) contained in $X_{L_{12}^C}$: $\{\{3, 4\}, \{3, 6\}\}$, which is contained in $X_{(3)}$ so nothing new here; and $\{\{3, 4\}, \{6, 7\}\} \sim_c \{\{1, 2\}, \{6, 7\}\}$ via $(1 \quad 3)(2 \quad 4)$. 

\smallskip

Assume that $|T| = 3$; from Proposition~\ref{prop_no_p2} and Lemma~\ref{123Sets} we get that $T$ is either collinear to $X_{L_{12}}$ or $X^{(1)}$, or one of the following holds: 

- $T \subseteq X_{(1)}$. There exists $i$, $j$, $k$ in $I$ such that $T = \{\{1, i\}, \{1, j\}, \{1, k\}\}$. 
If   two of the elements in $T$ belong to   $X_L$ for some line $L$, then we can assume that $i = 2$, $j = 5$. Then $\tilde\sigma(T) = \{\{1, 2\}, \{1, 5\}, \{1, 3\}\}$, for $\sigma$ the collineation sending the generating triplet $\{1, 2 , k\}$ to the generating triplet $\{1, 2, 3\}$. 
Otherwise, $\{1, i, j\}$ is a generating triplet and we may assume it to be $\{1, 2, 3\}$. From here we obtain that $T$ is either $\{\{1, 2\}, \{1, 3\}, \{1, 4\}\}$ or $\{\{1, 2\}, \{1, 3\}, \{1, 7\}\}$ since $k = 5, 6$ leads us to the previous case. 

\smallskip

- $T \subseteq X_{L^C_{12}}$ and $T \not \subseteq X_{(\ell)}$  for all $\ell\in I$.  
Then $T$ is of the form $T = \{\{i_1, j_1\}, \{i_2, j_2\}, \{i_3, j_3\}\}$, for $i_s, j_s \in L^C_{12}= \{3, 4, 6, 7\}$ such that none of the elements of $L^C_{12}$ appears three times in $\{i_1, j_1, i_2, j_2, i_3, j_3\}$. If three of the elements of $L^C_{12}$ (we can take $3$, $4$ and $6$) appear twice each, then $T=\{\{3, 4\}, \{3, 6\}, \{4, 6\}\}  \sim_c \{\{1, 2\}, \{1, 6\}, \{2, 6\}\}$. Otherwise, two elements must appear twice and the remaining other two must appear once, for instance let $3$, $4$ appearing twice; then
$T = \{\{3, 4\}, \{3, 6\}, \{4, 7\}\} \sim_c \{\{1, 2\}, \{1, 6\}, \{6, 7\}\}$ via $(3 \quad 1 \quad 5)(6 \quad 2 \quad 4)$.

\smallskip

Suppose that $|T| = 4$ and let $T = \{\{1, i\}, \{1, j\}, \{1, k\}, \{1, \ell\}\} \subseteq X_{(1)}$. 
 Keeping in mind that $i, j, k, \ell$ are pairwise distinct and 
$i, j, k, \ell \in I\setminus \{1\}= (L_{12} \setminus \{1\}) \cup (L_{13} \setminus \{1\}) \cup (L_{14} \setminus \{1\})$,
we have two possibilities: 

\noindent - $i, j \in L_{1m} \setminus \{1\}$ and $k, \ell \in L_{1n} \setminus \{1\}$, for $m \neq n$.

\smallskip

In this case, $T \sim_c \{\{1, 2\}, \{1, 5\}, \{1, 3\}, \{1, 6\}\}$.

\medskip

\noindent - $i, j \in L_{1m} \setminus \{1\}$, $k \in L_{1n} \setminus \{1\}$ and $\ell \in L_{1p} \setminus \{1\}$, for $m, n, p$ pairwise distinct.

\smallskip

Without loss of generality we can assume $i = 2$, $j = 5$, $k = 3$ and $\ell = 4$ or 7. Notice that $\{\{1, 2\}, \{1, 5\}, \{1, 3\}, \{1, 4\}\} \sim_c \{\{1, 2\}, \{1, 5\}, \{1, 3\}, \{1, 7\}\}$ via the collineation sending 
the generating triplet $\{1, 2, 3\}$ onto the generating triplet $\{1, 5, 3\}$. Lastly, notice that this set is nice by Corollary~\ref{todosnice}.

Assume now that $T \not\subseteq X_{(\ell)}$ for all $\ell \in  I$. Then, up to collineation,  $T \subseteq X_{L_{34}^C}$ by Proposition~\ref{prop_no_p2}. Let $T = \big\{\{i_1, j_1\}, \{i_2, j_2\}, \{i_3, j_3\}, \{i_4, j_4\}\big\}$, where $i_s ,j_s \in L_{34}^C = \{1, 2, 6, 7\}$, for $s = 1, \ldots, 4$. Notice that each element of $L_{34}^C$ appears at most three times in the pairs belonging to $T$. Suppose that, for instance, 1 appears exactly three times, then $\{1, 2\}, \{1, 6\}, \{1, 7\} \in T$ and all the possible options for the fourth element give rise to collinear sets to $\{\{1, 2\}, \{1, 6\}, \{1, 7\}, \{2, 6\}\}$. The remaining case is all the elements of $L_{34}^C$ appearing twice; in such a case, $T \sim_c \{\{1, 2\}, \{1, 6\}, \{2, 7\}, \{6, 7\}\}$.

\smallskip

Assume that $|T| = 5$; if $T \subseteq X_{(\ell)}$ for some $\ell$, then it is straightforward to check that $T  \sim_c \{\{1, 2\}, \{1, 3\}, \{1, 4\}, \{1, 5\}, \{1, 6\}\}$; otherwise  Proposition~\ref{prop_no_p2} applies to get that 
$T = \{\{i_1 ,j_1\}, \{i_2 ,j_2\}, \{i_3, j_3\}, \{i_4, j_4\}, \{i_5, j_5\}\} \subseteq X_{L_{34}^C}$. Then the Pigeonhole principle yields that one of the elements of $L_{34}^C$ appears more than twice in $\{i_1, j_1, \ldots, i_5, j_5\}$; the number of such occurrences is exactly three, since $|L_{34}^C| = 4$. Thus, we can assume $\{1, 2\}, \{1, 6\}, \{1, 7\} \in T$ and $i_4, j_4, i_5, j_5 \in \{2, 6, 7\}$. In any case, we get that 
$T \sim_c \{\{1, 2\}, \{1, 6\}, \{1, 7\}, \{2, 6\}, \{2, 7\}\}$.
\smallskip

Suppose that $|T| = 6$; if $T \subseteq X_{(\ell)}$ for some $\ell$, then $T \sim_c X_{(1)}$. If $T$ contains some $P_{\{i, j, k\}}$ (for $i, j, k$ generative), then $T \sim_c P_{\{1, 2, 3\}}$. Otherwise, $T  = X_{L ^C}$ (for some line $L$) by Proposition~\ref{prop_no_p2}.  

\smallskip

To finish, if $|T| > 6$, then $T$ must contain strictly some $P_{\{i, j, k\}}$ (for $i, j, k$ generative) by Proposition~\ref{prop_no_p2}. From here, Proposition~\ref{prop_p2} allows us to conclude that $T$ is collinear to either $X\setminus X_{L_{12}^C}$, $T_{\{1, 2, 3\}}$ or $T = X$, finishing the proof.
\end{proof}

\begin{conclusion}\label{coro_almenos24}
At the moment, we have found   $24$ nice sets and hence $24$   Lie algebras obtained by graded contractions of $\Gamma_{\f{g}_2}$, these ones obtained as $\la^{\ep_{\eta^{T_i}}}$ for $i=1,\dots,24$.
In the next section we will prove that 23 of them are non-isomorphic, the exception will be the algebras related to $T_8$ and $T_{10}$ which are isomorphic, as checked in Example~\ref{horror}.
We will also prove that the only cases in which there are more than one equivalence class with the same support will be those   with support  collinear to $T_{14}$, $T_{17}$ and $T_{20}$. Fixed any such nice set, there will be
an infinite number of  non-isomorphic algebras
 that have it as their support.
\end{conclusion}

\begin{remark}\label{numeronicesets}
To be more precise, there are   $779$ nice sets.  To compute this number, we have to compute how many nice sets are there
in the orbit $ S_\ast(I)\cdot T_i=\{\tilde\sigma(T_i) \mid \sigma \in S_\ast(I)\}$, for different values of $1\le i\le24$. 
Recall that, if $S_\ast(I)_{T_i}:=\{\sigma \in S_\ast(I)\mid \tilde\sigma(T_i)=T_i\}$ denotes the subgroup of collineations which
leave $T_i$ invariant, its cardinal is related with the cardinal of the orbit by $|S_\ast(I)\cdot T_i|=\frac{168}{|S_\ast(I)_{T_i}|}$.
Now we compute these cardinals case by case:\smallskip

\begin{center}
\begin{tabular}{r||c| c| c |c| c| c||} 
 \hline
 $i$ &1,\,24 &2,\,4,\,5,\,14,\,16,\,22&3,\,9,\,12,\,13,\,15&6,\,7,\,19,\,20,\,23&8,\,10,\,11,\,21&17,\,18 \\ [0.5ex] 
 \hline 
  $|S_\ast(I)_{T_i}|$ &168 &8&2&24&6&4 \\ [0.5ex] 
 \hline
 $|S_\ast(I)\cdot T_i|$ &1 &21&84&7&28&42 \\ [0.5ex] 
 \hline
\end{tabular}
\end{center}\smallskip

\noindent For instance, the orbit of $T_1$ is $S_\ast(I)\cdot T_1=\{\emptyset\}$, which  contains only one nice set. 
The orbit of $T_2=\{\{1,2\}\}$ has 21 elements, namely, $S_\ast(I)\cdot T_2=\{\{t\}\mid t\in X\}$.
For  $T_3=\{\{1,2\},\{1,3\}\}$, if a collineation $\sigma$ satisfies $\tilde\sigma(T_3)=T_3$, then $\sigma(1)=1$ and $\sigma(\{2,3\})=\{2,3\}$. Besides the identity map, there is only one such collineation, so that the subgroup of collineations fixing $T_3$ has 2 elements and the orbit of $T_3$ has 84 elements. Look at $T_4=\{\{1,2\},\{1,5\}\}$. A collineation $\sigma$ leaving $T_4$ invariant is determined by $\sigma(1)=1$, $\sigma(2)\in\{2,5\}$,
$\sigma(3)\in\{3,4,6,7\}$, so that there are 8 elements in the stabilizer. There are also  8 collineations leaving $T_5=\{\{1,2\},\{6,7\}\}$ invariant:
$\sigma(1)\in\{1,2,6,7\}$, this forces $\sigma(2)=2,1,7,6$ respectively, and the possibilities for $\sigma(6)$ are two ($6/7$ in each of the first two cases, and $1/2$ in the   other two). The orbit of $T_6=X_{L_{12}}$ has 7 elements, since there are 7 lines. Similarly, $S_\ast(I)\cdot T_7=\{X^{(i)}\mid i\in I\}$ has cardinal 7 too, just like $I$. Now, the stabilizer of $T_8=\{\{1,2\},\{1,3\},\{1,4\}\}$ has $6$ elements, since a collineation such that $\tilde\sigma(T_8)=T_8$
satisfies $\sigma(1)=1$ and is determined by $\sigma(2)$ and $\sigma(3)$ arbitrary distinct elements in $\{2,3,4\}$, since $4=1*2*3$. Also, there are two possibilities for  $\sigma $ leaving $T_9=\{\{1,2\},\{1,3\},\{1,5\}\}$ invariant, since necessarily $\sigma(1)=1$, $\sigma(3)=3$ and $\sigma(2)\in\{2,5\}$. We can proceed similarly for the remaining values $i\ge10$.   Thus the total number of nice sets is the sum of the cardinal of the orbits, $1\cdot 2+6\cdot 21+( 7+84)\cdot 5+28\cdot 4+42\cdot 2=779$. Thus we have 
  779 Lie algebras $\{\la^{\ep_{\eta^{\tilde\sigma(T_i)}}}\mid\sigma\in S_\ast(I),i\le24\}$ which are not graded-isomorphic, since they have different support. Of course, this   is not  relevant for   classifying graded contractions up to equivalence,  which is our main objective.
\end{remark}

\section{Classification of graded contractions of $\f{g}_2$} \label{secjunto}

Next, we explore how many non-isomorphic Lie algebras  can be obtained by graded contractions of $\Gamma_{\f{g}_2}$    with a fixed support, in Sections~\ref{sec4} and \ref{sec5}. 
For most of the nice sets, there is only one isomorphism class of Lie algebras attached. This assertion can be concluded only from the study of the equivalence classes via normalization, which will be the first aim in \S\ref{sec4}. We will need    more specific arguments in \S\ref{sec5} for dealing with several nice sets contained in some $X_{(i)}$.
With some extra work, this will give the classification of the graded contractions of $\Gamma_{\f{g}_2}$ up to equivalence in \S\ref{sec_revision}.

\subsection{Equivalent graded contractions via normalization} \label{sec4}

Recall that a first step towards the classification of all the equivalence classes $\tilde{\mathcal G}/\sim$ of graded contractions consists in classifying all the equivalence classes $\tilde{\mathcal G}/\sim_n$ of graded contractions 
via normalization; this turns out to be equivalent to describing the equivalence classes $\mathcal G/ \sim_n$ of 
admissible graded contractions via normalization, since the sets $\tilde{\mathcal G}/\sim$ and ${\mathcal G}/\sim$ are bijective (by Lemma~\ref{ref_existeadmisible}) and $\sim_n$ trivially restricts to $\mathcal G$ (that is, $\ep^\alpha$ is an admissible graded contraction provided $\ep$ is so). 

On the other hand, Proposition~\ref{bijection} allows us to work in the set $\mathcal A$ in   \eqref{def_A}, by defining $\eta \sim_n \eta'$ if $\ep_\eta \sim_n \ep_{\eta'}$ for any $\eta, \eta'\in \mathcal A$. Note that $\eta \sim_n \eta'$ if there exists $\alpha \colon  I \to \mathbb{C}^\times$ (written as $\alpha(i) = \alpha_i$) such that $\eta'=\eta^\alpha$, 
where $\eta_{ij} = \eta(\{i, j\})$ for all $\{i, j\} \in X$, and
\begin{equation}\label{eq_alphaij}
\eta^\alpha_{ij} := \eta_{ij} \alpha_{ij}, \qquad   
\alpha_{ij} := \frac{\alpha_i \alpha_j}{\alpha_{i \ast j}}.
\end{equation} 

Our goal here is to determine the equivalence classes in ${\mathcal A/}\sim_n$.
Given $\eta \in \mathcal A$ with support $T = \{\{i_1, j_1\}, \ldots, \{i_s, j_s\}\}$, lexicographically ordered, that is, $i_k < j_k$,
$i_1 \leq i_2 \leq \ldots \leq i_s$, and if $i_k=i_{k+1}$ then $j_k<j_{k+1}$;
to ease the notation, we write $\eta = (\eta_{i_1j_1}, \ldots, \eta_{i_sj_s})$. 
For instance, for $T = \{\{1, 2\}, \{1, 6\}, \{2, 7\}, \{6, 7\}\}$, we write $\eta = (\eta_{12}, \eta_{16}, \eta_{27}, \eta_{67})$. If $\eta_{i_kj_k} = 1$ for all $k$, then we write $\bm{1}^T = (1, 1, \ldots^{(s}, 1)$.

\begin{theorem} \label{prop_equiv}
 Any   $\eta \in\mathcal A$ with nontrivial support $T$ from Theorem~\ref{classiN} is equivalent via normalization to  $\bm{1}^T$ except in the following three cases: 
\begin{itemize}
 \setlength\itemsep{0.5em}
\item [\rm (i)] If $T =T_{14}= \{\{1, 2\}, \{1, 3\}, \{1, 5\}, \{1, 6\}\}$, then $\eta\sim_n  (1, 1, 1, \lambda) $, for $\lambda =  {\frac{\eta_{13}\eta_{16}}{\eta_{12}\eta_{15}}}$. \item[] Moreover, $(1, 1, 1, \lambda) \sim_n (1, 1, 1, \lambda')$ if and only if $\lambda= \lambda'$.

\item[\rm (ii)] If $T = T_{17}=\{\{1, 2\}, \{1, 3\}, \{1, 4\}, \{1, 5\}, \{1, 6\}\}$, then
$\eta\sim_n  (1, \lambda, 1, 1, \lambda)$,  for $\lambda^2 =  \frac{\eta_{13}\eta_{16}}{\eta_{12}\eta_{15}}$. 

\item[] Moreover, $(1, \lambda, 1, 1, \lambda)\sim_n(1, \lambda', 1, 1, \lambda')$ if and only if $\lambda = \pm\lambda'$.

\item[\rm (iii)] If $T =T_{20}= X_{(1)}$, then $\eta\sim_n  
(1, \lambda , \mu, 1, \lambda, \mu)$, for $\lambda^2 =  \frac{\eta_{13}\eta_{16}}{\eta_{12}\eta_{15}}$ and $\mu^2=\frac{\eta_{14}\eta_{17}}{\eta_{12}\eta_{15}}$. 

\item[] Moreover, $(1, \lambda , \mu, 1, \lambda, \mu)\sim_n(1, \lambda' , \mu', 1, \lambda', \mu')$ if and only if $\lambda = \pm\lambda'$, $\mu = \pm\mu'$.
\end{itemize}
\end{theorem}

 For the proof,  it is convenient to establish notation, since any nonzero complex number admits two \emph{square roots}. 
In order to choose one of them, for any $\alpha\in\bb C^\times$,
we may uniquely express $\alpha=\vert \alpha\vert e^{\mathfrak{i}\theta_{\alpha}},$ for some $\theta_{\alpha}\in [0,2\pi),$ 
and then we
denote by $\sqrt{\alpha}:= \sqrt{\vert \alpha\vert}e^{\mathfrak{i}\theta_{\alpha}/2}$.
  Here we use $\mathfrak{i}$ for the imaginary unit in the underlying field of complex numbers, to distinguish it from $\bf i\in\mathcal O$,
used through the manuscript for an octonion.
Note that we do not have the property that $\sqrt{\alpha\alpha'}=\sqrt{\alpha}\sqrt{\alpha'}$ for any $\alpha,\alpha'\in\bb C^\times$.

\begin{proof}
First, note that for any $T$ with $T_{14}\subseteq T\subseteq X_{(1)}$, any $|T|$-tuple in $(\mathbb C^\times)^{|T|}$ does provide 
a  map 
in $\mathcal A$ (an admissible graded contraction, with a minor  abuse of language)
with support $T$; because, as there is no generating triplet  $\{i, j, k\}$ such that $\{i, j\}, \{i\ast j,k\}\in T$,  then the condition $\eta_{ijk}=\eta_{jki}$ necessary to assure $\eta\in\mathcal A$   is satisfied trivially, since $\eta_{ijk}=0$ for any generating triplet.
 This implies that all the tuples used in (i), (ii) and (iii) really provide  equivalence classes up to normalization related to those supports. 
 
\smallskip

 (i) Suppose that $T = \{\{1, 2\}, \{1, 3\}, \{1, 5\}, \{1, 6\}\}$. Taking $\alpha_1 = \frac{1}{\sqrt{\eta_{12}}\sqrt{\eta_{15}}}$, $\alpha_2 = \sqrt{\eta_{15}}$, $\alpha_3 = \frac{\sqrt{\eta_{12}}\sqrt{\eta_{15}}}{\eta_{13}}$, $\alpha_5 = \sqrt{\eta_{12}}$, and $\alpha_4 = \alpha_6 = \alpha_7 = 1$, we obtain that $\eta^\alpha = (1, 1, 1, \lambda)$, where $\lambda =  {\frac{\eta_{13}\eta_{16}}{\eta_{12}\eta_{15}}}$; which shows that $\eta \sim_n (1, 1, 1, \lambda)$.

\smallskip

Now, if $(1, 1, 1, \lambda')$ is another admissible graded contraction with support $T$ satisfying that $(1,  1, 1, \lambda') \sim_n (1,  1, 1, \lambda)$, then there exists a map $\beta\colon I \to \mathbb{C}^\times$ such that   $\beta_{12} =\beta_{13} =\beta_{15} =1$ and $\lambda' \beta_{16}=\lambda$
(notation as in \eqref{eq_alphaij}).
From here we get $1 =  \beta_{12}\beta_{15}=(\beta_1)^2=\beta_{13}\beta_{16}=\beta_{16}$ so that $\lambda=\lambda'$, concluding the proof of (i).

\medskip

\noindent (ii)  Suppose that $T = \{\{1, 2\}, \{1, 3\}, \{1, 4\}, \{1, 5\}, \{1, 6\}\}$. Then taking $\alpha_1 = \alpha_7 = \frac{1}{\sqrt{\eta_{12}}\sqrt{\eta_{15}}}$, $\alpha_2 = \sqrt{\eta_{15}}$, $\alpha_3 = \frac{\sqrt{\eta_{16}}}{\sqrt{\eta_{13}}}$, $\alpha_4 = \frac{1}{\eta_{14}}$, $\alpha_5 = \sqrt{\eta_{12}}$ and $  \alpha_6 = 1$, we obtain that $\eta^\alpha = (1, \lambda, 1, 1, \lambda)$, for $\lambda = \frac{\sqrt{\eta_{13}}\sqrt{\eta_{16}}}{\sqrt{\eta_{12}}\sqrt{\eta_{15}}}$; which means that $\eta \sim_n (1, \lambda, 1, 1, \lambda)$.

\smallskip

Now, if $(1, \lambda', 1, 1, \lambda')$ is another admissible graded contraction with support $T$ 
such that   $(1, \lambda', 1, 1, \lambda') \sim_n (1, \lambda, 1, 1, \lambda)$, then there exists a map $\beta\colon I \to \mathbb{C}^\times$ satisfying that $(1, \lambda', 1, 1, \lambda')^\beta = (1, \lambda, 1, 1, \lambda)$. 
This means $\beta_{12} =\beta_{14} =\beta_{15} =1$, $\lambda' \beta_{13} = \lambda' \beta_{16} =\lambda$.
From here we get $(\beta_1)^2 =  \beta_{12}\beta_{15} = 1$, which implies $\beta_1 = \pm 1$, and so 
\begin{align*} 
\left(\frac{\beta_6}{\beta_3}\right)^2 =\frac{\beta_{16}}{\beta_{13}}= 1\  \Rightarrow\ 
 \frac{\lambda}{\lambda'} =\beta_{16}=\beta_1 \frac{\beta_6}{\beta_3} = \pm 1\ \Rightarrow\ 
 \lambda' = \pm \lambda.
\end{align*}
In order to finish
 the proof of (ii) we only need to find $\beta\colon I \to \mathbb{C}^\times$ 
such that
$\beta_{12} =\beta_{14} =\beta_{15} =1$, $ \beta_{13} =  \beta_{16}=-1$, so that $(1, \lambda, 1, 1, \lambda)^\beta=(1, -\lambda, 1, 1, -\lambda)$. 
For instance, $\beta=(1,1,-1,1,1,1,1)$ is such a map.

\medskip 

\noindent  (iii) Let $T = X_{(1)}$, $\lambda = \frac{\sqrt{\eta_{13}}\sqrt{\eta_{16}}}{\sqrt{\eta_{12}}\sqrt{\eta_{15}}}$ and $\mu = \frac{\sqrt{\eta_{14}}\sqrt{\eta_{17}}}{\sqrt{
\eta_{12}}\sqrt{\eta_{15}}}$. Notice that $\eta \sim_n (1, \lambda, \mu, 1, \lambda, \mu)$; in fact, take $\alpha_1 = \frac{1}{\sqrt{\eta_{12}}\sqrt{\eta_{15}}}$, $\alpha_2 = \sqrt{\eta_{15}}$, $\alpha_3 = \frac{\sqrt{\eta_{16}}}{\sqrt{\eta_{13}}}$, $\alpha_4 = \frac{\sqrt{\eta_{17}}}{\sqrt{\eta_{14}}}$, $\alpha_5 = \sqrt{\eta_{12}}$ and $ \alpha_6 = \alpha_7 = 1$.

\smallskip

Next, suppose that $(1, \lambda', \mu', 1, \lambda', \mu')$ is an admissible graded contraction with support $T$ such that $(1, \lambda', \mu', 1, \lambda', \mu') \sim_n (1, \lambda, \mu, 1, \lambda, \mu)$. Then there exists a map $\beta\colon I \to \mathbb{C}^\times$ satisfying that $(1, \lambda', \mu', 1, \lambda', \mu')^\beta = (1, \lambda, \mu, 1, \lambda, \mu)$. This means
$$
\beta_{12}=1=\beta_{15},\quad \beta_{13}=\frac{ \lambda}{ \lambda'}=\beta_{16},\quad \beta_{14}=\frac{\mu}{\mu'}=\beta_{17}.
$$
As above $(\beta_1)^2 =  \beta_{12}\beta_{15}=1$ and $\left(\frac{\beta_6}{\beta_3}\right)^2 = 1$, so that 
$ \frac{\lambda}{\lambda'} = \beta_1\frac{\beta_6}{\beta_3} = \pm 1 $.
Similarly, 
\begin{align*}
\left(\frac{\beta_7}{\beta_4}\right)^2 =\frac{\beta_{17}}{\beta_{14}}= 1 \Rightarrow \frac{ \mu }{ \mu '} =\beta_{17}=\beta_1 \frac{\beta_7}{\beta_4} = \pm 1 \Rightarrow \mu' = \pm \mu. 
\end{align*}
To finish the proof of (iii), note that $ (1, \lambda, \mu, 1, \lambda, \mu)^\delta= (1, -\lambda, \mu, 1, -\lambda, \mu)$
and $ (1, \lambda, \mu, 1, \lambda, \mu)^\gamma= (1, \lambda, -\mu, 1, \lambda, -\mu)$ for
$\delta=(1,1,-1,1,1,1,1)$ and $\gamma=(1,1,1,-1,1,1,1)$.

\smallskip

It remains to show that $\eta\sim_n  \bm{1}^T $, for the remaining $T$ in Theorem~\ref{classiN}; to do so it is enough to find a map $\alpha\colon  I \to \mathbb{C}^\times$ such that $\eta^\alpha = \bm{1}^T$, or equivalently, $\eta_{ij} = \frac{\alpha_{i \ast j}}{\alpha_i \alpha_j}$ for all $\{i, j\}\in T$.

\smallskip

\noindent - $T = \{\{1, 2\}\}$: let $\alpha_5 = \eta_{12}$, and $\alpha_i = 1$ for all $i \ne5$. 

\smallskip

\noindent -  $T = \{\{1, 2\}, \{1, 5\}\}$: let $\alpha_1 = \alpha_2 = \frac{1}{\sqrt{\eta_{12}}\sqrt{\eta_{15}}}$, $\alpha_5 = \frac{1}{\eta_{15}}$ and $\alpha_i = 1$ for all $i \ne1,2,5$.

\smallskip

\noindent - $T = \{\{1, 2\}, \{1, 3\}\}$: let $\alpha_5 = \eta_{12}$, $\alpha_6 = \eta_{13}$ and $\alpha_i = 1$ for all $i\ne 5, 6$.

\smallskip

\noindent - $T = \{\{1, 2\}, \{6, 7\}\}$: let $\alpha_1 = \frac{1}{\eta_{12}}$, $\alpha_6 = \frac{1}{\eta_{67}}$ and $\alpha_i = 1$ for all $i \ne1,6$. 

\smallskip

\noindent -  $T = X_{L_{12}}$: let $\alpha_1 = \frac{1}{\sqrt{\eta_{12}}\sqrt{\eta_{15}}}$, $\alpha_2 = \frac{1}{\sqrt{\eta_{12}}\sqrt{\eta_{25}}}$, $\alpha_5 = \frac{1}{\sqrt{\eta_{15}}\sqrt{\eta_{25}}}$ and $\alpha_i = 1$ for all $i \ne1,2,5$.

\smallskip

\noindent - $T = X^{(1)}$: let $\alpha_2 = \frac{1}{\eta_{25}}$, $\alpha_3 = \frac{1}{\eta_{36}}$, $\alpha_4 = \frac{1}{\eta_{47}}$ and $\alpha_1 = \alpha_5 = \alpha_6 = \alpha_7 = 1$.

\smallskip

\noindent - $T = \{\{1, 2\}, \{1, 3\}, \{1, 4\}\}$: let  
$\alpha_5 = \eta_{12}$, $\alpha_6 = \eta_{13}$, $\alpha_7 = \eta_{14}$ and $\alpha_1 = \alpha_2 = \alpha_3 = \alpha_4 = 1$.

\smallskip

\noindent -  $T = \{\{1, 2\}, \{1, 3\}, \{1, 5\}\}$: let $\alpha_1 = \frac{1}{\sqrt{\eta_{12}}\sqrt{\eta_{15}}}$, $\alpha_2 = \sqrt{\eta_{15}}$, $\alpha_3 = \sqrt{\eta_{12}}\sqrt{\eta_{15}}$, $\alpha_5 = \sqrt{\eta_{12}}$,

\quad \quad \quad \quad \quad \quad \quad \quad \quad \quad \, \, \quad $\alpha_6 =  {\eta_{13}}$ and 
$\alpha_4 = \alpha_7 = 1$. 

\smallskip

\noindent - $T = \{\{1, 2\}, \{1, 3\}, \{1, 7\}\}$: let $\alpha_5 = \eta_{12}$, $\alpha_6 = \eta_{13}$, $\alpha_4 = \eta_{17}$ and $\alpha_i = 1$ for all $i\ne 4,5,6$.

\smallskip

\noindent - $T = \{\{1, 2\}, \{1, 6\}, \{2, 6\}\}$: let $\alpha_3 = \eta_{16}$, $\alpha_4 = \eta_{26}$, $\alpha_5 = \eta_{12}$, and $\alpha_i = 1$ for all $i\ne 3,4,5$.

\smallskip

\noindent - $T = \{\{1, 2\}, \{1, 6\}, \{6, 7\}\}$: let $\alpha_2 = \frac{1}{\eta_{12}}$, $\alpha_3 = \eta_{16}$, 
$\alpha_7 = \frac{1}{\eta_{67}}$ and $\alpha_i = 1$ for all $i\ne 2,3,7$.

\smallskip

\noindent -  $T = \{\{1, 2\}, \{1, 3\}, \{1, 4\}, \{1, 5\}\}$: let
$\alpha_1 = \frac{1}{\sqrt{\eta_{12}}\sqrt{\eta_{15}}}$, $\alpha_2 = \sqrt{\eta_{15}}$, $\alpha_3 = \alpha_4 = \sqrt{\eta_{12}}\sqrt{\eta_{15}}$,

\quad \quad \quad \quad \quad \quad \quad \quad \quad \quad \quad \quad \quad \quad \, $\alpha_5 = \sqrt{\eta_{12}}$, $\alpha_6 = \eta_{13}$,  and $\alpha_7 = \eta_{14}$.

\smallskip

\noindent  - $T = \{\{1, 2\}, \{1, 6\}, \{1, 7\}, \{2, 6\}\}$: let $\alpha_3 = \eta_{16}$, $\alpha_4 = \eta_{26}$, $\alpha_5 = \eta_{12}$, $\alpha_7 = \frac{\eta_{26}}{\eta_{17}}$ and 

\quad \quad \quad \quad \quad \quad \quad \quad \quad \quad \quad \quad \quad \quad \, $\alpha_1 = \alpha_2= \alpha_6 = 1$.

\smallskip 
 
\noindent  -  $T = \{\{1, 2\}, \{1, 6\}, \{2, 7\}, \{6, 7\}\}$: let $\alpha_1 = \frac{\sqrt{\eta_{67}}}{\sqrt{\eta_{12}}\sqrt{\eta_{16}}\sqrt{\eta_{27}}}$, $\alpha_2 = \alpha_3 = \eta_{16}$, $\alpha_5 = \frac{\sqrt{\eta_{12}}\sqrt{\eta_{16}}\sqrt{\eta_{67}}}{\sqrt{\eta_{27}}}$,

\smallskip

\quad \quad \quad \quad \quad \quad \quad \quad \quad \quad \quad \quad \quad \quad \,  $\alpha_6 = \frac{\sqrt{\eta_{12}}\sqrt{\eta_{16}}\sqrt{\eta_{27}}}{\sqrt{\eta_{67}}}$, $\alpha_7 = \frac1{\eta_{27}}$ and $\alpha_4 = 1$.

\smallskip

\noindent - $T = \{\{1, 2\}, \{1, 6\}, \{1, 7\}, \{2, 6\}, \{2, 7\}\}$: let $\alpha_1 = \sqrt{\eta_{26}}\sqrt{\eta_{27}}$, 
$\alpha_2 = \sqrt{\eta_{16}}\sqrt{\eta_{17}}$, $\alpha_6 = 1$,

\quad \quad \quad \quad \quad \quad \quad \quad \quad \quad \quad   \quad \quad \quad \quad   $\alpha_3 = \eta_{16}\sqrt{\eta_{26}}\sqrt{\eta_{27}}$, $\alpha_4 = \eta_{26}\sqrt{\eta_{16}}\sqrt{\eta_{17}}$, $\alpha_7 = \frac{\sqrt{\eta_{16}}\sqrt{\eta_{26}}}{\sqrt{\eta_{17}}
\sqrt{\eta_{27}}}$,

\quad \quad \quad \quad \quad \quad \quad \quad \quad \quad \quad \quad \quad   \quad   \, \, $\alpha_5 = \eta_{12}\sqrt{\eta_{16}}\sqrt{\eta_{17}}
\sqrt{\eta_{26}}\sqrt{\eta_{27}}$.

\smallskip

\noindent -  $T = X_{L_{12}^C}$: let $\alpha_1 = \frac{1}{\sqrt{\eta_{34}}\sqrt{\eta_{37}}\sqrt{\eta_{46}}\sqrt{\eta_{67}}}$, $\alpha_2 = \frac{1}{\sqrt{\eta_{34}}\sqrt{\eta_{36}}
\sqrt{\eta_{47}}\sqrt{\eta_{67}}}$, $\alpha_3 = \frac{1}{\sqrt{\eta_{34}}\sqrt{\eta_{36}}\sqrt{\eta_{37}}}$, 

\quad \quad \quad   \, \,  $\alpha_4 = \frac{1}{\sqrt{\eta_{34}}\sqrt{\eta_{46}}\sqrt{\eta_{47}}}$, $\alpha_5 = \frac{1}{\sqrt{\eta_{36}}\sqrt{\eta_{37}}\sqrt{\eta_{46}}
\sqrt{\eta_{47}}}$, $\alpha_6 = \frac{1}{\sqrt{\eta_{36}}\sqrt{\eta_{46}}\sqrt{\eta_{67}}}$, $\alpha_7 = \frac{1}{\sqrt{\eta_{37}}\sqrt{\eta_{47}}\sqrt{\eta_{67}}}$.

\medskip

\noindent - $T = P_{\{1, 2, 3\}}$: let $\alpha_1 = \alpha_4 = \alpha_6 = \frac{1}{\eta_{12}}$, $\alpha_2 = \alpha_5 = \frac{1}{\eta_{26}}$, $\alpha_3 = \frac{1}{\eta_{13}}$, $\alpha_7 = \frac{1}{\eta_{17}}$. 
Note that  $\alpha_{23}\eta_{23}=1$ and $\alpha_{35}\eta_{35}=1$ follows from $\eta_{132}=\eta_{321}=\eta_{213}$, 
since we are assuming $\eta$  belongs to $\mathcal A$.

\medskip

\noindent -  $T = T_{\{1, 2, 3\}}$: Take $\beta_1 = \frac{1}{\sqrt{\eta_{12}}\sqrt{\eta_{15}}}$, $\beta_2 = \sqrt{\eta_{15}}$, $\beta_3 = \frac{1}{\sqrt{\eta_{13}}}$,
$\beta_4 = \frac{1}{\sqrt{\eta_{14}}}$,
$\beta_5 = \sqrt{\eta_{12}}$, $\beta_6 = \frac1{\sqrt{\eta_{16}}}$, $\beta_7 = \frac1{\sqrt{\eta_{17}}}$,
to get that $\eta^\beta=\eta'= (1, \, \lambda_1, \, \lambda_2, \, 1, \, \lambda_1, \, \lambda_2, \, \lambda_3, \, \lambda_4, \lambda_5, \, \lambda_6)$, where 
 $\lambda_1 = \frac{\sqrt{\eta_{13}}\sqrt{\eta_{16}}}{\sqrt{\eta_{12}}\sqrt{\eta_{15}}}$, \,
$\lambda_2 = \frac{\sqrt{\eta_{14}}\sqrt{\eta_{17}}}{\sqrt{\eta_{12}}\sqrt{\eta_{15}}}$, \, 
$\lambda_3 = \frac{\sqrt{\eta_{15}}\sqrt{\eta_{17}}}{\sqrt{\eta_{13}}}\eta_{23}$, \, 
$\lambda_4 = \frac{\sqrt{\eta_{15}}\sqrt{\eta_{14}}}{\sqrt{\eta_{16}}}\eta_{26}$, \, 
$\lambda_5 = \frac{\sqrt{\eta_{12}}\sqrt{\eta_{14}}}{\sqrt{\eta_{13}}}\eta_{35}$ and 
$\lambda_6 = \frac{\sqrt{\eta_{12}}\sqrt{\eta_{17}}}{\sqrt{\eta_{16}}}\eta_{56}$.
 Now,
using that $\eta' \in\mathcal A$  we derive
\begin{align*} 
\eta'_{123} & = \eta'_{312} = \eta'_{231} \Leftrightarrow 
\lambda_2 \lambda_3 = \lambda_5 = \lambda_4 \lambda_1 , 
\\
\eta'_{126} & = \eta'_{612} = \eta'_{261}
\Leftrightarrow 
\lambda_2\lambda_4  = \lambda_6= \lambda_3 \lambda_1,
\\
\eta'_{135} & = \eta'_{513} = \eta'_{351}
\Leftrightarrow 
\lambda_2\lambda_5 = \lambda_6 \lambda_1= \lambda_3 ,
\\
\eta'_{156} & = \eta'_{615} = \eta'_{561}
\Leftrightarrow
\lambda_2\lambda_6 =   \lambda_4 = \lambda_5 \lambda_1.
\end{align*}
From here we obtain that   
\begin{align*}
\lambda_5 & = \lambda_1\lambda_4 = (\lambda_1)^2 \lambda_5 \Rightarrow \lambda_1 = \pm 1,   
\\  
\lambda_6 & = \lambda_2\lambda_4 = (\lambda_2)^2 \lambda_6 \Rightarrow \lambda_2 = \pm 1.     
\end{align*}
Hence, $\eta'= (1, \, \lambda_1, \, \lambda_2, \, 1, \, \lambda_1, \, \lambda_2, \, \mu, \, \lambda_1\lambda_2\mu, \lambda_2\mu, \, \lambda_1\mu)$
for some $\mu\in\mathbb{C}^\times$,  $\lambda_1,\lambda_2 \in\{\pm 1\}$.
Defining $\gamma\colon  I \to \mathbb{C}^\times$ by    $\gamma_3 = \gamma_6 = \frac1{\mu}$ and $\gamma_i=1$ for all $i\ne3,6$, we get that 
$$(\eta')^\gamma =  
(1, \, \lambda_1, \, \lambda_2, \, 1, \, \lambda_1, \, \lambda_2, \, 1, \, \lambda_1\lambda_2 , \lambda_2, \, \lambda_1 ).
$$
But $(1, \, -1, \, 1, \, 1, \, -1, \, 1, \, 1, \, -1 , 1, \, -1 )^\delta= \bm{1}^T$
and $(1, \, 1, \, -1, \, 1, \, 1, \, -1, \, 1, \, -1 , -1, \, 1 )^{\delta'}= \bm{1}^T$,
for $\delta=(1,1,-1,-1,1,1,-1)$ and 
 $\delta'= (1,1,1,-1,1,1,1) $.

\smallskip

\smallskip  

- Let $T = X\setminus X_{L^C_{12}}$. Since $X_{(1)} \subseteq T$, we can take $\beta\colon I\to \bb{C}^\times$ the same map as in case   $T_{\{1,2,3\}}$ to get 
$\eta^\beta=\eta'=(1, \, \lambda_1, \, \lambda_2, \, 1, \, \lambda_1, \, \lambda_2, \, \lambda_3, \ldots, \lambda_{11})$, for some $\lambda_i \in \mathbb{C}^\times$.
On the other hand, using that $\eta'_{ijk} = \eta'_{jki}$ for all $i, j, k$ generative, we obtain that 
$\lambda_1, \lambda_2\in\{\pm1\}$ and
$
\eta'=(1,  \lambda_1, \lambda_2, 1, \lambda_1, \lambda_2, \, \lambda, \, \mu, \, \lambda\mu\lambda_1\lambda_2, \, \lambda\lambda_1\lambda_2, \, \mu\lambda_1\lambda_2, \, \lambda\lambda_2, \, \mu\lambda_1, \, \lambda\lambda_1, \,\mu\lambda_2)$, for some  $\lambda, \mu \in \mathbb{C}^\times$. To finish, take
 $\gamma_1 = \lambda_1  \lambda_2$, $\gamma_2 =  \frac{1}{\sqrt{\lambda_1}\sqrt{\lambda_2}\sqrt{\lambda}\sqrt{\mu}}$, 
$\gamma_3 = \frac{1}{\sqrt{\lambda_1}\sqrt{\lambda}}$, $\gamma_4   = \frac{1}{\sqrt{\lambda_2}\sqrt{\mu}}$,
$\gamma_5 = \lambda_1 \lambda_2 \gamma_2$,
$\gamma_6 = \lambda_2  \gamma_3$,
$\gamma_7 = \lambda_1  \gamma_4$ to get that $(\eta')^\gamma = \bm{1}^T$. \end{proof}

This finishes the equivalence classes via normalization attached to the possible non-trivial supports up to collineation. 
Lastly, in case $T = X$ it is also true that any $\eta \in\mathcal A$ with   support $X$ is equivalent via normalization to  $\bm{1}^X$.
The proof is similar to the one in case $X\setminus X_{L^C}$; it can also be seen as a consequence of 
\cite[Theorem~3.1]{ref33}, which deals with graded contractions \emph{without zeroes} and states (in our notation) the following:

\emph{``If $\ep$ is a complex $G$-graded contraction without zeroes, then there exist some nonzero complex numbers  $\{\alpha_g: g\in G\}$ such that $\ep(g, h) = \frac{\alpha_g\alpha_h}{\alpha_{g + h}}$''.}

\noindent Here in \cite{ref33}, $G$ denotes an arbitrary finite abelian group and there are no restrictions on the $G$-graded Lie algebra.

\begin{conclusion}\label{conclus}
The set $\mathcal G/\sim_n$ consists of  one isolated equivalence class  related to each of the $21$ nice sets in $\{T_i:i\ne 14,17,20\}$ (and to those ones collinear to them); 
together with three infinite families related to $T_{14}$, $T_{17}$, and $T_{20}$,   parametrized by $\bb{C}^\times$, $\bb{C}^\times/\bb{Z}_2$ and   $(\bb{C}^\times)^2/\bb{Z}^2_2$, respectively (and those ones collinear to them). 

  (More precisely, as in Remark~\ref{numeronicesets}, we have 70 parametrized families 
  jointly with 674 isolated equivalence classes by normalization,
  whose related Lie algebras can be obtained from $\la^{\ep_{\eta^{T_i}}}$ by applying the Weyl group of the grading as in Proposition~\ref{colli}.)
\end{conclusion}

The next aim is to prove if the corresponding Lie algebras are non-isomorphic. 
At the moment we know that there is no an isomorphism between two of these algebras which is a scalar multiple of the identity on each homogeneous component.


\subsection{Classification up to strong equivalence}\label{sec5}


Our goal here is to prove that any two strongly equivalent admissible graded contractions of $\Gamma_{\f{g}_2}$  
are also equivalent by normalization. 
That is, we have to prove that, 
  if $\varphi\colon  \la^{\ep_\eta} \to \la^{\ep_{\eta'}}$ denotes a graded isomorphism, then there exists a graded isomorphism 
  $\varphi'\colon  \la^{\ep_\eta} \to \la^{\ep_{\eta'}}$ such that $\varphi'\vert_{\la_i}$ is a scalar multiple of the identity for all $i\in I$. (We are using the notation $\la_i\equiv (\f{g}_2)_{g_i}$ 
  as in Lemma~\ref{unaporsoporte}.)
 This is not a trivial problem by any means, and it seems to rely heavily on the properties of the grading. 
 
  There is no precedent in   dealing with this problem,  
   so we will try to explain where our ideas for addressing it come from.
First, our results in the previous sections allow us to restrict our attention to admissible graded contractions with support contained in $X_{(1)}$, since we proved in Lemma~\ref{unaporsoporte} that two strongly equivalent admissible graded contractions have the same support and,  
for the remaining (non-collinear) supports, Theorem~\ref{prop_equiv} tells us that there is only one equivalence class up to normalization, so that in particular only one class up to strong equivalence. 
 The difficulty in dealing with a nice set $T \subseteq X_{(1)}$  lies in the fact that we have no much information 
 about the nonzero values of an admissible map $\eta\in\mathcal A$ with support $T$, since  any map $\eta\colon X\to \bb C^\times$ with support $T$ belongs to $\mathcal A$ (there is no $\{i, j\},\{i\ast j,k\}\in T$ with $\{i, j,k\}$ a generating triplet, so $\eta_{ijk}=0$). We will obtain valuable  information on the values of 
 $\eta$ in Corollary~\ref{co_secdificil}; the main tool being thinking of  $\varphi\vert_{\la_i}$ as an endomorphism of a  2-dimensional vector space 
  and to take   advantage of our knowledge of   the products among the subspaces $\la_i$'s as in Lemma~\ref{lem_base}. We begin by  adapting the notation used in
the basis \eqref{eq_para5} in order to   handle  several  basis of the same homogeneous component simultaneously. 
(The notation in \eqref{eq_para5}, less precise but much simpler, has been used through the remaining sections of this paper.)

\begin{remark} \label{re_base}
 We denote our basis of the space of zero trace octonions as 
\[
e_1 = {\bf i},  \quad 
e_2 = {\bf j},  \quad 
e_3 = {\bf l},  \quad
e_4 = {\bf kl}, \quad 
e_5 = {\bf k},  \quad 
e_6 = {\bf il}, \quad 
e_7 = {\bf -jl}.
\]
Then  $e_ie_j = e_{i\ast j}$ if either the ordered line $(i, j, i\ast j)$ or some of its cyclic permutations belong to the set
${\bf L} = \{(1, 2, 5), \, (5, 6, 7),\, (7, 4, 1), \, (1, 3, 6), \, (6, 4, 2),\ (2, 7, 3), \, (3, 4, 5)\}$, 
and  $e_ie_j = -e_{i\ast j}$, otherwise. (Note we have used brackets 
instead of braces because the order in the lines is relevant for describing the signs of the products.
Also, we use here $\ell,\ell'\dots$ for ordered lines instead of indices, due to the necessity of adopting a  very precise notation.)

Take $\ell \in {\bf L}$ and fix $i\in\ell$, $k \notin \ell$. If $ j\in \ell\setminus\{i\}$, then $i,j,k$ are generative, $\mathcal{Q}:= \langle 1,e_i,e_j,e_{i\ast j}\rangle$ is a quaternion subalgebra (isomorphic to $\mathcal{H}$) and  
$e_k$ is orthogonal to $\mathcal{Q}$ with respect to the norm $n$. So $\mathcal{O}=\mathcal{Q}\oplus \mathcal{Q}e_k $.
Consider the derivations of $\mathcal{O}$ given by, for any $q\in \mathcal{Q}$,
\begin{equation}\label{eq_para5'}
\begin{array}{lll}
E^{\ell,k}_{i }: &q\mapsto 0, & qe_{k}\mapsto \frac12 (e_{i}q)e_{k},   \\
F^\ell_{i }: &q\mapsto \frac12[ e_{i},q],&qe_{k}\mapsto-\frac12 (qe_{i})e_{k}. 
\end{array}
\end{equation}
The definition of $F^\ell_{i } $ does not depend on $k$ since $F^\ell_{i } = \frac14D_{e_{j}, e_ie_j}$ for both $j\in \ell\setminus\{i\}$.
For the other derivation, the choice of $k\notin\ell$ is not very relevant either, since
$$
E^{\ell,k}_{i}=E^{\ell,i\ast k}_{i}=-E^{\ell,j \ast k}_{i}=-E^{\ell,i\ast j\ast k}_{i}.
$$
The set $B^{\ell,k}_{i}:=\{E^{\ell,k}_{i}, F^\ell_{i}\}$ is a basis of $\la_{i}  $ and 
each homogeneous component $\la_{i}  $  has six of such bases since each index belongs exactly to three different lines in ${\bf L}$ and there are two possible signs for  \lq\lq$E$\rq\rq. 
Now, as in   \eqref{eq_para5}, for any $r,r'\in\ell$, the elements in the basis multiply as follows,
\[
[E^{\ell,k}_{i}, E^{\ell,k}_{j}] = E^{\ell,k}_{i\ast j}, \quad 
[F^\ell_{i}, F^\ell_{j}] = F^\ell_{i\ast j}, \quad 
[E^{\ell,k}_{r}, F^\ell_{r'}] = 0, 
\]
if $\ell$ is any cyclic permutation of $ (i, j, i\ast j )\in {\bf L}$.
As a consequence, for any $a,b\in\bb C$, and any $i\ne j\in\ell$,
\begin{equation}\label{eqSpec}
\Spec(\ad^2(aE^{\ell,k}_{i}+bF^{\ell}_{i})\vert_{\la_{ j}}) = \{-a^2, -b^2\},
\end{equation} 
 where   $E^{\ell,k}_{j}$ and $F^{\ell}_{j}$ are   eigenvectors related to $-a^2$ and $-b^2$, respectively.
 Here, $\Spec$ refers to the spectrum of an endomorphism, that is, the set of eigenvalues.
\end{remark}

\begin{prop}\label{le_matrices}
Let $T$ be a nice set and $\varphi\colon \la^{\ep } \to \la^{\ep_{\eta^T}}$ a graded isomorphism, with $\eta^T\in\mathcal A$  defined in Equation~\eqref{defietaT}. 
If $\big \{\{i, j\}, \{i, i\ast j\} \big \}\subseteq T$, then the following assertions hold: 
\smallskip

\noindent {\rm (i)} For any $z\in\la_i$,
$\Spec(\ad^2\varphi(z)\vert_{\la_{ j}}) = \{ \ep_{ij}\ep_{i\,i\ast j} \lambda\mid \lambda\in \Spec(\ad^2 z\vert_{\la_{ j}}) \}.$
\smallskip

\noindent {\rm (ii)} The matrix of $\varphi\vert_{\la_i}$ with respect to the basis $B^{\ell,k}_i$, 
for $i,j\in\ell $ and $k\notin\ell$, is one of the following
\begin{equation*}\label{eq_matrices}
 \pm\alpha\begin{pmatrix}
1  & 0  \\
0  & 1 
\end{pmatrix},\quad
\pm \alpha\begin{pmatrix}
1  & 0  \\
0  & -1 
\end{pmatrix},\quad
\pm\alpha\begin{pmatrix}
0  & 1  \\
1  & 0 
\end{pmatrix},\quad
\pm\alpha\begin{pmatrix}
0  & 1  \\
-1  & 0 
\end{pmatrix},
\end{equation*}
for $\alpha^2 =  {\ep_{ij}\ep_{i \, i\ast j}}$.

\smallskip 

\noindent {\rm (iii)} $\det(\varphi\vert_{\la_i})=\pm\ep_{ij}\ep_{i \, i\ast j}$. 
\end{prop}
 
\begin{proof}
 Recall that   $\ep_{\eta^T}(i,j)=\ep_{\eta^T}(i,i\ast j)=1$. 

\noindent (i) For   $z \in \la_i$ and  $w\in \la_j$ we have $[\varphi(z),[\varphi(z),\varphi(w)]]=\varphi([z,[z,w]^{\ep } ]^{\ep } )=\ep_{ij}\ep_{i\,i\ast j}\varphi([z,[z,w]])$, which implies  
$\ad^2(\varphi(z))\vert_{\la_j}= \ep_{ij}\ep_{i\,i\ast j}   \varphi \circ\ad^2z\circ  \varphi^{-1}\vert_{\la_j}$.

\smallskip

\noindent (ii) and (iii). Let 
$P  = \tiny
{\begin{pmatrix}
a  & b  \\
c  & d 
\end{pmatrix}}$ be the matrix of $\varphi\vert_{\la_i}$ with respect to the basis $B^{\ell,k}_i$ given in Remark~\ref{re_base}.
 To ease the notation, write $E = E^{\ell,k}_i$, $F = F^\ell_i$; notice that  $\varphi(E) = aE+ cF$  and  $\varphi(F) = bE + dF$.
From (i) we get that 
$\Spec(\ad^2\varphi(E)\vert_{\la_{ j}}) = \{ -\ep_{ij}\ep_{i\,i\ast j} ,0 \}$, since $\Spec(\ad^2 E\vert_{\la_{ j}}) = \{ -1 ,0 \}$. 
As $\Spec(\ad^2\varphi(E)\vert_{\la_{ j}}) =\{-a^2, -c^2\}$
by Equation~\eqref{eqSpec},
 thus, either $a = 0$ and $c^2=\ep_{ij}\ep_{i\,i\ast j}$, or $c = 0$ and $a^2=\ep_{ij}\ep_{i\,i\ast j}$. 
 Arguing with $F$, we get in a similar way that either $b = 0$ and $d^2=\ep_{ij}\ep_{i\,i\ast j}$, or $d = 0$ and $b^2=\ep_{ij}\ep_{i\,i\ast j}$. 
  Therefore, 
  \begin{itemize}
  \item either $P  = 
\tiny\begin{pmatrix}
a  & 0  \\
0  & d 
\end{pmatrix}$ with $a=\pm d$ and $\det(P)=ad\in\{\pm\ep_{ij}\ep_{i\,i\ast j}\}$;
\item or $P  = 
\tiny\begin{pmatrix}
0  & b  \\
c  & 0 
\end{pmatrix}$ with $b=\pm c$  and $\det(P)=-bc\in\{\pm\ep_{ij}\ep_{i\,i\ast j}\}$.
\end{itemize}
 \end{proof}

\begin{cor}\label{co_secdificil}
Let $T$ be a nice set such that $\big \{\{i, j\}, \{i, i\ast j\}, \{i, k\}, \{i, i\ast k\} \big \} \subseteq T$. If
$\eta \in \mathcal{A}$ is strongly equivalent to $\eta^T$, then 
$\eta_{ij}\eta_{i, i\ast j} = \pm \eta_{ik}\eta_{i, i\ast k}$.   
\end{cor}

\begin{proof}
Let $\varphi\colon \la^{\ep_\eta } \to \la^{\ep_{\eta^T}}$ be the corresponding graded isomorphism. Applying the previous proposition to the two lines 
$\ell  $ and $\ell'  $ in ${\bf L}$ which are reorderings of  $\{i,j,i*j\}$ and $\{i,k,i*k\}$ respectively, we get $\det(\varphi\vert_{\la_i})=\pm\eta_{ij}\eta_{i \, i\ast j}=\pm\eta_{ik}\eta_{i \, i\ast k}$, and the result follows.
\end{proof}

At the moment it is not yet immediate whether there could be an admissible map    $\eta \in \mathcal{A}$ satisfying $\eta\approx\eta^T$ but $\eta\not \sim_n\eta^T$, but 
a lot of information on the possibilities for $T$ and $\eta$ can be extracted from Corollary~\ref{co_secdificil}.  
Of course we can assume that $\{\{1, 2\}, \{1, 3\},  \{1, 5\}, \{1, 6\}\}\subseteq T\subseteq X_{(1)}$, by an application of Theorem~\ref{prop_equiv}. Let us prove that:
\begin{itemize}
\item[-] If $ T= \{\{1, 2\}, \{1, 3\},  \{1, 5\}, \{1, 6\}\}$, then $\eta  \sim_n (1,1,1,-1)$.
\item[-] If $T = X_{(1)}\setminus  \{\{1, 7\}\}$, then $\eta  \sim_n(1,\mathfrak{i},1,1, \mathfrak{i})$, for   $\mathfrak{i}\in\bb C$.
\item[-] If $T = X_{(1)}$, then $\eta  \sim_n (1,\lambda,\mu,1, \lambda,\mu)$, for $\lambda, \mu\in\{\pm1,\pm  \mathfrak{i}\}$ no both in $\{\pm1\}$.
\end{itemize}
For the first case, we can assume that $\eta_{12} = \eta_{13} = \eta_{15} = 1$ by Theorem~\ref{prop_equiv}.
Now, $\eta_{16}=\eta_{16}\eta_{13}=\pm\eta_{12}\eta_{15}=\pm1$, due to Corollary~\ref{co_secdificil}. But $\eta_{16} \ne1$ since $\eta\not\sim_n \eta^T$.
 If $T = X_{(1)}\setminus  \{\{1, 7\}\}$, then $\eta\sim_n (1,\lambda,1,1, \lambda)$, for some $\lambda \in \bb C$ such that $\lambda^2=\eta_{16}\eta_{13}=\pm\eta_{12}\eta_{15}=\pm1$. Now  
$\lambda^2\ne1$, since $\eta\not\sim_n\eta^T$; the result clearly follows since $(1,\mathfrak{i},1,1, \mathfrak{i})\sim_n (1,-\mathfrak{i},1,1,- \mathfrak{i})$. The last case follows similarly, since $\lambda^2=\pm1$ and $\mu^2=\pm1$. Now the question   is whether these situations can really occur: is it $ (1,1,1,-1)$
strongly equivalent to $ \bm{1}^T$? It will be easy to give a negative answer once we know the relation among the 3 different basis of the same homogeneous component.

\begin{lemma}\label{le_cambiodebase}
Let $\ell = (1, 2, 5)$, $\ell' = (1, 3, 6)$ and $B_1^{\ell,3}$, $B_1^{\ell',2}$ be the bases of $\la_1$ defined in Equation~\eqref{eq_para5'}. 
Denote  $E=E_1^{\ell,3}$, $F=F_1^{\ell}$, $E'=E_1^{\ell',2}$, $F'=F_1^{\ell'}$. Then
 $$ E'=\frac12(E+F),\qquad  F'=\frac12(3E-F).$$
\end{lemma}

\begin{proof}
Straightforward computations permit to check that for any derivation $d\in\{E,F,E',F'\}$, then $d(e_i)=\alpha_{d,i}e_{\sigma(i)}$, for the permutation
$\sigma=(2\quad 5)(3\quad 6)(4\quad 7)$ and the scalar $\alpha_{d,i}$ given by the next table:
$$
\begin{array}{c|ccccccc|}
d/i& 1& 2& 3& 4& 5& 6& 7  \\\hline
E&0&0&1&1&0&-1&-1\\
F&0&2&-1&1&-2&1&-1\\
E'&0&1&0&1&-1&0&-1\\
F'&0&-1&2&1&1&-2&-1 \\\hline
\end{array}
$$
\end{proof}

\begin{theorem}\label{teo_normalizacionbastaba}
Strong equivalence  and equivalence via normalization coincide for $\Gamma_{\f{g}_2}$. 
\end{theorem}

\begin{proof}

The result will follow by proving that $ \bm{1}^T$ is not strongly equivalent to  $(1, 1, 1, -1)$, $(1, \mathfrak{i}, 1, 1, \mathfrak{i})$ and $(1, \lambda, \mu, 1, \lambda,\mu)$; notice that there is no ambiguity in the support ($T=T_{14},T_{17},T_{20}$ respectively).  We prove here for 
  $T = \{\{1, 2\}, \{1, 3\},  \{1, 5\}, \{1, 6\}\}$. Suppose, on the contrary, that $\varphi\colon \la^{ \ep } \to \la^{\ep'}$ is a graded isomorphism for $\eta^\ep = (1, 1, 1, -1)$ and $\eta^{\ep'} = \bm{1}^T$.
   Applying Proposition~\ref{le_matrices}~(ii) twice with $\ell =(1, 2, 5)$ and $\ell' = (1, 3, 6)$, we obtain that the 
matrices $P$, $P'$ of the endomorphism $\varphi\vert_{\la_1}$ relative to the bases $ B_1^{\ell,3}$ and $ B_1^{\ell',2}$, respectively, must be 
$\pm\alpha  P_s$ and $\pm\alpha'  P_r$ for some $s,r\in\{0,1,2,3\}$, where
  $$
  P_0=\begin{pmatrix}
1  & 0  \\
0  & 1 
\end{pmatrix},\quad
P_1 =\begin{pmatrix}
1  & 0  \\
0  & -1 
\end{pmatrix},\quad
P_2=\begin{pmatrix}
0  & 1  \\
1  & 0 
\end{pmatrix},\quad
P_3=\begin{pmatrix}
0  & 1  \\
-1  & 0 
\end{pmatrix},
 $$
with
 $\alpha^2= \ep_{12}\ep_{15}=1$ and $(\alpha')^2= \ep_{13}\ep_{16}=-1$.
 Notice that $
 Q=\frac12\begin{pmatrix}
1  & 3  \\
1  & -1 
\end{pmatrix}$ is the order 2 matrix of the change of bases, by Lemma~\ref{le_cambiodebase}.
Now observe that no scalar multiple of $QP_iQ$ belongs to $\{P_0,P_1,P_2,P_3\}$ if $i\ne0$.
This forces $r=s=0$ and then $\alpha^2 =\det (\varphi\vert_{\la_1})= (\alpha')^2$, a contradiction. 
\end{proof}



\subsection{Classification up to equivalence}\label{sec_revision}
  We would like to take advantage of all the above information to solve the problem of how many classes of Lie algebras can be obtained up to equivalence (in the sense of Definition~\ref{defequivnorm1}).
For now, we can be sure that there are at the most 21 classes along with 3 infinite families. 
In fact, if $\eta\in\mathcal{A}$, then there are $\sigma\in S_\ast(I)$ and $i\le 24$ such that $\tilde\sigma(S^\eta)=T_i$
(Proposition~\ref{suppisnice}  and Theorem~\ref{classiN}). Now Lemma~\ref{le_relacionsoportes}
says that the support of $\eta'=\eta^{\sigma^{-1}}$ is just $T_i$, and $\eta'\sim \eta$ by Proposition~\ref{colli}.
If $i\ne 14,17,20$, we know that $\eta'\sim_n\eta^{T_i}$ by  Theorem~\ref{prop_equiv}, and then $\eta\sim\eta^{T_i}$.
Similarly for $i=14,17,20$, $\eta$ is equivalent to one of the graded contractions exhibited in Theorem~\ref{prop_equiv}. 
But all this gives an upper bound for the number of equivalence classes.
A priori it could happen that $\eta^{T_i}\sim\eta^{T_j}$ for $i\ne j$.
Even that, for a fixed $i=14,17,20$, two admissible maps with support $T_i$ could be equivalent although not strongly equivalent. We have to 
discuss
carefully all these possibilities for getting the classification up to equivalence.  
 
Our next goal will be to prove that equivalent graded contractions will have collinear supports with only one   exception. This is not an easy task and we need  some preparation: 
 first we need to observe the relation between support of an admissible graded contraction and center of the related Lie algebra;
 and second, we will find some very convenient
   collection of isomorphisms of some of the algebras obtained by graded contractions of $\Gamma_{\f{g}_2}$. 
\smallskip


The support gives immediate information about the center of the algebra. Here   $\f{z}(\f{g}) = \{x \in \f{g} \colon [x,\f{g}] = 0\}$ denotes the center of a Lie algebra $\f{g}$. 

\begin{prop} \label{centre}
Let $\ep\colon G\times G\to \mathbb C$ be an admissible graded contraction.
The center $\f{z}(\la^{{\ep}})$ is the direct sum of the homogeneous components $\la_i = (\f{g}_2)_{g_i}$ such that  
$i$ does not appear in any of the elements of $T=S^{\eta^\ep}$. In other words, $\f{z}(\la^{\ep}) = \bigoplus\limits_{i\in I_T}\la_i$, where 
$I_T = \setbar{i\in I}{i \notin t, \, \, \forall \, t \in T}$. In particular, $\dim \f{z}(\la^{\ep }) = 2|I_T|$.  
\end{prop}

\begin{proof}
Notice that if $i \in I_T$, then $\{i, j \}\notin T$ for all $j \ne i$. Thus $\eta^\ep_{ij} = 0$ and so $[\la_i,\la_j]^\ep = 0$ for all $j \neq i$, which implies that $\left[\la_i,\sum_{j\ne i}\la_j\right]^\ep = 0$. On the other hand, $[\la_i, \la_i]^\ep = 0$, since $\la_i$ is abelian. This shows that $\la_i \subseteq \f{z}(\la^\ep)$.

Conversely, take $z \in \f{z}(\la^\ep)$ and write $z = \sum_{i\in I} z_i$, where $z_i \in \la_i$ for   $i \in I$. We claim that $\tilde z = z - \sum_{i\in I_T} z_i = \sum_{s \notin I_T} z_s \in \f{z}(\la^\ep)$ is zero. In fact, if $\tilde z \neq 0$, then $z_i \ne 0$ for some $i \notin I_T$; then there exists $j \in I$ such that $\{i, j\} \in T$, so $\eta^\ep_{ij}\ne 0$. On the other hand, $\mathrm{ad} \, z_i\vert_{\la_{j} } \colon \la_{j} \to \la_{i\ast j}$ is surjective by Lemma~\ref{lem_base} (iii), so there exists $w\in \la_{j}$ such that $0 \ne [z_i, w]\in\la_{i\ast j}$. Using that $\tilde z \in \f{z}(\la^\ep)$, we obtain that 
\[
0 = [\tilde z,w]^{\ep} = \sum_{s\notin I_T} [z_s, w]^{\ep} = \sum_{s\notin I_T}\eta^\ep_{sj} [z_s, w],
\] 
where $[z_s, w] \in \la_{s\ast j}$. This implies that $\eta^\ep_{sj} [z_s, w] = 0$ for all $s \notin I_T$, since the map $I\to I$, $s\mapsto s\ast j$, is injective. 
In particular, $\eta^\ep_{ij} [z_i, w] = 0$; but this is impossible since $\eta^\ep_{ij}\ne 0$ and
$[z_i, w] \ne 0$. Thus $\tilde z = 0$ and so $z \in \bigoplus\limits_{i\in I_T}\la_i$. 
\end{proof}

Now we will look for suitable auxiliar linear maps.
We begin by observing, for   the Lie algebra $  \mathfrak{so}(3,\bb C)$ of skew-symmetric matrices with basis $\{x_1, x_2, x_3\}$ as in 
Example~\ref{ex_noeslomismo}, that the map   given by $x_1\mapsto x_1$, $x_i\mapsto  -x_j$ and $x_j\mapsto x_i$ is an automorphism, for $(i,j)=(2,3)$ or $(3,2)$.
This gives, for any $i,j\in I$ distinct, an automorphism of $\varphi_{ij}$ of the semisimple Lie algebra $\la_i\oplus\la_{j}\oplus\la_{i*j}\le\la=\f{g}_2$ (two copies of $  \mathfrak{so}(3,\bb C)$):
\begin{equation*}\label{fi_ij}
\varphi_{ij}\vert_{\la_{i*j}}=\id,\qquad  
\varphi_{ij}(x_i)=-x_j,\quad  
\varphi_{ij}(y_i)=-y_j,\quad  
\varphi_{ij}(x_j)=x_i, \quad  
\varphi_{ij}(y_j)=y_i, 
\end{equation*}
  where now $x_i, y_i, x_j, y_j$ are as in \eqref{eq_para5}. (Both are compatible notations.)
This can be extended to the bijective linear map 
\begin{equation}\label{titaij}
 \theta_{ij}:=
\begin{cases} 
	 \theta_{ij}\vert_{\la_t}= \varphi_{ij}, & \text{if  $t\in \{i,j,i*j\},$} \\
	 \theta_{ij}\vert_{\la_t}= \id, & \text{otherwise.}
\end{cases}
\end{equation}
 As in Remark~\ref{re_nodeltotal}, $ \theta_{ij}$ is not an automorphism of $\la$, but it is an automorphism of some of the Lie algebras obtained by graded contractions from $\la$. 
 For instance, take $T=X^{(i*j)}$, that is, $T=\{\{ r,s \},\{ i,j \},\{ k,l \}  \}\subseteq X$ with different indices such that $i*j=k*l=r*s$ and note
 that $\theta_{ij}\in\Aut(\la^{\ep })$, for $ \ep=\Phi^{-1}(\eta^T)$. 
 Simply note that $\theta_{ij}$ coincides with $\varphi_{ij}$ in $\la_i\oplus\la_{j}\oplus\la_{i*j}$, and with the identity map
 in $\la_r\oplus\la_{s}\oplus\la_{r*s}$ and in $\la_k\oplus\la_{l}\oplus\la_{k*l}$. The same arguments can be used to check 
 how $\ep$ has to be  for $\theta_{ij}$ to be an automorphism:

\begin{lemma}\label{le_automo}
If $\ep$ is an admissible graded contraction of $\Gamma_{\f{g}_2}$, and $T=S^{\eta^\ep}$ is the support, then, $\theta_{ij}$ is an automorphism of $\la^\ep$ if and only if:
\begin{itemize}
\item[\rm(i)] For any $\{t_1,t_2\}\in T$, either $t_1,t_2,t_1*t_2\in I\setminus\{i,j\}$ or $t_1,t_2,t_1*t_2\in\{i,j,i*j\}$;
\item[\rm(ii)] $\ep_{it}=\ep_{jt}$ for all $t\ne i,j$.
\end{itemize}
\end{lemma}
This is consistent with the above: $\theta_{ij}\notin\Aut(\la)$  since $T=X$ does not satisfies (i). However, $\theta_{ij}\in\Aut(\la^\ep)$ 
if $S^{\eta^\ep}=X^{(i*j)}  $.  Another relevant example is $T\subseteq X_{(i*j)}$, where (i) is always satisfied, although (ii) could be false:  
for instance,  
for $\theta_{ij}\in\Aut(\la^\ep)$ it is necessary the condition  that $\{i,i*j\}\in T$ if and only if $\{j,i*j\}\in T$.

\begin{example}\label{horror}
Let $\{i, j, k\}$ a generating triplet.
Take $T=\{\{i,j\},\{i,k\},\{i,j*k\}\}$ and $T'=\{\{i,j\},\{i,k\},\{i,i*j*k\}\}$. Note that $T$ and $T'$ are not collinear: $T\sim_c T_{10}$ and $T'\sim_c T_8$.
However, we claim that  $ \eta^{T}$ and $ \eta^{T'}$ defined by Equation~\eqref{defietaT} are equivalent.  
In fact,  it is easy to check that the map $\theta=\theta_{j*k,i*j*k}$   in Equation~\eqref{titaij}
is an isomorphism of graded algebras 
$\theta\colon\la^{\Phi^{-1}({T})}\to\la^{\Phi^{-1}({T'})}$.  
\end{example}

This is a surprising example, because one could think that equivalent admissible graded contractions have collinear supports, but, at least in this example is not true. 
The following technical result proves in particular that this is the only case in which $ \eta^{T}\sim \eta^{T'}$ but $T\not\sim_c T'$:



\begin{prop} \label{nueva}   
Let $\eta,\eta' \in \mathcal{A}$, $\eta\sim\eta'$.
Denote by $T = S^{\eta}$ and $T' = S^{\eta'}$.
Then the following assertions are true.
\begin{enumerate} 
 \item [\rm (i)] 
Either there are   an isomorphism $\psi\colon  \la^{\ep_{\eta} } \to \la^{\ep_{\eta'}}$ 
and a collineation $\sigma\colon I\to I $
such that $\psi\big(\la_i\big) = \la_{\sigma(i)}$ for all $i \in I$; or $T\subseteq X_{(i)}$ has between 3 and 5 elements.
\item[\rm (ii)] If $T\not\sim_c\{T_8,T_{10}\}$, then there exists $\sigma\in S_\ast(I)$  such that $\eta^\sigma\approx\eta'$.

\item[\rm (iii)] If $T \not\sim_c T'$, then either
  $ T\sim_c T_8$, $T'\sim_c T_{10}$, or viceversa, and there are 
  an isomorphism $\psi\colon  \la^{\ep_{\eta} } \to \la^{\ep_{\eta'}}$, a collineation $\sigma\colon I\to I $
and two distinct indices $r,s\in I$ such that $\psi\theta_{rs}\big(\la_i\big) = \la_{\sigma(i)}$ for all $i \in I$.
\item[\rm (iv)] If $T \sim_c T'$, then there exists $\sigma\in S_\ast(I)$  such that $\eta^\sigma\approx\eta'$.
\end{enumerate}
\end{prop}

\begin{proof}
To ease the notation, write  $\ep=\ep_{\eta}$ and $\ep'=\ep_{\eta'}$.
Take an isomorphism of graded algebras $ \varphi\colon \la^{\ep }\to\la^{\ep'}$ and   a bijection 
$\mu  \colon I\to I$ defined by $\varphi\big(\la_{i}\big) = \la_{{\mu (i)}}$. Note that the determined map $\mu$ might not be a collineation. 
For $i, j \in I$ distinct and $x \in \la_i$ and $y \in \la_j$, we have that
${\ep}_{ij}\varphi([x,y])={\ep'}_{\mu(i),\mu(j)}[\varphi(x),\varphi(y)]$.
Thus $\mu$ satisfies   the following property:
\begin{equation} \label{P}
\{i, j\} \in T\Rightarrow \{\mu(i), \mu(j)\} \in T', \quad \mu(i) \ast \mu(j)=\mu(i \ast j). \tag{P}
\end{equation} 
In particular, $\mu({T}) = {T'}$.
Also, $\mu(I_{{T}}) = I_{T'}$ with the notations  in Proposition \ref{centre}, because 
  $\varphi$ maps the center onto the center and  $\f{z}(\la^{\ep}) = \bigoplus\limits_{i\in I_T}\la_i$. 
  
Clearly, $\mu$ is a collineation provided that $T= X$.
Our first goal will be to prove that, 
except for the case $T\subseteq X_{(i)}$, $3\le|T|\le5$,
 we can   replace $\mu$ with a collineation $\sigma$ and the isomorphism $\varphi$ with another isomorphism $\psi\colon \la^{\ep}\to\la^{\ep'}$   that  satisfies the property
\begin{equation} \label{Q}
\psi\big(\la_{i}\big) = \la_{{\sigma (i)}}, \quad \forall \, \, i \in I.\tag{Q}
\end{equation} 
Ir order to prove it, we distinguish some cases according to the indices involved in the support. 

$\bullet$ 
If ${T} = \emptyset$, we have that $\ep =0$ and $\la^{\ep}$ is abelian. Thus, $\la^{\ep'}$ is also abelian (and $\mu$ can be any bijection). Take $\sigma$ any collineation, and for each $i \in I$ choose any bijective linear map $f_{i, \sigma}\colon \la_i\to\la_{\sigma(i)}$. Then we can define $\psi\colon \la^{\ep}\to\la^{\ep'}$ by $\psi\vert_{\la_i}=f_{i,\sigma}$ for all $i\in I$, which so satisfies (Q).
As both algebras are abelian, $\psi$ is an isomorphism.
 
\smallskip

$\bullet$
If $\{i, j\}\in {T}\subseteq  \big\{\{i, j\}, \{i, i\ast j\}, \{j, i\ast j\}\big\}$, then choose $k\in I\setminus\{i,j,i*j\}$ and $\ell\in I\setminus\{\mu(i), \mu(j),\mu(i)*\mu(j)\}$.
Notice that $\{i, j, k\}$ and $\{\mu(i), \mu(j), \ell\}$ are both generating triplets; hence, there exists a collineation $\sigma: I \to I$ such that $\sigma(i) = \mu(i)$, $\sigma(j) = \mu(j)$ and $\sigma(k) = \ell$. Moreover, $\sigma(i \ast j) = \mu(i \ast j)$, since
$ 
 \sigma(i*j)=\sigma(i)*\sigma(j)=\mu(i)*\mu(j)\stackrel{\eqref{P}}{=}\mu(i*j).
$ 
Now define $\psi\colon \la^{\ep}\to\la^{\ep'}$  by 
$$
\psi\vert_{\la_i\oplus\la_j\oplus\la_{i*j}}=\varphi,\qquad
\psi\vert_{\la_t}=f_{t,\sigma} \textrm{ for any } t\in\{k,k*i,k*j,k*i*j\},
$$ 
(the bijective linear maps $f_{t,\sigma}$ chosen as in the previous item). 
Notice that $\psi$ is an isomorphism, since $\la_t \subseteq \f{z}(\la^{\ep})$, for any $t \in \{k, \, k \ast i, \, k \ast j,\, k \ast i \ast j\}$
and similarly, $\la_r \subseteq \f{z}(\la^{\ep'})$, for any $r \in \sigma \big(\{k, \, k \ast i, \, k \ast j,\, k \ast i \ast j\}\big) = I \setminus
\{\mu(i), \mu(j) ,\mu(i\ast j)\} \subseteq I_{T'}$. 
\smallskip

 $\bullet$
 If ${T}$ does not satisfy any of the conditions covered in the two previous cases, there exists a generating triplet $\{i, j, k\}$ such that $\{i, j\}, \{k, \ell \}\in {T}$ for some $\ell \in I$ with $\ell \ne k$. 
 Consider the collineation $\sigma\in S_\ast(I)$ determined by 
$\sigma(i) = \mu(i)$, $\sigma(j) = \mu(j)$ and $\sigma(k) = \mu(k)$. Notice that $\sigma$ is well defined since $\{\mu(i), \mu(j), \mu(k)\}$ is also a generating triplet (we can apply \eqref{P} to $\mu^{-1}$ since this is related to the isomorphism $\varphi^{-1}$). 
%
For $K:=\{t\in I\mid \sigma(t)=\mu(t)\}$, note that 
\begin{equation} \label{U}
t_1,t_2\in K,\quad
\{\{t_1,t_2\},\{t_1,t_1*t_2\},\{t_2,t_1*t_2\}\}\cap T\ne\emptyset\quad
\Rightarrow\quad t_1*t_2\in K.\tag{U}
\end{equation} 
In fact, if $\{t_1,t_2\}\in T$, then 
 $
 \sigma(t_1*t_2)=\sigma(t_1)*\sigma(t_2)=\mu(t_1)*\mu(t_2)\stackrel{\eqref{P}}{=}\mu(t_1*t_2).
 $
 Also, if $\{t_1,t_1*t_2\}\in T$, then $\sigma(t_2 \ast t_1)=\mu(t_2 \ast t_1)$ since:
$$
\sigma(t_2 \ast t_1)\ast \mu(t_1)=\sigma(t_2 \ast t_1)\ast \sigma(t_1)=\sigma(t_2)=\mu( t_2)\stackrel{\eqref{P}}{=} \mu(t_2 \ast t_1) \ast \mu(t_1).
$$ 
Now, as $i,j,k\in K$, then \eqref{U} implies $i*j\in K$. Also, either $\ell$ or $\ell*k$ belongs to $\{i,j,i*j\}\subseteq K$
 since any two lines in the Fano plane always intersect, so that both $\ell,\ell*k\in K $.
 As both $\sigma$ and $\mu$ are bijections, either $K=I$   or $\{i, j, i\ast j, k, \ell, \ell \ast k \}=K$ (just 5 elements).
In the first case, $\sigma=\mu$ and there is nothing to prove, since $\mu$ would be a collineation and $\varphi$ the required isomorphism. 
Then, assume    $\sigma\ne\mu$.
Labelling the remaining two elements of $I$ as $r$ 
and $s$, we know that $\sigma(r)= \mu(s)$, and $ \sigma(s)= \mu(r)$.  
%

 If $r,s\in I_T$, then  the next map 
 \[
\psi:=
\begin{cases} 
	\psi |_{\la_t}= \varphi, & \text{if  $t\in  K$}, \\ 
	\psi |_{\la_t}= f_{t,\sigma}, &  t=r,s,
\end{cases}
\]
  is an isomorphism satisfying
	\eqref{Q}:
 We only have to note that, if $\{t_1,t_2\}\in {T}$, then $t_1,t_2\in K$ 
 and hence $t_1*t_2\in K$ by \eqref{U}; so that $\psi\vert_{\la_{t_1}\oplus\la_{t_2}\oplus\la_{t_1*t_2}}$ coincides with the restriction of $\varphi$.
 

Now assume $r\notin I_T$ and let us prove that either $T=X^{(i*j)}$ or $T\subseteq X_{(i)}$ (interchanging $i$ and $j$ if necessary).
%
To argue easily, note the following facts.
\begin{enumerate}[label={(\alph*)}]
\item If $\{r,t\}\in {T}$, then either $t= s$ or $t=r*s$. (Similarly, if $\{s,t\}\in {T}$, then either $t= r$ or $t=r*s$.)

\item If $\{t_1,t_2\}\in {T}$, $t_1,t_2\ne r,s$, then $ t_1*t_2\ne r,s$. 

\end{enumerate} 
Indeed, if (a) is not true, $t,r*t\ne r,s$ implies that $t,r*t\in K$ and by \eqref{U}, also $r\in K$, a contradiction. 
And (b) says that $t_1,t_2\in K$ implies $ t_1*t_2\in K$, which is a  trivial fact from \eqref{U}.
%
Now we can discuss the possible supports according to the values of the index $\ell\ne k$. 
\begin{enumerate}  
\item If $\ell=i*j$, then $K=\{i,j,k,i*j,i*j*k\}$ and $\{r,s\}=\{j*k,i*k\}$.
As $T$ is nice, $P_{ijk}\subseteq T$. 
As $\{i,j*k\}\in P_{ijk}\subseteq T$ and $i,i*j*k\in K$, \eqref{U} implies $j*k\in K$, a contradiction.

\item If $\ell=i*k$, then $K=\{i,j,k,i*k,i*j\}$ and $\{r,s\}=\{j*k,i*j*k\}$.
As $T$ is nice, $P_{k,i*k,j}\subseteq T$. 
In particular $\{k,i*j*k\}\in T$, and, as $k,i*j\in K$, \eqref{U} implies $i*j*k\in K$, a contradiction.
(This argument works similarly for $\ell=j*k$.)

\item If $\ell=i*j*k$, then $K=\{i,j,k,i*j*k,i*j\}$ and $\{r,s\}=\{j*k,i*k\}$.
In this case $i*j=k*\ell=r*s$ and we claim that $X^{(i*j)}= T$.

In fact, if $\{r,r*s\}$ or $\{s,r*s\}$ are in $T$, then either $P_{i,j,j*k}\subseteq T$ or $P_{i,j,i*k}\subseteq T$, since $T$ is nice.
This gives $ \{i*j,j*k\}$ or $ \{i*j,i*k\}$ belongs to $T$, respectively, which contradicts (a). As $r\notin I_T$, the only option is 
$\{r,s\}\in T$ so that  $X^{(i*j)}=\{\{ i,j\},\{k.l \},\{ r,s\}\}\subseteq T$. If the containment is strict, then there is 
   $\{t_1,t_2 \}\in T$ (both $t_1,t_2\in K$) which is not in $X^{(i*j)}$ (so $t_1*t_2\ne r*s$).
By (b), $t_1*t_2\ne r,s$. This  we can assume that  $t_1*t_2=i$ ($i$, $j$, $k$ and $l$ play the same role here).
 As $j\notin\{t_1,t_2,t_1*t_2=i\}$ (since $\{i*j,t \}\notin T$ for all $t\in K$), then $t_1$, $t_2$, $j$ are   generative
 and $P_{t_1,t_2,j}\subseteq T$. But $\{i*k,k \}$ belongs to $ P_{t_1,t_2,j}$ and hence to $T$, which again contradicts (a).


\item If $\ell=i$, then $K=\{i,j,k,i*k,i*j\}$ and $\{r,s\}=\{j*k,i*j*k\}$. In this case $r*s=i$ and let us check that  $T\subseteq X_{(i)}$. (Of course
 $T\subseteq X_{(j)}$ if $\ell=j$.)
 
 On one hand, if $\{r,s\}\in T$, then $\{j,s\}\in P_{r,s,j}\subseteq T$, 
 a contradiction with (a). So, if $\{t_1,t_2 \}\in T$ with $t_1=r,s$, then $t_2=i$, and, in any case, $\{r,i \}\in T$. 
 On the other hand assume we have $\{t_1,t_2 \}\in T$ with $t_1,t_2\ne i,r,s$.
 There are only 6 possibilities for $\{t_1,t_2 \}$: if this element is one of 
  $\{j,k \}$, $\{ j,i*k\}$, $\{ k,i*j\}$, or $\{i*k,i*j \}$, then $t_1*t_2$ is either $r$ or $s$, a contradiction with (b).
  If $\{t_1,t_2 \}=\{j,i*j \}$, then $\{i*j,r\}\in P_{j,i*j,r}\subseteq T$, a contradiction. The same argument says that 
  $\{t_1,t_2 \}\ne \{k,i*k \}$. So we have proved $\{ \{i,j\},\{i,k\},\{i,r\} \}\subseteq T\subseteq X_{(i)}$.

\end{enumerate} 
Let us return to our search  for a convenient isomorphism according to the possible supports.

$\star$  $T=X^{(i*j)}=\{\{ r,s \},\{ i,j \},\{ k,l \}  \} $ with   $i*j=k*l=r*s$. 
Now $\theta_{rs}\in\Aut(\la^\ep)$ by
Lemma~\ref{le_automo}. Hence $\varphi\theta_{rs}\colon \la^{\ep }\to\la^{\ep'}$ is isomorphism (composition of isomorphisms). 
Moreover, $\varphi\theta_{rs}$ satisfies \eqref{Q}. In fact, if $t\in K$, $\varphi\theta_{rs}(\la_t)=\varphi(\la_t)=\la_{\mu(t)}=\la_{\sigma(t)}$.
And   $\varphi\theta_{rs}(\la_r)=\varphi(\la_s)=\la_{\mu(s)}=\la_{\sigma(r)}$ (and similarly for $s$).\smallskip

$\star$ 
 $\{ \{i,r\},\{i,s\} \}\subseteq T\subseteq X_{(i)}$. According to
Lemma~\ref{le_automo}, the map  $\theta_{rs} $ is   an automorphism of $ \la^\ep$ if and only if $\ep_{ir}=\ep_{is}$. 
In any case, as both $\ep_{ir},\ep_{is}\ne0$, we can slightly modify the map in Eq.~\eqref{titaij} to define
$\theta'$  by, if $t\in K$,
\begin{equation*} 
\theta'\vert_{\la_{t}}=\id,\quad  
\theta'(x_r)=-\frac{\ep_{ir}}{\ep_{is}}x_s,\quad  
\theta'(y_r)=-\frac{\ep_{ir}}{\ep_{is}}y_s,\quad  
\theta'(x_s)=x_r, \quad  
\theta'(y_s)=y_r. 
\end{equation*}
It is easy to prove that $\theta'\in\Aut(\la^\ep)$ and that $\varphi\theta'\colon \la^{\ep }\to\la^{\ep'}$ is the required isomorphism satisfying \eqref{Q}.
\smallskip

This finishes the proof of (i). Besides, the existence of the map $\psi\colon \la^{\ep}\to\la^{\ep'}$ satisfying \eqref{Q} also says that
 $\eta\approx(\eta')^\sigma$, by composing $\psi$ with the map $\tilde f_{\sigma}$ in Proposition~\ref{colli}:
 $\tilde f_{\sigma}\psi\colon \la^{\ep}\to\la^{\sigma\cdot\ep'}$ is a graded isomorphism.\smallskip

$\star$ 
$\{ \{i,j\},\{i,k\},\{i,r\} \} \subsetneq T\subseteq X_{(i)}$, $\{i,s\} \notin T$.
We can assume without loss of generality that $\{i,i*j\}\in T$. 
Again we modify   the map  $\theta_{j,i*j}$  in  Lemma~\ref{le_automo}, by considering\begin{equation*} 
\hat\theta\vert_{\la_{t}}=\id,\quad  
\hat\theta(x_j)=\frac{-\ep_{ij}}{\ep_{i{i*j}}}x_{i*j},\quad  
\hat\theta(y_j)=\frac{-\ep_{ij}}{\ep_{i{i*j}}}y_{i*j},\quad  
\hat\theta(x_{i*j})=x_j, \quad  
\hat\theta(y_{i*j})=y_j, 
\end{equation*}
if $t\ne i,j$. The same argument as above says that $\hat\theta\in\Aut(\la^\ep)$. 
Now take the collineation  $\nu\in S_\ast(I)$ determined by $\nu(i)=i$, 
$\nu(j)=i*j$ and $\nu(k)=k$. 
%
%
This induces, as in the proof of Proposition~\ref{colli}, the isomorphism $\tilde f_{\nu\sigma^{-1}}\colon \la^{\ep'}\to\la^{\sigma\nu^{-1}\cdot\ep'}$.
The composition of these maps gives the  isomorphism  $\Psi=\tilde f_{\nu\sigma^{-1}}\varphi \hat\theta \colon \la^{\ep}\to\la^{\sigma\nu^{-1}\cdot\ep'}
$. Note that $\Psi$ is a graded isomorphism: writing $t_1\mapsto t_2\mapsto t_3\mapsto t_4$ for shorten $\hat\theta(\la_{t_1}) =\la_{t_2}$, $\varphi(\la_{t_2}) =\la_{t_3}$, $\tilde f_{\nu\sigma^{-1}}(\la_{t_3}) =\la_{t_4}$, we have to check that $t_1=t_4$ for any choice of $t_1\in I$. Indeed,
$t\mapsto t\mapsto\sigma(t)\mapsto\nu(t)=t$ for any $t=i,k,i*k$, and
$$
\begin{array}{ll}
r\mapsto r\mapsto\sigma(s)\mapsto\nu(s)=r,\qquad\qquad 
&i*j\mapsto j\mapsto\sigma(j)\mapsto\nu(j)=i*j, 
 \\
s\mapsto s\mapsto\sigma(r)\mapsto\nu(r)=s,\qquad 
&j\mapsto i*j\mapsto\sigma(i*j)\mapsto\nu(i*j)=j. 
\end{array}
$$
%
That is, $\ep\approx \sigma\nu^{-1}\cdot\ep'$ and, by  Lemma~\ref{unaporsoporte}, $T$ and $T'$ are collinear. This finishes the proof of (ii).
\smallskip

$\star$ 
Finally, consider $T=\{ \{i,j\},\{i,k\},\{i,r\} \}$. 
If  $r=j*k$, then $T\sim_c T_{10}$ and $T'\sim_c T_{8}$,  
otherwise $r=i*j*k$, $T\sim_c T_{8}$ and $T'\sim_c T_{10}$. In both cases $T\not\sim_c T'=\mu(T)$.
Trivially, for the isomorphism $\varphi\colon \la^{\ep}\to\la^{\ep'}$, 
$\varphi\theta_{rs}\big(\la_i\big) = \la_{\sigma(i)}$ holds  for any $i \in I$. (iii) and (iv) follow.  
\end{proof}

At the moment, we know there are 3 infinite families and 21 strong equivalence classes of graded contractions, and that, among those 21,   
there are  only 20 equivalence classes of graded contractions. Let's find out  what happens to the parametrized families.
Again there will be less equivalence classes  of graded contractions than strongly equivalence classes. 

\begin{remark}\label{re_412}
Observe that, for each conflictive nice set $T\subseteq X_{(1)}$,
 there are also admissible maps with support $T$  which are equivalent but not strongly equivalent:
 
 $\bullet$ 
If $\eta$ has support $T_{14}$ and we take $\sigma\in S_\ast(I)$ with $\tilde\sigma(T_{14})=T_{14}$, then, with the notation as in 
Theorem~\ref{prop_equiv}, $\eta^\sigma=(\eta^\sigma_{12},\eta^\sigma_{13},\eta^\sigma_{15},\eta^\sigma_{16})$. It is clear that $\sigma(1)=1$ and so $\eta^\sigma=(\eta_{1\sigma(2)},\eta_{1\sigma(3)},\eta_{1\sigma(5)},\eta_{1\sigma(6)})$. For instance take $\sigma_0$  the collineation determined by 
$\sigma_0(1)=1$, $\sigma_0(2)=3$ and $\sigma_0(3)=2$.  
Then we have 
$(1, 1, 1, \lambda)^{\sigma_0}=(1, 1,  \lambda,1)\sim_n (1, 1, 1, \frac1\lambda)$.
In this last step we have applied   Theorem~\ref{prop_equiv} (i). So, for $\lambda\ne1$ , $(1, 1, 1, \lambda)\sim (1, 1, 1, \frac1\lambda)$, though we know that they are not strongly equivalent.\smallskip

$\bullet$ 
If $\eta$ has support $T_{17}$, again $\eta^\sigma=(\eta_{1\sigma(2)},\eta_{1\sigma(3)},\eta_{1\sigma(4)},\eta_{1\sigma(5)},\eta_{1\sigma(6)})$ for each collineation $\sigma$ such that $\tilde\sigma(T_{17})=T_{17}$. Note that we have used that $\sigma(1)=1$ for any such collineation. For 
$\sigma_0$ as in the above item, $(1, \lambda,1, 1, \lambda)^{\sigma_0}=(\lambda,1, 1, \lambda,1)\sim_n  (1, \frac{1}{\sqrt{\lambda^2}},1, 1, \frac{1}{\sqrt{\lambda^2}})\sim_n(1, \frac1\lambda,1, 1, \frac1\lambda)$. Here we have used  the proof of Theorem~\ref{prop_equiv} (ii). If $\lambda\ne\pm1$, then $(1, \frac1\lambda,1, 1, \frac1\lambda)$ and  $(1, \lambda,1, 1, \lambda)$ are then equivalent but not strongly equivalent.\smallskip

$\bullet$ 
If $\eta$ has support $T_{20}$, and $\sigma_1$ is  the collineation determined by 
$\sigma_1(1)=1$, $\sigma_1(2)=2$ and $\sigma_1(3)=4$, then $(1, \lambda , \mu, 1, \lambda, \mu)^{\sigma_1}=(1,\mu,  \lambda , 1, \mu, \lambda)$.
Also, $(1, \lambda , \mu, 1, \lambda, \mu)^{\sigma_0}=( \lambda ,1, \mu,  \lambda, 1,\mu)\sim_n (1, \frac1\lambda , \frac\mu\lambda, 1, \frac1\lambda , \frac\mu\lambda)$, by   Theorem~\ref{prop_equiv} (iii) and taking into account possible changes of signs. By composing these collineations,
$(1, \lambda , \mu, 1, \lambda, \mu)^{\sigma_0\sigma_1}=(\mu, 1, \lambda, \mu,1, \lambda )\sim_n(1,\frac1\mu,\frac\lambda\mu,1,\frac1\mu,\frac\lambda\mu)$. But recall that $(1, \lambda , \mu, 1, \lambda, \mu)\sim_n(1, \lambda' , \mu', 1, \lambda', \mu')$ if and only if $\lambda = \pm\lambda'$, $\mu = \pm\mu'$. Thus, there are less equivalence classes than strong equivalence classes.
\end{remark}

Next we will prove that we have essentially found all the possible $\eta,\eta'$ which are equivalent but not strongly equivalent.

\begin{theorem}\label{th_lasclasesdeverdad}
Representatives of all the classes up to equivalence of the graded contractions of $\Gamma_{\f{g}_2}$ are:
\begin{enumerate} 
\item [\rm (i)]  $\{\eta^{T_i}\mid i\ne 8,14,17,20\}$;
\item [\rm (ii)] $\{(1, 1, 1, \lambda)\mid  \lambda\in\bb C^\times\}$ related to $T_{14}$,  where $(1, 1, 1, \lambda)\sim (1, 1, 1, \lambda')$ if and only if $ \lambda'\in\{\lambda,\frac1{\lambda}\}$;
\item [\rm (iii)]  $ \{(1, \lambda, 1, 1, \lambda)\mid  \lambda\in\bb C^\times\}$ related to $T_{17}$,
where $(1, \lambda,1, 1, \lambda)\sim (1, \lambda',1, 1, \lambda')$ if and only if $\lambda'\in\{\pm\lambda,\pm\frac1{\lambda}\}$;

\item [\rm (iv)] $\{(1, \lambda , \mu, 1, \lambda, \mu)\mid \lambda,\mu\in\bb C^\times\}$ related to $T_{20}$,
where two   maps $(1, \lambda , \mu, 1, \lambda, \mu)\sim (1, \lambda' , \mu', 1, \lambda', \mu')$ if and only if the set
$\{\pm\lambda',\pm\mu'\}$ coincides with either $\{\pm\lambda,\pm\mu\}$ or $\{\pm\frac1\lambda,\pm\frac\mu\lambda\}$ or $\{\pm\frac\lambda\mu,\pm\frac1\mu\}$.
\end{enumerate} 
\end{theorem}

\begin{proof}
We know that $\eta^{T_8}\sim \eta^{T_{10}}$ as in  Example~\ref{horror}, but  the maps in $\{\eta^{T_i}\mid 1\le i\le 24,i\ne10\}$ are all of them not equivalent
 by Proposition~\ref{nueva}  (ii). (If $\eta,\eta'\in\mathcal{A}$ were equivalent and $S^\eta\not\sim_c T_{8},T_{10}$, then their supports would be collinear.)
Of course the number of equivalence classes with support $T_i$ is at most the number of strong equivalence classes with support $T_i$, and we know that this number is $1$ if $i\ne14,17,20$, so that we have to think only of the supports collinear to 
$T_{14}$, $T_{17}$ and $T_{20}$. Besides,
by Proposition~\ref{colli}, we only have to find the classes of maps in $\mathcal{A}$ with supports equal to $T_{14}$, $T_{17}$ and $T_{20}$.

If $\eta=(1, 1, 1, \lambda)\sim \eta'=(1, 1, 1, \lambda')$ for $S^{\eta}=S^{\eta'}= T_{14}$, let us prove that $\lambda'=\lambda^{\pm1} $. Proposition~\ref{nueva}  (ii) implies that there is a collineation $\sigma$ such that $ \eta^\sigma\approx\eta' $. 
By Theorem~\ref{teo_normalizacionbastaba},  $(1, 1, 1, \lambda)^\sigma=\eta^\sigma\sim_n\eta'=(1, 1, 1, \lambda')$. 
As $\eta^\sigma_{ij}=\eta_{\sigma(i)\sigma(j)}$, observe that $(1, 1, 1, \lambda)^\sigma$ has to be 
one of the next four possibilities according to the value of $\sigma^{-1}(6)$ (respectively $6,5,3,2$):
$(1, 1, 1, \lambda)$, $(1, 1, \lambda,1)$, $( 1, \lambda,1,1)$ or $(\lambda,1, 1, 1)$. The first and third admissible maps are $\sim_n (1, 1, 1, \lambda)$, so that $\lambda=\lambda'$ by 
Theorem~\ref{prop_equiv}.
 In a similar way, the second and fourth cases are $\sim_n (1, 1, 1, \frac1\lambda)$, so that $\frac1\lambda=\lambda'$.
 Conversely, if  $ \lambda'\in\{\lambda,\frac1{\lambda}\}$, then $(1, 1, 1, \lambda)\sim (1, 1, 1, \lambda')$ 
by Remark~\ref{re_412}.

Second,  assume that both $\eta=(1, \lambda,1, 1, \lambda)\sim \eta'=(1, \lambda',1, 1, \lambda')$ have support equal to $T_{17}$. Again
 there is a collineation $\sigma$ such that $\tilde\sigma(T_{17})=T_{17}$ and $(1, \lambda,1, 1, \lambda)^\sigma\approx (1, \lambda',1, 1, \lambda')$ by Proposition~\ref{nueva}  (ii). Observe that $\sigma(1)=1$, $\sigma(4)=4$, and $\sigma(2)$ can be any value in $\{2,3,5,6\}$.
 This   leads to two possibilities: $\eta^\sigma\in\{ (1, \lambda,1, 1, \lambda),(\lambda,1, 1, \lambda,1) \}$. In the first case $\lambda=\pm\lambda'$, and, in the second one,
   $\eta^\sigma=(\lambda,1, 1, \lambda,1) \sim_n  (1, \frac1\lambda,1, 1, \frac1\lambda)  $. As this  is equivalent by normalization to $ (1, \lambda',1, 1, \lambda')$, 
  hence $\frac1\lambda\in\{\pm\lambda'\}$. The converse is an immediate consequence of Remark~\ref{re_412}.
 
Third, if  $(1, \lambda , \mu, 1, \lambda, \mu)$ and $(1, \lambda' , \mu', 1, \lambda', \mu')$ are two equivalent admissible maps with supports equal to $T_{20}$, then there exists a collineation $\sigma$ such that $\tilde\sigma(T_{20})=T_{20}$ and 
$(1, \lambda , \mu, 1, \lambda, \mu)^\sigma\approx  (1, \lambda' , \mu', 1, \lambda', \mu')$, applying by Proposition~\ref{nueva}  as in the above cases.
Note that $\sigma(1)=1$, and $\sigma$ sends lines to lines, so  $(1, \lambda , \mu, 1, \lambda, \mu)^\sigma$ is necessarily one of the next maps
with support $T_{20}$:
$$
\begin{array}{lll}
(1, \lambda , \mu, 1, \lambda, \mu),
&
( \lambda , 1,\mu,  \lambda,1, \mu)\sim_n (1, \frac1\lambda , \frac\mu\lambda, 1, \frac1\lambda , \frac\mu\lambda),
&
( \lambda , \mu, 1, \lambda, \mu,1)\sim_n (1, \frac\mu\lambda, \frac1\lambda , 1,   \frac\mu\lambda,\frac1\lambda ),
\\
(1, \mu,  \lambda, 1,\mu, \lambda ), 
&
( \mu,1,  \lambda, \mu, 1,\lambda )\sim_n(1,\frac1\mu,\frac\lambda\mu,1,\frac1\mu,\frac\lambda\mu),
&
( \mu,  \lambda, 1,\mu, \lambda,1 )\sim_n(1,\frac\lambda\mu,\frac1\mu,1, \frac\lambda\mu,\frac1\mu).
\end{array}
$$
By looking at the second element of the 6-tuples, Theorem~\ref{prop_equiv} says   $ \pm\lambda'\in\{\lambda,\mu, \frac1\lambda,\frac1\mu,\frac\mu\lambda,\frac\lambda\mu\} $, 
and looking at the the third one, we know that, respectively, $ \pm\mu'\in\{\mu,\lambda,\frac\mu\lambda,\frac\lambda\mu, \frac1\lambda,\frac1\mu\} $. 
This leads to the three possibilities for the set $\{\pm\lambda',\pm\mu'\}$ described in (iv). The converse is consequence of a suitable choice of collineations, as in Remark~\ref{re_412}.
\end{proof}

 \begin{remark}
 Observe that Theorem~\ref{th_lasclasesdeverdad} tells that, if $\eta,\eta'\in\mathcal A$ are such that $\eta\sim\eta'$ and $S^\eta=S^{\eta'}$, then  
 $\eta\approx\eta'$ except at most if $S^\eta\sim_c T_{14},T_{17},T_{20}$.
 \end{remark}

	
\section{Properties of the Lie algebras obtained as graded contractions of $\f{g}_2$} \label{sec6}
	
Finally, we study  the properties of the Lie algebras obtained in Theorem~\ref{th_lasclasesdeverdad}.
We begin by revisiting some notions from Lie theory. In what follows, $\f{g}$ denotes a finite-dimensional complex Lie algebra, and $\f{g}'$ its derived algebra $[\f{g}, \f{g}]$.  

The \emph{lower central series} of $\f{g}$ is the sequence of subalgebras $\f{g}^{0} = \f{g}$, $\f{g}^{1}= \f{g}'$ and $\f{g}^{n} = [\f{g}, \f{g}^{n-1}]$, for all $n \geq 2$; 
while the \emph{derived series} of $\f{g}$ is the sequence of subalgebras $\f{g}^{(1)} = \f{g}'$ and $\f{g}^{(n)} = [\f{g}^{(n-1)}, \f{g}^{(n-1)}]$, for all $n \geq 2$. We say that ${\mathfrak {g}}$ is \emph{nilpotent} (respectively, \emph{solvable}) if its lower central series (respectively, its derived series) terminates in the zero subalgebra; in other words, there exists $n$ such that $\f{g}^{n} = 0$ (respectively, $\f{g}^{(n)} = 0$). If $m$ is the least natural number such that $\f{g}^{m} = 0$ (respectively, $\f{g}^{(m)} = 0$), then $\f{g}$ is called $m$-step nilpotent (respectively, $m$-step solvable) and $m$ is called the nilindex of $\f{g}$. As usual, $\f{r}(\f{g})$ denotes the \emph{radical} of $\f{g}$ (its largest solvable ideal), and $\f{z}(\f{g}) = \{x \in \f{g} \colon [x,\f{g}] = 0\}$ refers to its \emph{center}; clearly, $\f{z}(\f{g}) \subseteq \f{r}(\f{g})$. We say that $\f{g}$ is \emph{semisimple} if $\f{r}(\f{g}) = 0$, or equivalently, if $\f{g}$ is the  direct sum of simple ideals; $\f{g}$ is \emph{reductive} if $\f{r}(\f{g}) = \f{z}(\f{g})$, or equivalently, if $\f{g}$ can be decomposed as the direct sum of a semisimple Lie algebra and an abelian Lie algebra. 
 Furthermore,  any $\f{g}$ can be   written as a direct sum of  $\f{r}(\f{g})$  and a semisimple subalgebra  called a \emph{Levi subalgebra}. This decomposition is called the \emph{Levi decomposition} of $\f{g}$.\smallskip

In what follows, $\ep\colon G\times G\to \mathbb C$ denotes  an admissible graded contraction and $T$ denotes the support $S^{\eta^\ep}$ of the associated admissible map $\eta^\ep\in\mathcal{A}$. Our goal here is to investigate which properties the 14-dimensional Lie algebra $\la^\ep = (\f{g}_2)^\ep$ satisfies. 
  Many properties are determined by the support. For instance, we saw in Proposition~\ref{centre}  that this is the situation of the center:
$\f{z}(\la^\ep) = \bigoplus\limits_{i\in I_T}\la_i$, for $I_T = \setbar{i\in I}{i \notin t, \, \, \forall \, t \in T}$.

Thus,
for each of the supports of admissible graded contractions (given up to collineation in Theorem~\ref{classiN}), we describe the properties of the related Lie algebra or family of Lie algebras (when appropriated). The proof becomes a straightforward calculation (and we leave it to the reader) by using Proposition~\ref{centre} and Lemma~\ref{lem_base}, keeping in mind how the Lie bracket $[\cdot, \cdot]^{\ep^T}$ works (see Definition~\ref{defepbracket}).

\begin{theorem}\label{teo_lasalgebras}   
Let $\ep\colon G\times G\to \mathbb C$ be an admissible graded contraction such that $T=S^{\eta^\ep}=T_i$ for some $1\le i\le24$. 

\smallskip

{\rm (1)} If $T =T_1= \emptyset$, then $\lae$ is abelian. 

\smallskip

{\rm (2)} If $T =T_{2}=\{\{1, 2\}\}$, then  
\begin{itemize}
\item[-] $\f{z}(\f{\lae}) = \la_3 \oplus \la_4\oplus\la_5\oplus\la_6\oplus\la_7$ and $\dim\f{z}(\la^\ep) = 10$,
\item[-] $(\f{\lae})'= \la_5$ and $\dim(\lae)' = 2$,
\item[-] $(\lae)^{(2)} = (\lae)^{2} = 0$, that is, $\lae$ is 2-step nilpotent.
\end{itemize}

\smallskip

{\rm (3)} If $T =T_{3}= \{\{1, 2\}, \{1, 3\}\}$, then 
\begin{itemize}
\item[-] $\f{z}(\f{\lae}) = \la_4 \oplus \la_5\oplus \la_6 \oplus \la_7$ and $\dim\f{z}(\la^\ep) = 8$,
\item[-] $(\f{\lae})'= \la_5 \oplus \la_6$ and $\dim(\lae)' = 4$,
\item[-] $(\lae)^{(2)} = (\lae)^{2} = 0$, that is, $\lae$ is 2-step nilpotent.
\end{itemize}

\smallskip

{\rm (4)} If $T = T_{4}=\{\{1, 2\}, \{1, 5\}\}$, then  
\begin{itemize} 
\item[-] $\f{z}(\la^\ep) = \la_3 \oplus \la_4 \oplus \la_6 \oplus \la_7$ and  $\dim\f{z}(\la^\ep) = 8$,
\item[-] $(\lae)'= (\lae)^{n} = \la_2 \oplus \la_5$ for all $n\ge 1$, $\dim(\lae)' = 4$ and $\lae$ is not nilpotent,
\item[-] $(\lae)^{(2)} =0$, that is, $\lae$ is 2-step solvable.
\end{itemize}

\smallskip

{\rm (5)} If $T = T_{5}=\{\{1, 2\}, \{6, 7\}\}$, then 
\begin{itemize}
\item[-] $\f{z}(\la^\ep) = \la_3 \oplus \la_4 \oplus \la_5$ and $\dim\f{z}(\la^\ep) = 6$,
\item[-] $(\lae)' = \la_5$ and $\dim(\lae)'= 2$,
\item[-] $(\lae)^{(2)} = (\lae)^{2} = 0$, that is, $\lae$ is 2-step nilpotent.
\end{itemize}

\smallskip

{\rm (6)} If $T =T_{6}= X_{L_{12}}=  \{\{1, 2\}, \{1, 5\}, \{2, 5\}\}$, then 
\begin{itemize}
\item[-] $\f{z}(\la^\ep) = \f{r}(\la^\ep) = \la_3\oplus\la_4\oplus \la_6\oplus\la_7$, $\dim\f{z}(\la^\ep) = 8$ and $\lae$ is reductive,
\item[-] $(\lae)' = (\lae)^{(n)} = (\lae)^{n} = \la_1 \oplus \la_2 \oplus \la_5$  for all $n\ge1$, and $\lae$ is neither nilpotent nor solvable,
\item[-] $(\lae)' \cong \f{sl}(2,\bb C) \oplus \f{sl}(2,\bb C)$ is the Levi subalgebra of $\lae$,  
\item[-]  $\lae = \f{z}(\lae) \oplus (\lae)'$ is the Levi decomposition of $\lae$.  
\end{itemize}

\smallskip

{\rm (7)} If $T = T_{7}=X^{(1)}= \{\{2, 5\}, \{3, 6\}, \{4, 7\}\}$, then  
\begin{itemize}
\item[-] $\f{z}(\la^\ep)= \la_1 = (\lae)'$ and $\dim\f{z}(\la^\ep) = 2$, 
\item[-] $(\lae)^{(2)}=(\lae)^{2} = 0$, that is, $\lae$ is 2-step nilpotent.
\end{itemize}

\smallskip

{\rm (8)} If $T =T_{8}= \{\{1, 2\}, \{1, 3\}, \{1, 4\}\}$, then 
\begin{itemize}
\item[-] $\f{z}(\la^\ep) = (\lae)' = \la_5\oplus \la_6\oplus\la_7$ and $\dim\f{z}(\la^\ep) = 6$, 
\item[-] $(\lae)^{(2)}=(\lae)^{2} = 0$, that is, $\lae$ is 2-step nilpotent.
\end{itemize}

\smallskip

{\rm (9)} If $T =T_{9}= \{\{1, 2\}, \{1, 3\}, \{1, 5\}\}$, then   
\begin{itemize}
\item[-] $\f{z}(\la^\ep) = \la_4\oplus \la_6\oplus\la_7$ and $\dim\f{z}(\la^\ep) = 6$,
\item[-] $(\lae)' = \la_2 \oplus \la_5 \oplus \la_6$ and $\dim(\lae)' = 6$,
\item[-] $(\f{\lae})^{(2)} = 0$, that is, $\lae$ is 2-step solvable,
\item[-] $(\lae)^{n} = \la_2 \oplus \la_5$ for all $n\ge2$, and $\lae$ is not nilpotent. 
\end{itemize}

%

\smallskip

{\rm (11)} If $T =T_{11}= \{\{1, 2\}, \{1, 6\}, \{2, 6\}\}$, then   
\begin{itemize}
\item[-] $\f{z}(\la^\ep) = \la_3 \oplus \la_4 \oplus \la_5 \oplus \la_7$ and $\dim\f{z}(\la^\ep) = 8$,
\item[-] $(\lae)' = \la_3 \oplus \la_4 \oplus \la_5$ and $\dim(\lae)' = 6$,
\item[-] $(\lae)^{(2)} = (\lae)^{2} = 0$, that is, $\lae$ is 2-step nilpotent.
\end{itemize}

\smallskip

{\rm (12)} If $T = T_{12}=\{\{1, 2\}, \{1, 6\}, \{6, 7\}\}$, then    
\begin{itemize}
\item[-] $\f{z}(\f{\lae}) = \la_3\oplus\la_4\oplus \la_5$ and $\dim\f{z}(\la^\ep) = 6$,
\item[-] $(\lae)' = \la_3 \oplus \la_5$ and $\dim(\lae)' = 4$, 
\item[-] $(\lae)^{(2)}=(\lae)^{2}=0$, that is, $\lae$ is 2-step nilpotent.
\end{itemize}

\smallskip

{\rm (13)} If $T =T_{13}= \{\{1, 2\}, \{1, 3\}, \{1, 4\}, \{1, 5\}\}$, then   
\begin{itemize}
\item[-] $\f{z}(\la^\ep) = \la_6 \oplus \la_7$ and $\dim\f{z}(\la^\ep) = 4$,
\item[-] $(\lae)' = \la_2 \oplus \la_5 \oplus \la_6 \oplus \la_7$ and $\dim(\lae)'=8$,
\item[-] $(\lae)^{(2)} =0$, that is, $\lae$ is 2-step solvable,
\item[-] $(\lae)^{n} = \la_2\oplus\la_5$ for all $n\ge2$, and $\lae$ is not nilpotent.
\end{itemize}

\smallskip

{\rm (14)} If $T =T_{14}= \{\{1, 2\}, \{1, 3\}, \{1, 5\}, \{1, 6\}\}$, then   
\begin{itemize}
\item[-] $\f{z}(\la^\ep) = \la_4 \oplus \la_7$ and $\dim\f{z}(\la^\ep) = 4$.
\item[-] $(\lae)' = \la_2\oplus\la_3\oplus \la_5\oplus\la_6 = (\lae)^{n}$, for all $n \geq 1$, $\dim(\lae)' = 8$ and $\lae$ is not nilpotent,
\item[-] $(\lae)^{(2)} = 0$, that is, $\lae$ is 2-step solvable, 
\end{itemize}

\smallskip

{\rm (15)} If $T =T_{15}= \{\{1, 2\}, \{1, 6\}, \{1, 7\}, \{2, 6\} \}$, then    
\begin{itemize}
\item[-] $ \f{z}(\la^\ep) = (\lae)' = \la_3 \oplus \la_4\oplus \la_5 $ and $\dim (\lae)' = \dim\f{z}(\la^\ep) = 6$,
\item[-] $(\lae)^{(2)} = (\lae)^{2} = 0$, that is, $\lae$ is 2-step nilpotent.
\end{itemize}

\smallskip
 
{\rm (16)} If $T =T_{16}= \{\{1, 2\}, \{1, 6\}, \{2, 7\}, \{6, 7\}\}$, then    
\begin{itemize}
\item[-] $\f{z}(\la^\ep) = \la_3 \oplus \la_4\oplus \la_5$ and $\dim\f{z}(\la^\ep) = 6$,
\item[-] $(\lae)' = \la_3 \oplus \la_5$ and $\dim(\lae)' = 4$,
\item[-] $(\lae)^{(2)} = (\lae)^{2} = 0$, that is, $\lae$ is 2-step nilpotent.
\end{itemize}

\smallskip
 
{\rm (17)} If $T =T_{17}= \{\{1, 2\}, \{1, 3\}, \{1, 4\}, \{1, 5\}, \{1, 6\}\}$, then  
\begin{itemize}
\item[-] $\f{z}(\la^\ep) = \la_7$ and $\dim\f{z}(\la^\ep)=2$, 
\item[-] $(\lae)'= \la_2 \oplus \la_3 \oplus \la_5 \oplus \la_6 \oplus \la_7$ and $\dim(\lae)' = 10$,
\item[-] $(\lae)^{(2)} =0$, that is, 2-step solvable,
\item[-] $(\lae)^{n} = \la_2\oplus\la_3\oplus \la_5\oplus\la_6$ for all $n\ge2$, and $\lae$ is not nilpotent.
\end{itemize}

\smallskip

{\rm (18)} If $T =T_{18}= \{\{1, 2\}, \{1, 6\}, \{1, 7\}, \{2, 6\}, \{2, 7\}\}$, then     
\begin{itemize}
\item[-] $\f{z}(\la^\ep) = (\la^\ep)' = \la_3 \oplus \la_4 \oplus \la_5$ and $\dim (\la^\ep)' = \dim\f{z}(\la^\ep) = 6$, 
\item[-] $(\lae)^{(2)}=(\lae)^{2}=0$, that is, $\lae$ is 2-step nilpotent. 
\end{itemize}

\smallskip

{\rm (19)} If $T =T_{19}= X_{L^C_{12}}=\{\{3, 4\}, \{3, 6\}, \{3, 7\}, \{4, 6\}, \{4, 7\}, \{6, 7\}\}$, then   
\begin{itemize}
\item[-] $\f{z}(\la^\ep) = (\la^\ep)' = \la_1 \oplus \la_2 \oplus \la_5$ and $\dim (\la^\ep)' = \dim\f{z}(\la^\ep) = 6$, 
\item[-] $(\lae)^{(2)} = (\lae)^{2} = 0$, that is, $\lae$ is 2-step nilpotent. 
\end{itemize}

\smallskip

{\rm (20)} If $T =T_{20}= X_{(1)}=\{\{1, 2\}, \{1, 3\}, \{1, 4\}, \{1, 5\}, \{1, 6\},\{1, 7\}\}$, then    
\begin{itemize}
\item[-] $\f{z}(\la^\ep) = 0$,
\item[-] $(\lae)' = \la_2\oplus\la_3\oplus\la_4\oplus  \la_5\oplus\la_6\oplus \la_7$ and $\dim(\lae)' = 12$,
\item[-] $(\lae)^{(2)} = 0$, that is, $\lae$ is 2-step solvable,
\item[-] $(\lae)^{n}=\la'$ for all $n \ge 2$ and $\lae$ is not nilpotent.
\end{itemize}

\smallskip

{\rm (21)} If $T =T_{21}= P_{\{1, 2, 3\}}=\{\{1, 2\}, \{1, 3\}, \{1, 7\}, \{2,3\}, \{2, 6\},\{3,5\}\}$, then   
\begin{itemize}
\item[-] $\f{z}(\la^\ep) = \la_4$ and $\dim\f{z}(\la^\ep) = 2$,
\item[-] $(\lae)' = \la_4 \oplus \la_5 \oplus \la_6 \oplus \la_7$ and $\dim(\lae)' = 8$, 
\item[-] $(\lae)^{(2)} = 0$, that is,  $\lae$ is 2-step solvable, 
\item[-] $(\lae)^2 = \f{z}(\la^\ep)$ and $\lae$ is 3-step nilpotent.
\end{itemize}

\smallskip

{\rm (22)} If $T =T_{22}= T_{\{1,2,3\}}$, then   
\begin{itemize}  
\item[-] $\f{z}(\la^\ep) = 0$,
\item[-] $(\lae)^{n} = \la_2 \oplus \la_3 \oplus \la_4 \oplus \la_5 \oplus \la_6 \oplus \la_7$ 
for all $n \geq 1$, $\dim(\lae)' = 12$ and  $\lae$ is not nilpotent,
\item[-] $(\lae)^{(2)} = \la_4 \oplus \la_7$, $(\lae)^{(3)} = 0$ and $\lae$ is 3-step solvable.
\end{itemize}

\smallskip

{\rm (23)} If $T =T_{23}= X\setminus X_{L^C_{12}}$, then 
\begin{itemize}
\item[-]  $\f{z}(\la^\ep) = 0$,
\item[-]  $(\lae)^{n} = (\lae)^{(n)} = \lae$ for all $n\ge1$, and $\lae$ is neither nilpotent nor solvable,
\item[-]  $\f{r}(\lae) = \la_3 \oplus \la_4 \oplus \la_6 \oplus \la_7$ is an abelian ideal, $\dim\f{r}(\lae) = 8$ and $\lae$ is not reductive,
\item[-]  $\f{h} = \la_1 \oplus \la_2 \oplus \la_5 \cong \f{sl}(2,\bb C) \oplus \f{sl}(2,\bb C)$ is the Levi subalgebra of $\lae$,
\item[-]  $\lae = \f{r}(\lae) \oplus \f{h}$ is the Levi decomposition of $\lae$.  
\end{itemize}

\smallskip

{\rm (24)} If $T =T_{24}= X$, then $\lae$ is simple.
\end{theorem}

Of course the algebras related to $T_8$ and $T_{10}$ share the main properties, since they are isomorphic.

We close the section by summarizing the information obtained in the previous theorem. 
We order the algebras so that they appear progressively  ``less and less abelian''.
The pair $d_{(i)}: = \big(\dim \f{z}(\la^\ep), \dim(\lae)'\big)$, for $\ep$ an admissible graded contraction with support $T_i$, gives a key invariant. 

\begin{cor}\label{cor_recopilacion}
The Lie algebras obtained by graded contractions of $\Gamma_{\f{g}_2}$ are, 
\begin{itemize}
\item[-] 1 abelian: $d_{(1)}=(14,0)$;
\item[-] 12 nilpotent not abelian:  
\begin{itemize}
\item   1 with nilindex 3: $d_{(21)}=(2,8)$,
\item  11 with nilindex 2:  $d_{(2)}=(10,2)$, $d_{(3)}=(8,4)$, $d_{(5)}=(6,2)$,  $d_{(7)}=(2,2)$, 
$d_{{(8)}}=d_{(15)}=d_{(18)}=d_{(19)}=(6,6)$, $d_{(11)}=(8,6)$ and $d_{(12)}=d_{(16)}=(6,4)$;  
\end{itemize}
\item[-]   solvable not nilpotent:    
\begin{itemize}
 \item 1 with solvability index 3: $d_{(22)}=(0,12)$,
  \item 3 isolated cases with solvability index 2: $d_{(4)}=(8,4)$, $d_{(9)}=(6,6)$  and $d_{(13)} =(4,8)$,   
  \item 3 infinite families depending on parameters with solvability index 2: $ d_{(14)}=(4,8)$, $d_{(17)}=(2,10)$ and $d_{(20)}=(0,12)$;
\end{itemize}
\item[-] 1  which is sum of a semisimple Lie algebra and a non-trivial center:  $d_{(6)}=(8,6)$;
 \item[-] 1  not reductive: $d_{(23)} =(0,14)$;  
\item[-] 1 simple: $ d_{(24)}=(0,14)$.
\end{itemize}
\end{cor}

Thus the  invariant  $d_{(i)}$ jointly with the nilpotency and solvability indices allow to distinguish 
the equivalence class of two graded contractions of $\Gamma_{\f{g}_2}$,
except in the case both related Lie algebras are 2-step nilpotent (with $\dim\f{z}(\la^\ep)  =6$), and except in the case where both algebras are 2-step solvable not nilpotent   (with $\dim\f{z}(\la^\ep)  \le4$). 

\section{Conclusions and further work}

In this paper we have tackled  the problem of the classification of the graded contractions of $\Gamma_{\f{g}_2}$ up to equivalence $\sim$, 
that is, where the isomorphism between the related Lie algebras   permutes the homogeneous components of the grading. 
Some other equivalence relations have been considered to help to our study, namely, $\sim_n$ and $\approx$.
  If $\ep$ and $\ep'$ are  graded contractions of a grading $\Gamma$ on a Lie algebra $\la$, then
$\ep\sim_n\ep'\Rightarrow \ep\approx\ep' \Rightarrow \ep\sim\ep'\Rightarrow\la^{\ep}$ and $\la^{\ep'}$ are isomorphic, although
 none of the converses is true in general. 

Our approach  to the case of $\f{g}_2$ and its $\bb{Z}_2^3$-grading can be summarized as follows.
First,   each equivalence class contains an admissible representative.  
The supports of the admissible graded contractions  are nice sets, 
and each nice set is related to at least one graded contraction.
The Weyl group  of $\Gamma_{\f{g}_2}$ allows us to obtain    equivalent   graded contractions from   collinear nice sets (see Proposition~\ref{colli}).
There are 24  nice sets up to collineation according to the purely combinatorial classification   achieved in Theorem~\ref{classiN}. 
For 21 of these nice sets, all the admissible graded contractions with just that support are equivalent by normalization
(Theorem~\ref{prop_equiv}), in particular equivalent.
This is not the case for the remaining 3 nice sets, which give families of  graded contractions parametrized by $\bb{C}^\times$, $\bb{C}^\times/\bb{Z}_2$ and   $(\bb{C}^\times)^2/\bb{Z}^2_2$, not equivalent by normalization. 
The problem of whether these graded contractions could be strong equivalent, even though they are not equivalent by normalization, is a difficult task.  Theorem~\ref{teo_normalizacionbastaba}   gives a negative answer in our case: strong equivalence and equivalence by normalization coincides for $\Gamma_{\f{g}_2}$. 
Putting together the above results,   we have found a   representative of each equivalence class of graded contractions of $\Gamma_{\f{g}_2}$. 
More precisely:
if $\ep$ is such graded contraction, there is $\eta\in\mathcal A$ with $\ep\sim\ep^\eta$, 
there is  a collineation $\sigma$ such that $\tilde\sigma^{-1}(S^{\eta})=T$
 is one of the 24 nice sets shown in Theorem~\ref{classiN}, and hence
$\ep\sim\ep^\eta\sim\ep^{\eta^\sigma}\approx\ep^{\eta^T}$ except for $T$   one the three nice sets, subsets of $X_{(i)}$, mentioned in Theorem~\ref{prop_equiv}. In this situation, $\ep\sim
\ep^{\eta'}$ for $\eta'$ of the form $(1,1,1,\lambda)$, $(1,\lambda,1,1,\lambda)$ or $(1,\lambda,\mu,1,\lambda,\mu)$, with the notations used there.
But not all these classes are not equivalent, Theorem~\ref{th_lasclasesdeverdad} deals with this issue  in order to distinguish when two   equivalent graded contractions with the same support are strongly equivalent. A point to be careful with    is that $\eta \sim\eta'$ does not imply that there exists a collineation $\sigma $  such that $\eta^\sigma\approx\eta'$, according to Proposition~\ref{nueva}.
Thus Theorem~\ref{th_lasclasesdeverdad}  gives  the collection of representatives of the possible graded  Lie algebras, up to isomorphism, obtained from $\Gamma_{\f{g}_2}$.
Finally, we study some properties of these Lie algebras  in Theorem~\ref{teo_lasalgebras}.
It remains to be studied whether two of these algebras   are non-isomorphic in the standard sense. Some invariants have been considered, but this   issue needs further research in order to conclusively demonstrate when two graded contractions with nilindex 2  are indeed
non-isomorphic as a Lie algebras, and similarly for  two algebras within
the parametric continua.
We leave this study pending for the near future. 
One idea for tackling this task is to first  give     precise models of all the algebras studied in Theorem~\ref{teo_lasalgebras}, taking into account the  recent work \cite{twisted}. 
There, a $\bb Z_2^3$-graded Lie algebra over the reals isomorphic to the compact Lie algebra $\f{g}_2^c$ is explicitly
constructed without   using octonions or derivations, which makes such algebra extremely easy to use. 

Additionally, we would also like to study the real case. Many of our results are still valid for the real field and the compact Lie algebra of derivations of the octonion division algebra $\f{g}_2^c=\Der(\bb O)$, 
but Theorem~\ref{prop_equiv} is no longer true (using  the field $\mathbb C$ was  relevant there). The real Lie algebras obtained  
by using the admissible maps in Theorem~\ref{th_lasclasesdeverdad} are not isomorphic and
satisfy the   properties described in  Theorem~\ref{teo_lasalgebras}, although these real algebras do not cover all the Lie algebras that could be obtained by graded contractions of $\f{g}_2^c$. 

Another question to study is how a $\bb Z_2^3$-grading on an algebra has to be  in order to apply the results obtained in this paper to it. 
At first glance it might seem that none, because throughout this work we have used many properties that are specific to $\f{g}_2$.
We have begun to investigate this line of work  and  we can announce that  there are some suitable graded Lie algebras. 
This will allow to take advantage of the (very technical)  classification of the nice sets to obtain more new Lie algebras. 
\smallskip


 \textbf{Acknowledgment:}  
The authors are indebted to the anonymous referee  for his/her careful reading and valuable suggestions which have improved significantly the manuscript.


\begin{thebibliography}{999}
	
   \bibitem{gradsg22} 
 Bahturin, Y.A. and Tvalavadze, M.V.;
 \emph{Group gradings on $G_2$.} 
 Comm. Algebra \textbf{37} (2009), no. 3, 885--893. 
 
 \bibitem{review}
 Boza, L.; Fedriani, E.M.; N\'u\~nez, J. and Tenorio,  A.F.;
 \emph{A historical review of the classifications of Lie algebras.} 
 Rev. Un. Mat. Argentina \textbf{54} (2013), no. 2, 75--99.

 \bibitem{Burde07}
   Burde, D.; 
   \emph{Contractions of Lie algebras and algebraic groups.} 
   Arch. Math. (Brno) \textbf{43} (2007), no. 5, 321--332.
   
    \bibitem{CPSW}
Couture, M.; 
Patera, J.; 
Sharp, R.T. and  
Winternitz, P.;
\emph{Graded contractions of $\mathfrak{sl}(3,\bb C)$.}
J. Math. Phys. \textbf{32} (1991), no. 9, 2310--2318.	 
 


 
  \bibitem{NotesG2}
 Draper Fontanals, C.; 
 \emph{Notes on $G_2$: the Lie algebra and the Lie group.}
 Differential Geom. Appl. \textbf{57} (2018), 23--74.
 
 
 
 \bibitem{twisted}
 Draper Fontanals, C.; 
 \emph{The compact  exceptional Lie algebra $\mathfrak g^c_2$ as a twisted ring group.}
 To appear in Proceedings of the American Mathematical Society.
  
 
 	
	 \bibitem{GradsG2}
 Draper, C. and  Mart\'{\i}n, C.;
\emph{Gradings on $\f{g}_2$}.
 Linear Algebra Appl. \textbf{418} (2006), no.~1, 85--111. 

\bibitem{gradsO} 
 Elduque, A.;
 \emph{Gradings on octonions.} 
 J. Algebra \textbf{207} (1998), no. 1, 342--354.
 

 
  \bibitem{EK13}
  Elduque, A. and Kochetov, M.; 
  Gradings on simple Lie algebras. 
  Mathematical Surveys and Monographs, 189. American Mathematical Society, Providence, RI; 
  Atlantic Association for Research in the Mathematical Sciences (AARMS), 
  Halifax, NS, 2013. xiv+336 pp. ISBN: 978-0-8218-9846-8
  
   \bibitem{Fialo} 
   Fialowski, A. and de Montigny, M.;
   \emph{Deformations and contractions of Lie algebras.}
   J. Phys. A \textbf{38} (2005), no. 28, 6335--6349. 
 
  
  	\bibitem{Graaf}
 de Graaf, W.A.;
 \emph{Classification of solvable Lie algebras.} 
 Experiment. Math. \textbf{14} (2005), no. 1, 15--25. 
 
 \bibitem{checos}
Havl\'{\i}\v{c}ek, M.;
 Patera, J.;
Pelantov\'a, E. and
 Tolar, J.;
\emph{On Pauli graded contractions of $\mathfrak{sl}(3,\bb C)$.}  
J. Nonlinear Math. Phys. \textbf{11} (2004), suppl., 37--42.

\bibitem{Hirschfeld79}
Hirschfeld, J.W.P.;
Projective Geometries Over Finite Fields. 
Oxford Math. Monogr.
The Clarendon Press, Oxford University Press, New York, 1979. xii+474 pp.
ISBN:0-19-853526-0


\bibitem{gr-cont} 
Hrivn\'ak, J. and Novotn\'y, P.;
\emph{Graded contractions of the Gell-Mann graded $\mathfrak{sl}(3,\bb C)$. } 
J. Math. Phys. \textbf{54} (2013), no. 8, 081702, 25 pp.
  

\bibitem{checos06}
Hrivn\'ak, J.; Novotn\'y, P.; Patera, J. and Tolar, J.;  
\emph{Graded contractions of the Pauli graded $\mathfrak{sl}(3,\bb C)$.}
Linear Algebra Appl. \textbf{418} (2006), no. 2-3, 498--550. 



  \bibitem{IW53}
   In\"on\"u, E. and   Wigner, E.P.;
 \emph{ On the contraction of groups and their representations.}
Proc. Natl. Acad. Sci. U.S.A. \textbf{39}  (1953), 510--524.
     
  
\bibitem{ref15}
 Kashuba, I. and   Patera, J.;
\emph{Graded contractions of Jordan algebras and of their representations.}
J. Phys. A \textbf{36} (2003), 12453--12473.

\bibitem{ref17}
Kostyakov, I.V.; Gromov, N.A.  and  Kuratov, V.V.;
\emph{Modules of graded contracted Virasoro algebras.}
Nucl. Phys. B, Proc. Suppl. 102-103, 316--321 (2001).

  
   \bibitem{Lu}
 Lu, C.; 
 \emph{Classification of finite-dimensional solvable Lie algebras with nondegenerate invariant bilinear forms.} 
 J. Algebra \textbf{311} (2007), no. 1, 178--201.


\bibitem{ref19}
 de Montigny, M.;
\emph{Graded contractions of affine Lie algebras.}
J. Phys. A \textbf{29}, 4019--4034 (1996).

 \bibitem{otro91}
   de Montigny, M. and Patera, J.;
  \emph{Discrete and continuous graded contractions of Lie algebras and superalgebras.}
J. Phys. A \textbf{24} (1991), no. 3, 525--547.
 

 \bibitem{enci}
  Onishchik, A.L. and Vinberg, E.B.;
   Lie Groups and Lie Algebras,  1991. 
   Encyclopaedia of Mathematical Sciences, vol 41.

\bibitem{Okubo}
Okubo, S.;
Introduction to octonion and other non-associative algebras in physics. 
Montroll Memorial Lecture Ser. Math. Phys., 2
Cambridge University Press, Cambridge, 1995. xii+136 pp.
ISBN:0-521-47215-6

\bibitem{Rand}
 Rand, D.; Winternitz, P. and Zassenhaus, H.; 
 \emph{On the identification of a Lie algebra given by its structure constants. I. Direct decompositions, Levi decompositions, and nilradicals.} 
 Linear Algebra Appl. \textbf{109} (1988), 197--246. 

\bibitem{Schafer}
   Schafer, R.D.;
   An introduction to nonassociative algebras. 
   Pure and Applied Mathematics, Vol. 22.
    Academic Press, New York-London 1966 x+166 pp. 
   
   \bibitem{Segal}  
  Segal, I.E.;  
  \emph{A class of operator algebras which are determined by groups.}
  Duke Math. J. \textbf{18} (1951), 221--265.
    
  
  \bibitem{citadoreciente}
  \v{S}nobl, L. and  Kar\'asek, D.; 
  \emph{Classification of solvable Lie algebras with a given nilradical by means of solvable extensions of its subalgebras.} 
  Linear Algebra Appl. \textbf{432} (2010), no. 7, 1836--1850. 

\bibitem{book}
 \v{S}nobl, L. and  Winternitz, P.; 
 \emph{Classification and identification of Lie algebras.} 
 CRM Monograph Series, 33. American Mathematical Society, Providence, RI, 2014. xii+306 pp. ISBN: 978-0-8218-4355-0. 


\bibitem{ref33}
 Weimar-Woods, E.;
\emph{The general structure of $G$-graded contractions of Lie algebras, I: The classification.}
Can. J. Math. \textbf{58}(6), 1291--1340 (2006).

\bibitem{ref34}
 Weimar-Woods, E.;
\emph{The general structure of $G$-graded contractions of Lie algebras, II: The contracted Lie algebra.}
Rev. Math. Phys. \textbf{18}(6), 655--711 (2006).

\end{thebibliography}
\end{document}